\documentclass[	fontsize=12pt, 
	BCOR=12mm, 
	DIV=calc, 
	headinclude,
	bibliography=oldstyle,
	bibliography=totoc]{scrbook}

\usepackage{amsthm,amsfonts,amsmath, amssymb}
\usepackage[mathscr]{eucal}

\KOMAoptions{DIV=calc}

\usepackage{array}

\usepackage{pbsi}

\usepackage{fourier} 
\usepackage[scaled=0.875]{helvet} 

\usepackage[pdftex,bookmarks=true]{hyperref}
\usepackage{graphicx}
\usepackage{enumerate}
\usepackage{scrpage2}
\setheadsepline{.1pt}
\theoremstyle{plain}
\newtheorem{theorem}{Theorem}[chapter]
\newtheorem{cor}[theorem]{Corollary}
\newtheorem{result}[theorem]{Result}
\newtheorem{lemma}[theorem]{Lemma}

\newtheorem*{lemma*}{Technical Lemma}

\theoremstyle{definition}
\newtheorem{defn}[theorem]{Definition}
\newtheorem{example}[theorem]{Example}
\newtheorem{fact}[theorem]{Fact}
\theoremstyle{remark}
\newtheorem{remark}[theorem]{Remark}

\theoremstyle{plain}

  {\end{enumerate}}

\def\suchthat{\; : \;}

\def\C{\mathbb{C}}
\def\D{\mathbb{D}}

\def\X{{\mathcal X}}

\def\sgn{{\mb{sgn}}}

\def\l{\left}
\def\r{\right}
\def\<{\langle}
\def\>{\rangle}

\def\mb{\mbox}
\newcommand{\one}{{\mathbf 1}}
\newcommand{\E}{\mbox{\bf E}}

\def\bar{\overline}
\def\P{{\bf P}}

\def\d{\partial}

\newcommand{\sm}{{\raise0.3ex\hbox{$\scriptstyle \setminus$}}}

\def\mb{\mbox}
\def\l{\left}
\def\r{\right}

\def\CHI{\mathchoice%
{\raise2pt\hbox{$\chi$}}%
{\raise2pt\hbox{$\chi$}}%
{\raise1.3pt\hbox{$\scriptstyle\chi$}}%
{\raise0.8pt\hbox{$\scriptscriptstyle\chi$}}}
\def\smalloplus{\raise1pt\hbox{$\,\scriptstyle \oplus\;$}}

\usepackage{xcolor}	 
\definecolor{mycolour}{RGB}{150,0,0}

\hypersetup{
  colorlinks,
  citecolor=mycolour,
  linkcolor=mycolour,
  urlcolor=mycolour}

\begin{document}

\begin{titlepage}
	\begin{center}
		\vspace{.5in}
		{\LARGE  \textbf{Hole probabilities for determinantal point processes in the complex plane} } \\
		\vspace{0.8in}
		{\large  A Dissertation \\ 
			submitted in partial fulfilment of the requirements \\
		\vspace{0.04in}
		for the award of the degree of} \\
		\vspace{0.3in}
		{{\LARGE \bsifamily 
		Doctor of Philosophy}}\\  
		\vspace{.2in}
		{\large by}\\
		\vspace{0.05in}
		{\large  Kartick Adhikari}\\
		\vspace{1.1in}
		\includegraphics[scale=.5]{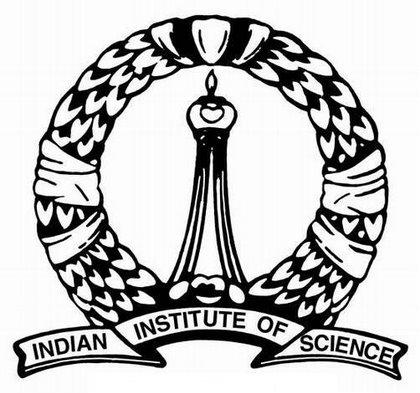} \\
		\vspace{.1in}
		{\large Department of Mathematics \\
		Indian Institute of Science \\
			Bangalore - 560012 \\
		July 2016 \\}
	\end{center}
\end{titlepage}

\frontmatter

\chapter{}

\begin{center}
To
\\{\it My Parents}
\\{\it and }
\\ {\it Teachers}
\end{center}

\chapter{Declaration}
\vspace{0.5in}
\noindent I hereby declare that the work reported in this thesis is entirely original and has
been carried out by me under the supervision of Prof.~Manjunath Krishnapur at the Department of
Mathematics, Indian Institute of Science, Bangalore. I further declare
that this work has not been the basis for the award of any degree, diploma, fellowship,
associateship or similar title of any University or Institution.\\
\vspace*{1in}

\noindent $\begin{array}{lcr}
\textrm{Kartick Adhikari} & \hspace*{1.95in} &~ \\
\textrm{S. R. No. 9310-310-091-06665} & \hspace*{1.95in}   &~ \\
\textrm{Indian Institute of Science} & \hspace*{1.95in}   & ~ \\
\textrm{Bangalore} & \hspace*{1.95in}  & ~ \\
\textrm{July} ,~ 2016 & \hspace*{1.95in}   & ~ \\
~ & \hspace*{1.95in}   & ~ \\
~ & \hspace*{1.95in}   & ~ \\
~ & \hspace*{1.95in}   & ~ \\
~& \hspace*{1.95in}   & \textrm{Prof.~Manjunath Krishnapur} \\
~& \hspace*{1.95in}  & \textrm{(Research advisor)}
\end{array}$

\chapter{Acknowledgement}

I take this opportunity to thank
my  PhD supervisor, Prof. Manjunath Krishnapur for his continuous support,
enthusiasm and motivations all throughout my stay
at IISc. I am thankful to him for introducing me
to the area of random analytic functions and random matrix theory  and for agreeing
to supervise my thesis. My discussions with him
throughout has been very positive and he has
always encouraged me to try out new ideas. 
He not only provided me with excellent academic
suggestions but also helped me out
with other non-academic issues as well. Many thanks
to him for providing me with opportunities to visit
some of the best conferences/workshops in my area.
In general, thesis writing may be quite demanding.
 Thanks to him for going through
the thesis draft meticulously and helping me to
improve the organization of the thesis.

 Many people were indirectly
involved with my PhD thesis. Specifically,  I thank Prof. Srikanth K. Iyer and Prof. Mrinal K. Ghosh for introducing me to the beautiful world of probability and stochastic
processes.  I am also thankful to the faculty members of the Dept. of Mathematics at IISc for
offering a wide variety of courses. I really benefited a lot from these courses. I thank Prof. Arup Bose for indicating the problem that we are considering in Chapter \ref{ch:beta}.

I had many fruitful discussions with colleagues at IISc. My special thanks to Nanda, Tulasi, Koushik da, Subhamay Da, Tamal Da, Indrajit, Ramiz, Vikram, Arpan, Samya, Chandan, Choiti, Eliza. 
I am also thankful to the Maths Dept. staff at IISc, Ms. Mahalakshmi, Ms. Bharathi, Mr. Narendranath, Parmesh and Parthivan. They helped me in carrying out smoothly, all the complex administrative formalities during my stay at IISc. Many thanks to our system
administrator, Ujey, Surendar and librarian, Shanmugavel
for their quick and prompt actions whenever
I needed help.\\

 I thank to my colleagues at IISc for all
the fun and support, that made my
campus life a second home. In this regard, I
thank  Jaikrishanan, Divakaran, Rajeev, Pranav, Bidyut Da,
 Dinesh Da, Abhijit Da, Sayan Da, Santanu Da, Prahllad Da, Sudipto Da, Sayani Di, Ratna Di, Soma Di,   Atreyee Di and Joyti Di for their love and
support in my long journey. Among the juniors I thank Bidhan, Hari, Amar, Soumitra, Samarpita, Anway, Samrat, Sanjay, Monojit, Somnath Hazra, Somnath Pradhan, Subhajit, Manish and Rakesh. I thank them for all the great moments that we had here together. 

My special thanks to  Debanjan Da, Prasenjit Da, Satadal Da, Nirupam Da, Abirami, Soumyadeep Da, Anirban Da  and Mathew for their enjoyable company.

I take this opportunity to express my gratitude to all my teachers. In particular, I thank
to Mrinal Babu, Indra Babu, Tapas babu, Mahadev Babu for introducing me in mathematics and helping  me a lot in difficult times.

I would thank my parents and family members.
I am fortunate to have such great family members.
Their blessings and prayer have helped me to
overcome the difficult days of my graduate life. 
I thank them for all their love, patience and for
always being there when I needed them. My Special thanks to my brothers for helping me to concentrate fully on my
study without getting bogged down a midst
other non-academic responsibilities.

The work was done in Department of Mathematics, Indian Institute of Science, and I thank to my Institute and CSIR for providing me the scholarship and other facilities.

\chapter{Abstract}
We study the hole probabilities for ${\mathcal X}_{\infty}^{(\alpha)}$ ($\alpha>0$), a determinantal point process in the complex plane with the kernel  $\mathbb K_{\infty}^{(\alpha)}(z,w)=\frac{\alpha}{2\pi}E_{\frac{2}{\alpha},\frac{2}{\alpha}}(z\bar w)e^{-\frac{|z|^{\alpha}}{2}-\frac{|w|^{\alpha}}{2}}$ with respect to Lebesgue measure on the complex plane, where $E_{a,b}(z)$ denotes the Mittag-Leffler function. Let $U$ be an open subset of $D(0,(\frac{2}{\alpha})^{\frac{1}{\alpha}})$ and ${\mathcal X}_{\infty}^{(\alpha)}(rU)$ denote the number of points  of ${\mathcal X}_{\infty}^{(\alpha)}$ that fall in $rU$. Then, under some  conditions on $U$, we  show that 
$$
\lim_{r\to \infty}\frac{1}{r^{2\alpha}}\log\mathbb P[\mathcal X_{\infty}^{(\alpha)}(rU)=0]=R_{\emptyset}^{(\alpha)}-R_{U}^{(\alpha)},
$$
where $\emptyset$ is the empty set and 
$$
R_U^{(\alpha)}:=\inf_{\mu\in \mathcal P(U^c)}\left\{\iint \log{\frac{1}{|z-w|}}d\mu(z)d\mu(w)+\int
|z|^{\alpha}d\mu(z) \right\},
$$
$\mathcal P(U^c)$ is the space of all compactly supported  probability measures with support in $U^c$. Using potential theory, we give an explicit formula for $R_U^{(\alpha)}$, the minimum possible energy of a probability measure compactly supported on $U^c$ under logarithmic potential with an external field $\frac{|z|^{\alpha}}{2}$. In particular, $\alpha=2$ gives the hole probabilities for the infinite ginibre ensemble. Moreover, we calculate $R_U^{(2)}$ explicitly for some special sets like annulus, cardioid, ellipse, equilateral triangle and half disk.

\tableofcontents

\chapter{Notation}

$
\begin{array}{ll}
\C  & \mbox{The set of complex numbers}
\\ \D & \mbox{Unit disk in the complex plane with center at origin}
\\ m & \mbox{Lebesgue measure on the complex plane}
\\ U & \mbox{Open set in the complex plane }
\\ rU & \{r.z\suchthat z\in U\}
\\ U^c & \C \backslash U , \mbox{complement of the set $U$}
\\ \mathcal P(E) & \mbox{The set of compactly supported probability measures with support in $E$}
\\ p_{\mu}(z) & \int \log\frac{1}{|z-w|}d\mu(w), \mbox{ the potential of $\mu$ at $z$}
\\ R_{\mu}^{(g)} & \iint \log\frac{1}{|z-w|}d\mu(z)d\mu(w)+\int g(|z|)d\mu(z), 
\\ R_U^{(g)} & \inf\{R_{\mu}^{(g)}\suchthat \mu \in \mathcal{P}(\C\backslash U)\}, 
\\ R_{\mu}^{(\alpha)} & \iint \log\frac{1}{|z-w|}d\mu(z)d\mu(w)+\int |z|^{\alpha}d\mu(z), 
\\ R_U^{(\alpha)} & \inf\{R_{\mu}^{(\alpha)}\suchthat \mu \in \mathcal{P}(\C\backslash U)\}, 
\\ \X & \mbox{A point process on the complex plane}
\\ \X(\Omega) & \mbox{The number of points of $\X$ that fall in $\Omega \subset \C$}
\\ \mathbb V(\X(\Omega)) & \mbox{The variance of $\X(\Omega)$}
\\ T & \mbox{the solution of $tg'(t)=2$}
\\ D(0,T) & \mbox{Disk of radius $T$ center at origin}
\end{array}
$

\mainmatter
\chapter{Introduction}
Determinantal point process was introduced by Macchi \cite{macchi} to describe the statistical distribution of a fermion system in thermal equilibrium. This process is also known as fermion random point process. The probability distribution of the determinantal  point process is characterized by a determinant of a matrix built from a kernel. Moreover, the probability density vanishes whenever two points are equal which implies that the points tend to repel each other. The determinantal point processes arise in  quantum physics, combinatorics and random matrix theory. Recently \cite{kulesza}, it has found applications in machine learning also.

Let  $\mathcal X$ be a point process (see \cite{verejones}, p. 7)  in the complex plane and let $U$ be an open set in $\C$. The probability that $U$ contains no points of $\mathcal X$ is called {\it hole}/{\it gap  probability}  for $U$. Hole probabilities for various point processes have been studied extensively, e.g., see \cite{sodin}, \cite{nishry10}, \cite{nishry11}, \cite{nishry12}, \cite{peled} and \cite{akemannhole}. We calculate the asymptotics of hole probabilities for certain determinantal point processes.

\section{Basic notions and definitions}
In this section we define point processes, correlation functions and determinantal point processes in the complex plane. We give examples of determinantal point processes that we are considering in this thesis.
\subsection{Point processes, joint intensities}
A {\bf point process} $\X$ in the complex plane is a random integer valued positive Radon measure on the complex plane. (Recall that a Radon measure is a Borel measure which is finite on compact sets). The point process $\X$ is said to be {\bf simple} if $\X$ almost surely assigns at most measure $1$ to singletons. Roughly speaking, the random set of points configuration in the complex plane is called a point process in the complex plane. The number of points of $\X$ that fall in $\Omega\subset \C$ is denoted by $\X(\Omega)$.

Let $\mu$ be a Radon measure on $\C$. If there exist functions  $\rho_k:\C^k\to[0,\infty)$ for $k\ge 1$, such that for any family of mutually disjoint subjects $D_1,\ldots,D_k$ of $\C$,
$$
\mathbb E\l[\prod_{i=1}^k\X(D_i)\r]=\int_{\prod_i D_i}\rho_k(x_1,\ldots,x_k)d\mu(x_1)\ldots d\mu(x_k),
$$
then  $\rho_k,\ k\ge 1$ are called {\bf correlation functions} or {\bf joint intensities} of a point process $\X$ with respect to $\mu$. In addition, we shall require that $\rho_k (x_1,\ldots,x_k)$ vanish if $x_i=x_j$ for some $i\neq j$. 

In most cases, the point process is described by its correlation functions. If the distribution of $\X(D_1),\X(D_2),\ldots,\X(D_k)$ is determined by its moments, then the correlation functions determine the law of $\X$ (for details, see \cite{manjubook}, Remark 1.2.4). The determinantal point process is determined by its correlation functions.

\subsection{Determinantal point processes}
Let $\mathbb K(x,y):\C^2 \to \mathbb C$ be a measurable function. A point process $\X $ in the complex plane is said to be a {\bf determinantal point process} with kernel $\mathbb K$ if it is simple and its joint intensities with respect to the measure $\mu$ satisfy
$$
\rho_k(x_1,\ldots,x_k)=\det(\mathbb K(x_i,x_j))_{1\le i,j\le k},
$$
for every $k\ge 1$ and $x_1,\ldots,x_k\in \C$. For a detailed  discussion on determinantal point processes we refer the reader to see \cite{soshnikovsurvey}, \cite{andersonbook},  \cite{manjubook} and \cite{russell}.

We skip the issue of existence and uniqueness of the determininatal point processes. We refer the reader to see Section 4.5 and Lemma 4.2.6 in \cite{manjubook} for existence and uniqueness respectively. The following fact says that if $\mathbb K$ is a finite dimensional projection kernel, then a determinantal process does exist.
\begin{fact}\label{ft:existence}
Let $\mu$ be a Radon measure on $\C$. Suppose $\{\varphi_k\}_{k=1}^n$ is an orthonormal set in $L^2(\C,\mu)$. Then there exists a determinantal point process, having $n$ points, with the kernel 
$$
\mathbb K_n(z,w)=\sum_{k=1}^n\varphi_k(z)\overline {\varphi_k(w)},
$$
with respect to background measure $\mu$.
\end{fact}
For the proof of Fact \ref{ft:existence} see \cite{manjubook}, pp. 65-66. This fact is also true for $n=\infty$, see Section 4.5 of \cite{manjubook}. Using Fact \ref{ft:existence} we give few examples determinantal point processes. 

\subsection{Examples of determinantal point processes}\label{sec:exdet}
 In this section we give few examples of determinantal point processes that we are considering in this thesis.
\begin{enumerate}
\item {\bf $n$-th Ginibre ensemble:} Let $G_{n}$ be a $n\times n$ matrix with i.i.d. complex Gaussian entries. Then the eigenvalues of $G_n$ form a determinantal point process in the complex plane with kernel
$$
\mathbb K_n^{(2)}(z,w)=\frac{1}{\pi}\sum_{k=0}^{n-1}\frac{(z\bar w)^k}{k!}e^{-\frac{1}{2}|z|^2-\frac{1}{2}|w|^2},
$$
with respect to Lebesgue measure on the complex plane. Equivalently, the vector of eigenvalues (with uniform order) has density
\begin{align*}
\frac{1}{\pi^n\prod_{k=1}^nk!}\prod_{i<j}|z_i-z_j|^2e^{-\sum_{k=1}^{n}|z_k|^2},
\end{align*} 
with respect to Lebesgue measure on $\C^n$. This is a determinantal point process from the random matrix theory, introduced by Ginibre \cite{ginibre}. For more determinantal point processes from the random matrix theory, irrelevant for us, we refer the reader to see \cite{dyson}, \cite{manjuthesis}, \cite{burda}, \cite{adhikari} and references therein.

\item{\bf Infinite Ginibre ensemble :} The determinantal point process in the complex plane with the kernel $\mathbb K_{\infty}^{(2)}(z,w)=\frac{1}{\pi} e^{z\bar w-\frac{1}{2}|z|^2-\frac{1}{2}|w|^2}$  with respect to Lebesgue measure on the complex plane, equivalently, with respect to the kernel $e^{z\bar w}$ with respect to background measure $\frac{1}{\pi}e^{-|z|^{2}}dm(z)$ on $\C$.  The kernel $e^{z\bar w}$ is a projection kernel from $L^2(\C, \frac{1}{\pi}e^{-|z|^{2}}dm(z))$ to the space of all entire functions in $L^2(\C, \frac{1}{\pi}e^{-|z|^{2}}dm(z))$. Since the kernel $\mathbb K_n^{(2)}(z,w)$ converge to the kernel $\mathbb K_{\infty}^{(2)}(z,w)$, the $n$-th Ginibre ensemble converges to the infinite Ginibre ensemble in distribution as $n\to \infty$.   

\item{\bf $\X_n^{(\alpha)}$:} For fixed $\alpha>0$, $\mathcal X_n^{(\alpha)}$ is a determinantal point process in the complex plane with the kernel 
 $$
 \mathbb K_n^{(\alpha)}(z,w)=\frac{\alpha}{2\pi}\sum_{k=0}^{n-1}\frac{(z\bar w)^k}{\Gamma(\frac{2}{\alpha}(k+1))}e^{-\frac{|z|^{\alpha}}{2}-\frac{|w|^{\alpha}}{2}},
 $$ 
 with respect to Lebesgue measure on the complex plane. Equivalently, the vector of points of $\mathcal X_n^{(\alpha)}$ (in uniform random order) has density
 \begin{align*}
 \frac{\alpha^n}{n!( 2\pi)^n\prod_{k=0}^{n-1}\Gamma(\frac{2}{\alpha}(k+1))}e^{-\sum_{k=1}^{n}|z_k|^{\alpha}}\prod_{i<j}|z_i-z_j|^2,
 \end{align*}
with respect to Lebesgue measure on $\C^n$. Note that, $\alpha=2$ gives the $n$-th Ginibre ensemble.

\item{\bf $\X_{\infty}^{(\alpha)}$ :} For fixed $\alpha>0$, $\X_{\infty}^{(\alpha)}$ is the determinatal point process in the complex plane with the kernel  $\mathbb K_{\infty}^{(\alpha)}(z,w)=\frac{\alpha}{2\pi}E_{\frac{2}{\alpha},\frac{2}{\alpha}}(z\bar w)e^{-\frac{|z|^{\alpha}}{2}-\frac{|w|^{\alpha}}{2}}$ with respect to Lebesgue measure on the complex plane, where $E_{a,b}(z)$ denotes the Mittag-Laffler function (see \cite{mittag}), an entire function when $a>0$ and $b>0$, defined by
$$
E_{a,b}(z)=\sum_{k=0}^{\infty}\frac{z^k}{\Gamma(ak+b)}.
$$
Note that, for $\alpha=2,$ $\mathcal X_{\infty}^{(2)}$ is the infinite Ginibre ensemble. The kernels $\mathbb K_n^{(\alpha)}(z,w)$ converge to the kernel $\mathbb K_{\infty}^{(\alpha)}(z,w)$ as $n\to \infty$. Therefore the point processes $\X_n^{(\alpha)}$ converge to the point process $\X_{\infty}^{(\alpha)}$ in distribution as $n\to \infty$. 

\item{\bf $\X_{n}^{(g)}$ :} Let the function $g:[0,\infty)\to [0,\infty)$ satisfies the following conditions:
\begin{enumerate}
\item $g(r)$ is increasing function in $r$ such that $re^{-\frac{g(r)}{2}}\to 0$ as $r\to \infty$.

\item $g$ is a twice differentiable function on $(0,\infty)$.

\item $\lim_{r\to 0^+}rg'(r)=0$ and $rg'(r)$ is increasing on $(0,\infty)$.

\item $c_k=\int_{0}^{\infty}r^{2k+1}e^{-g(r)}dr <\infty$ for all $k=0,1,\ldots$.
\end{enumerate} 
Observe that, $g(r)=r^{\alpha}$ satisfies the above conditions for $\alpha>0$. Then $\X_n^{(g)}$ is a determinantal point process with kernel 
$$
\mathbb K_{n}^{(g)}(z,w)=\sum_{k=0}^{n-1}\frac{(z\bar w)^k}{c_k^{(n)}}e^{-\frac{ng(|z|)}{2}-\frac{ng(|w|)}{2}},
$$
with respect to Lebesgue measure on the complex plane, where the constants $c_k^{(n)}=\int |z|^{2k}e^{-ng(|z|)}dm(z)$ for $k=0,1,\ldots, n-1$. Equivalently, the vector of points of $\X_n^{(g)}$ (in uniform random order) has the density
\begin{align*}
\frac{1}{Z_n^{(g)}}\prod_{i<j}|z_i-z_j|^2e^{-n\sum_{k=1}^{n}g(|z_k|)},
\end{align*}
with respect to the Lebesgue measure on $\C^n$, where $Z_n$ is the normalizing constant, i.e.,
\begin{align*}
Z_n^{(g)}=\int_{\C}\cdots \int_{\C}\prod_{i<j}|z_i-z_j|^2e^{-n\sum_{k=1}^{n}g(|z_k|)}\prod_{k=1}^ndm(z_k)=n!\prod_{k=0}^{n-1}c_k^{(n)}.
\end{align*}

\noindent Note that if $g=r^{\alpha}$ for $\alpha>0$ and if we scaled $\X_n^{(g)}$ by $n^{\frac{1}{\alpha}}$, then we get $\X_n^{(\alpha)}$, i.e., $$\X_n^{(\alpha)}=n^{\frac{1}{\alpha}}.\X_n^{(g)}:=\{n^{\frac{1}{\alpha}}.z\suchthat z\in z\in \X_n^{(g)}\}.$$

\end{enumerate}
In Chapter \ref{ch:finite} and Chapter \ref{ch:infinite}, we calculate the hole probabilities for these determinantal point processes.

\section{A brief survey of hole probabilities}
Here we shall mention few results of hole probabilities for few point processes.
\begin{itemize}
\item Let $\X_f$ be a point process of zeros set of the Gaussian analytic functions 
$$
f(z)=\sum_{k=0}^{\infty}a_k\frac{z^k}{\sqrt{k!}},
$$ 
where $a_k$ are i.i.d. standard complex normal random variables (i.e., each has the density $\frac{1}{\pi}e^{-|z|^2}$ with respect to Lebesgue measure on $\C$). Sodin and Tsirelson \cite{sodin} showed that for $r\ge 1$, 
$$
e^{-Cr^4}\le \P[\X_f(r\D)=0]\le e^{-cr^4},
$$
for some positive constants  $c$ and $C$.

\item  Later Alon Nishry \cite{nishry10} calculated the sharp constant in the exponent of the hole probability for $\X_f$ as $r\to \infty$. In fact he showed that
$$
\lim_{r\to \infty}\frac{1}{r^4}\log \P[\X_f(r\D)=0]=-\frac{3e^2}{4}.
$$
In the same paper, he calculated the asymptotics of hole probabilities for zeros of a wide class of random entire functions. For more results in  this direction, e.g., see \cite{nishry11}, \cite{nishry12}.

\item Let  $\X_{f_L}$ be the zeors set of hyperbolic Gaussian analytic function $f_L$ on unit disk, where 
$$
f_L(z)=\sum_{k=0}^{\infty}\sqrt{\frac{L(L+1)\cdots (L+k-1)}{k!}} a_kz^k,\;\; 0<L<\infty,
$$
and $a_k$ are i.i.d. standard complex normal random variables. Recently, Buckley, Nishry, Peled and Sodin \cite{peled} proved the following results, asymptotics of hole probabilities for $\X_{f_L}$, as $r\to 1^-$
$$
-\log\P[\X_{f_L}(r\D)=0]=\l\{\begin{array}{lcr}
\Theta\l(\frac{1}{(1-r)^L}\log\frac{1}{1-r}\r) & \mbox{if} &\mbox{$0<L<1$}
\\ \Theta\l(\frac{1}{1-r} \r) &\mbox{if} & \mbox{ $L=1$}
\\ \Theta\l(\frac{1}{1-r}\log^2\frac{1}{1-r}\r) & \mbox{if}& \mbox{$0<L<1.$}
\end{array}\r.
$$
Note that there is a transition in the asymptotics of the hole probability according to whether $0 < L < 1$ or $L = 1$ or $L > 1$.


\item Akemann and Strahov \cite{akemannhole} calculated the asymptotics for the hole probabilities for the eigenvalues of the product of finite matrices with i.i.d. standard complex normal entries.  

\end{itemize}

\section{Description of problems and proof techniques}
Consider the infinite Ginibre  ensemble $\mathcal X_{\infty}^{(2)}$ in the complex plane. It is known (see \cite{manjubook}, Proposition 7.2.1) that
\begin{align}\label{eqn:1/4}
\lim_{r\to \infty}\frac{1}{r^4}\log\P[\mathcal X_{\infty}^{(2)}(r\D)=0]=-\frac{1}{4},
\end{align}
where $\D$ is open unit disk. The immediate consequence of \eqref{eqn:1/4},   
for a general open set $U\subset \D$, there exist  constants $C_1$ and $C_2$ (depending on $U$)  such that
$$
C_1\le \liminf_{r\to \infty}\frac{1}{r^4}\log\P[\mathcal X_{\infty}^{(2)}(rU)=0]\le \limsup_{r\to \infty}\frac{1}{r^4}\log\P[\mathcal X_{\infty}^{(2)}(rU)=0]\le C_2,
$$
 as $U$ contains some  disk and is contained in some bigger disk. The natural question is following:

\noindent{\bf Question:} {\it What is the exact value of the limit (if exists)
$$
C_{U}:= \lim_{r\to \infty}\frac{1}{r^4}\log\P[\mathcal X_{\infty}^{(2)}(rU)=0]?
$$}
\noindent Recently, it was shown \cite{hole} that the limit exists, under some mild conditions on $U$. The constant, (say) decay constant, is explicit in terms of minimum energy of complement of the set quadratic external fields. Moreover, it was calculated the constants explicitly for sets  like disk, annulus, ellipse, cardioid, equilateral triangle and half disk.   

In this thesis we prove the similar asymptotic results of hole probabilities for $\X_{\infty}^{(\alpha)}$. The essential ideas of the proofs are borrowed from the paper \cite{hole}. As a consequence of these results, for $\alpha=2$, we get the hole probabilities for the infinite Ginibre ensemble.

The key idea of the proof of \eqref{eqn:1/4} is that the set of absolute values of the points of $\X_{\infty}^{(2)}$  has the same distribution as $\{R_1,R_2,\ldots \},$ where $R_k^2\sim \mbox{Gamma}( k,1)$ and all the $R_k$s are independent. We show  that the set of absolute values of the points of $\X_{\infty}^{(\alpha)}$  has the same distribution as $\{R_1,R_2,\ldots \},$ where $R_k^{\alpha}\sim \mbox{Gamma}(\frac{2}{\alpha}k,1)$ and all the $R_k$s are independent. Using this result we show that 
\begin{align*}
\lim_{r\to \infty}\frac{1}{r^{2\alpha}}\log\P[\mathcal X_{\infty}^{(\alpha)}(r\D)=0]=-\frac{\alpha}{2}\cdot\frac{1}{4},
\end{align*}
for $\alpha>0$. In particular, $\alpha=2$ gives \eqref{eqn:1/4}. By the same idea, we also calculate the decay constant for the annulus.

 The above  idea cannot be  applied for non circular domains. One possible way to calculate the hole probabilities, for general sets, is Fredholm determinant. The hole probability for $\X$, a determinantal point process with kernel $\mathbb K$ and measure $\mu$, in terms of Fredholm determinant is given by
 $$
 \P[\X(B)=0]=1+\sum_{n=1}^{\infty}\frac{(-1)^n}{n!}\int_{B}\cdots\int_{B} \det(\mathbb K(x_i,x_j))_{1\le i,j\le n}\prod_{k=1}^{n}d\mu(x_k),
 $$
 where $B$ is any Borel measurable set in $\C$. For the proof of the above equality and more details of the Fredholm determinant see Lemma 3.2.4 and Section 3.4 in \cite{andersonbook} respectively. But, it is difficult to calculate  hole probabilities using the above formula.

We use the potential theory techniques to calculate the hole probabilities for $\X_{\infty}^{(\alpha)}$ in general domains.  Now we explain the proof techniques for infinite Ginibre ensemble. Since the infinite Ginibre ensemble is the distributional limit of finite Ginibre ensembles,  the hole probabilities for infinite Ginibre ensemble is limit of the hole probabilities for finite Ginibre ensembles. The hole probability for $n$-th Ginibre ensemble, in $\sqrt{n} U$, is given by 
 \begin{align*}
 \P[\X_n^{(2)}(\sqrt{n} U)=0]=\frac{1}{Z_n}\int_{ U^c}\ldots \int_{ U^c}  e^{-n\sum_{k=1}^{n}|z_k|^2}\prod_{i<j}|z_i-z_j|^2\prod_{i=1}^n dm(z_i),
 \end{align*}
where $Z_n=n^{-\frac{n^2}{2}}{\pi^n \prod_{k=1}^{n} k!}$ and $U\subset \D$. 

The circular law \cite{ginibre} tells us that the empirical eigenvalue distribution $\rho_n$  of $\frac{1}{\sqrt{n}}G_n$ converges to the uniform measure on the unit disk $\D$ as $n \to \infty$. So, for  $U \subset \D$, $\P[\X_n^{(2)}(\sqrt{n} U)=0]$ converges to zero as $n \to \infty$. Observe  that $\P[\X_n^{(2)}(\sqrt{n} U)=0]=\P[\rho_n \in \mathcal P (U^c) ]$, where $\mathcal P(E)$ denotes the space of all compactly supported probability measures with support in $E$, closed subset of $\C$. Therefore by  Large deviation principle for the empirical eigenvalue distribution of Ginibre ensemble, proved in \cite{hiai}, we have an upper bound for the limits of the hole probabilities,
 $$
\limsup_{n\to
\infty}\frac{1}{n^2}\log \P[\X_n(\sqrt{n} U)=0]\le-\inf_{\mu\in
\mathcal P (U^c)}R_{\mu}^{(2)}+\frac{3}{4},
$$
 where the rate function $R_{\mu}^{(2)}$ is  the following functional on $\mathcal P (\C)$
\begin{equation*}
R_{\mu}^{(2)}=\iint \log\frac{1}{|z-w|}d\mu(z)d\mu(w)+\int
|z|^2d\mu(z),
\end{equation*}
as the set $\mathcal P (U^c)$ is closed in $\mathcal P (\C)$ with weak topology. No non-trivial lower bound for the hole probabilities  can be deduced from the large deviation principle, as the set $\mathcal P (U^c)$ has empty interior. See that, for $a \in U$ and $\mu \in\mathcal P (U^c)$, $(1-\frac{1}{n})\mu + \frac{1}{n} \delta_{a} \notin \mathcal P (U^c)$ for all $n$ and converges to $\mu$ as $n \to \infty$. Nonetheless, using the properties of balayage measures and Fekete points we have the lower bound, with same quantity, for two class of open sets. 

Using similar idea we calculate the hole probabilities for $\X_{\infty}^{(\alpha)}$ for large class of sets. Similar to circular law, we show that the limiting empirical distribution of points of $n^{-\frac{1}{\alpha}}.\X_n^{(\alpha)}$ (each point is scaled down by $n^{-\frac{1}{\alpha}}$) exists. Moreover we calculate the hole probabilities for the determinantal process $\X_n^{(g)}$, using potential theory.

\section{Plan of the thesis}
In this section we give a chapter wise brief description of this thesis.
\begin{itemize}
\item In Chapter \ref{ch:equilibrium} we give definitions of the equilibrium measure and the minimum energy with external fields and their properties. We derive a formula for the equilibrium measures and for the minimum energies for a certain class of sets, in terms of balayage measures. We give examples of balayage measures for a class of external fields. In particular, we calculate the equilibrium measures and the minimum energies for the complement sets of the disk, annulus, ellipse, cardioid, half disk and equilateral triangle with  the quadratic external fields.  

\item In Chapter \ref{ch:finite} we calculate the asymptotics, as $n\to \infty$, of the hole probabilities for $\X_n^{(g)}$ and $\X_n^{(\alpha)}$ for a large class of sets . It turns out that the decay constants are explicit in terms of the minimum energies of the sets, depending on external fields.  As a corollary of these results, we get the hole probability results for finite Ginibre ensembles, proved in \cite{hole}. 

\item In Chapter \ref{ch:infinite} we derive the asymptotic of  $\P[\X_{\infty}^{(\alpha)}(rU)=0]$ as $r\to \infty$ for a large class of open sets $U$. We derive the decay constants explicitly in terms of the minimum energies of the sets, with the external fields $|z|^{\alpha}$. In particular, $\alpha=2$ gives the hole probability results for infinite Ginibre ensemble, proved in \cite{hole}.

\item In chapter \ref{ch:beta} we calculate the hole probabilities for finite $\beta$-ensembles in the complex plane.


\item In Chapter \ref{ch:fluctuations} we move away from hole probabilities. We consider a family of determinantal point processes $\X_L$ in the unit disk with kernels $\mathbb{K}_L$ with respect to the measures $\mu_L$ for $L>0$, where
$$
\mathbb{K}_L(z,w)=\frac{1}{(1-z\bar{w})^{L+1}}\;\;\mbox{and}\;\;d\mu_L(z)=\frac{L}{\pi}(1-|z|^2)^{L-1}dm(z),
$$
for $z,w\in \mathbb{D}$. We show that the asymptotics of variances of $\X_L(r\D)$,  as $r\to 1^-$, does not change with $L$. We also calculate the variances of linear statistics of $\varphi_p(z)=(1-\frac{|z|^2}{r^2})_+^{\frac{p}{2}}$ for $p>0$ and $0<r<1$. Interestingly, it turns out that there is a transitions in variances at $p=1$ as $r\to 1^-$, does not depend on $L$.

\end{itemize}


\chapter{Equilibrium measures}\label{ch:equilibrium}
In this chapter we calculate the weighted equilibrium measures and the weighted minimum  energies for some sets with certain external fields. The minimum energies will play a crucial role in next two chapters, in calculating the hole probabilities.  We start with basic definitions and facts of classical potential theory from \cite{ransfordbook}, \cite{totikbook}.

\section{Preliminaries}
The support of a positive measure $\mu$ on $\mathbb{C}$, denoted by supp$(\mu)$, consists of all points $z$ such that $\mu(D(z,r))>0$ for every open disk $D(z,r)$ of radius $r>0$ and with center at $z$. The measure $\mu$ is said to be compactly supported if supp$(\mu)$ is compact. Let $\mu$ be a compactly supported probability measure on $\mathbb C$. Then its {\it potential} is the function $p_{\mu}:\mathbb C \to (-\infty, \infty]$ defined by
$$
p_{\mu}(z):=-\int \log |z-w|d\mu(w)\;\;\;\mbox{for all $z\in \mathbb C$}.
$$
Its {\it logarithmic energy}  $I_\mu$ is defined by
$$
I_\mu:=-\iint  \log|z-w|d\mu(z)d\mu(w)=\int p_{\mu}(z)d\mu(z).
$$

A set $E \subset \mathbb{C}$ is said to be polar if $I_{\mu}=\infty$ for all compactly supported probability measures $\mu$ with supp$({\mu}) \subset E$. The {\it capacity} of a subset $E$ of $\mathbb C$ is given by
$$
C(E):=e^{-\inf\{I_{\mu}\;:\;\mu \in \mathcal P(E)\}},
$$
where $\mathcal P(E)$ is the space of all compactly supported probability measure with support in $E$.  A set is polar if and only the capacity is zero. Singleton sets are polar sets and countable union of polar sets  is again a polar set. 
A property is said to  hold {\it quasi-everywhere} (q.e.) on $E\subset \mathbb C$ if it holds everywhere on $E$ except some Borel polar set. Every Borel probability measure with finite logarithmic  energy assigns zero measure to Borel polar sets (see Theorem 3.2.3, \cite{ransfordbook}). So, a property, which holds q.e. on $E$, holds {\it $\mu$-a.e.} on $E$, for every $\mu$ with finite energy. As a corollary, we have that every Borel polar set has Lebesgue measure zero and   a property, which holds q.e. on $E$, holds a.e. on $E$.

 A weight function $w:E\to [0,\infty)$, on a closed subset $E$ of $\mathbb{C}$, is said to be {\it
admissible} if it satisfies the following three conditions:
\begin{enumerate}
\item $w$ is upper semi-continuous,

\item $E_0:=\{z\in E|w(z)>0\}$ has positive capacity,
\item if $E$ is unbounded, then $|z|w(z)\to 0$ as
$|z|\to\infty,z\in E$.
\end{enumerate}

\noindent{\bf Weighted equilibrium measure: }A probability measure, with support in $E$,  which minimizes the following weighted energy 
$$
R_{\mu}=\int p_{\mu}(z)d\mu(z)+2\int Q(z)d\mu(z),
$$ where $w=e^{-Q}$ is an admissible weight function, is called {\it weighted equilibrium measure} for $E$ with external field $Q$. The weighted minimum energy is  $R_{E^c}:=\inf\{R_{\mu}\;:\;{\mu\in \mathcal P(E)}\}$. For simplicity we write {\bf minimum energy} and \textbf{equilibrium measure} instead of weighted minimum energy and weighted equilibrium measure respectively. We have the following fact regarding equilibrium measure.

\begin{fact}\label{ft:characterization}
Let $w=e^{-Q}$ be an admissible weight function on a closed set $E$. Then there exists a unique equilibrium measure $\nu$, for $E$ with external field $Q$. The equilibrium measure $\nu$  has compact support and $R_{\nu}$ is finite (so is $I_{\nu}$). Further, $\nu$ satisfies the following conditions 
\begin{eqnarray}\label{eqn:insupport}
 p_{\nu}(z)+Q(z)=C
\end{eqnarray}
for q.e. $z\in \mbox{supp}(\nu)$ and
\begin{eqnarray}\label{eq:outsidesupport}
 p_{\nu}(z)+Q(z)\ge C
\end{eqnarray}
for q.e. $z\in E$ for some constant $C$. Also, the above conditions uniquely characterize the equilibrium measure, i.e. a probability measure with compact support in $E$ and finite energy, which satisfies the conditions \eqref{eqn:insupport} and \eqref{eq:outsidesupport} for some constant $C$, is the equilibrium measure for $E$ with external field $Q$. 
\end{fact}
For a proof of this fact see \cite{totikbook} (Chapter I Theorem 1.3 and Theorem 3.3).
The discrete analogue of the above minimization problem of $R_{\mu}$ is the problem of finding the limit of 
$$
\delta_n^{\omega}(E):=\sup_{z_1,z_2,\ldots,z_n\in E}\left\{
\prod_{i<j}|z_i-z_j|\omega(z_i)\omega(z_j)\right\}^{\frac{2}{n(n-1)}},
$$
as $n \to \infty$. A sets $\mathcal{F}_n=\{z_1^*,z_2^*,\ldots,z_n^* \} \subset E$ is said to be an {\it $n$-th weighted Fekete set} for $E$ 
if 
$$
\delta_n^{\omega}(E)=\left\{
\prod_{i<j}|z_i^*-z_j^*|\omega(z_i^*)\omega(z_j^*)\right\}^{\frac{2}{n(n-1)}}.
$$
The points $z_1^*,z_2^*,\ldots,z_n^*$ in a weighted Fekete set $\mathcal{F}_n$ are called $n$-th  weighted Fekete points.  It is known that the sequence $\{\delta_n^{\omega}(E)\}_{n=2}^{\infty}$ decreases to $e^{-R_{\nu}}$, where $\nu$ is the weighted equilibrium measure, i.e. 
\begin{eqnarray}\label{eqn:limit}
\lim_{n \to \infty} \delta_n^{\omega}(E)= e^{-R_{\nu}}=e^{-\inf_{\mu\in \mathcal P(E)}R_{\mu}}.
\end{eqnarray}
Moreover, the uniform probability measures on $n$-th weighted Fekete sets converge weakly to equilibrium measure $\nu$, i.e.
$$
\lim_{n \to \infty} {\nu}_{\mathcal{F}_n}= {\nu},
$$
where ${\nu}_{\mathcal{F}_n}$ is uniform measure on ${\mathcal{F}_n}$. For the proofs of these facts, see \cite{totikbook}, Chapter III Theorem 1.1 and Theorem 1.3. The following fact  (an application of Theorem 4.7 in Chapter II, \cite{totikbook},  to bounded open sets) is about  the existence and uniqueness of the balayage measure.  
\begin{fact}\label{ft:balayage}
Let $U$ be an bounded open subset of $\C$ and $\mu$ be a finite Borel measure on $U$ (i.e., $\mu(U^c)=0$). Then there exists a unique measure $\hat \mu $ on $\partial U$ such that $\hat \mu(\partial U)=\mu(U)$, $\hat \mu(B)=0$ for every Borel polar set $B\subset \C$ and $p_{\hat \mu(z)}=p_{\mu}(z)$ for q.e. $z\in U^c$. $\hat \mu$ is said to be the {\bf balayage measure} associated with $\mu$ on $U$.  
\end{fact}
\noindent We use the following well known fact, known as Jensen's formula.
\begin{fact}\label{ft:fundamental}
For each $r>0$,
$$
\frac{1}{2\pi}\int_{0}^{2\pi}\log\frac{1}{|z-re^{i\theta}|}d\theta=\left\{\begin{array}{lr}\log
\frac{1}{r} & \mbox{ if } |z|\le r\\\log \frac{1}{|z|} & \mbox{ if
} |z|> r\end{array}\right..
$$
\end{fact}
\noindent To see this, note that $\log |1-z|$ is harmonic on $\D$, by mean value property we have 
$$
\frac{1}{2\pi}\int_{0}^{2\pi}\log|1-re^{i\theta}|d\theta=0 \;\;\mbox{for $r<1$}.
$$
This equality is  true also for $r=1$, by direct calculation. This implies that $\frac{1}{2\pi}\int_{0}^{2\pi}\log|z-re^{i\theta}|d\theta=\log(\max\{|z|,r\})$.

\section{Equilibrium measure in general settings}
In this section we derive the weighted equilibrium measures and the minimum energies of certain sets, for a class of external fields. Let $Q(z)=\frac{g(|z|)}{2}$, where the function $g:[0,\infty)\to [0,\infty)$ satisfies the following conditions:
\begin{enumerate}
\item $g(r)$ is increasing function in $r$ such that $re^{-\frac{g(r)}{2}}\to 0$ as $r\to \infty$.

\item $g$ is a twice differentiable function on $(0,\infty)$.

\item $\lim\limits_{r\to 0^+}rg'(r)=0$ and $rg'(r)$ is increasing on $(0,\infty)$.
\end{enumerate}
Observe that, $g(r)=r^{\alpha}$ satisfies the above conditions for $\alpha>0$. Clearly, $w(z)=e^{-\frac{g(|z|)}{2}}$ is an admissible weight function on such $E$. In this section we assume  that $g$ satisfies above three conditions.

Through out this thesis, $T$ is the solution of $tg'(t)=2$ (i.e., $Tg'(T)=2$) and $D(0,T)$ is the disk of radius $T$ with center at origin. The third condition on $g$ implies that $T$ is unique. We shall use following notations:
\begin{align*}
&R_U^{(g)}=\inf\{R_{\mu}^{(g)}\suchthat \mu \in \mathcal P(U^c)\}\;\;\mbox{ where $R_{\mu}^{(g)}=\iint\log \frac{1}{|z-w|}d\mu(z)d\mu(w)+\int g(|z|)d\mu(z)$};
\\&R_U^{(\alpha)}=\inf\{R_{\mu}^{(\alpha)}\suchthat \mu \in \mathcal P(U^c)\}\;\;\mbox{ where $R_{\mu}^{(\alpha)}=\iint\log \frac{1}{|z-w|}d\mu(z)d\mu(w)+\int |z|^{\alpha}d\mu(z)$}.
\end{align*}
Note that $R_U^{(\alpha)}$ is a slight abuse of the notation $R_U^{(g)}$. First we calculate the equilibrium measure for $E=\C$, then for $E=\C\backslash U$ for certain class of open sets $U\subset D(0,T)$ with the external fields $\frac{g(|z|)}{2}$. 
\begin{theorem}\label{thm:diskgeneral}
 Suppose $E=\mathbb C$ and $Q(z)=\frac{g(|z|)}{2}$ (external fields) on $\mathbb C$.
Then the equilibrium measure $\mu$, set $z=re^{i\theta}$, is given by
$$
d\mu(z)=\left\{\begin{array}{lr} \frac{1}{4\pi}[g''(r)+\frac{1}{r}g'(r)]dm(z)& \mbox{if  }
|z|\le T\\0& \mbox{otherwise}
\end{array}\right.,
$$
where $T>0$ such that $Tg'(T)=2$. The minimum energy is 
$$
R_{\emptyset}^{(g)}=\log\frac{1}{T}+g(T)-\frac{1}{4}\int_{0}^T r(g'(r))^2dr.
$$
\end{theorem}

\noindent Before proving the theorem we present two examples.
\begin{example}
\begin{enumerate}

\item If $g(r)=r^{\alpha}$, then $T=(\frac{2}{\alpha})^{\frac{1}{\alpha}}$. The equilibrium measure is $d\mu(z)=\frac{\alpha^2}{4\pi}r^{\alpha-2}dm(z)$ on $D(0,(\frac{2}{\alpha})^{\frac{1}{\alpha}})$ and the minimum energy is $R_{\emptyset}^{(\alpha)}=\frac{3}{4}\cdot\frac{2}{\alpha}-\frac{1}{\alpha}\log\frac{2}{\alpha}$.

\item In particular when $g(r)=r^2$, i.e., $\alpha=2$ then $T=1$. The equilibrium measure $\mu$ is uniform measure on $\D$, i.e., $d\mu(z)=\frac{1}{\pi}dm(z)$ on $\D$ and the minimum energy is $R_{\emptyset}^{(2)}=\frac{3}{4}$. Proof of this particular case can be found in \cite{hole}.

\end{enumerate}
\end{example}

 \begin{proof}[Proof of Theorem \ref{thm:diskgeneral}]
 Let $d\mu(z)=\frac{1}{4\pi}[g''(r)+\frac{1}{r}g'(r)]dm(z)$ when $z\in D(0,T)$ and zero other wise.  The condition $Tg'(T)=2$ implies that $\mu$ is a probability measure on $D(0,T)$. Indeed, 
 \begin{align*}
 \frac{1}{4\pi}\int_{0}^T\int_0^{2\pi}[g''(r)+\frac{1}{r}g'(r)]rdrd\theta=\frac{1}{2}\int_{0}^T[rg''(r)+g'(r)]dr=\frac{1}{2}Tg'(T)=1.
 \end{align*}
 We show that the measure $\mu$ satisfies the conditions \eqref{eqn:insupport} and \eqref{eq:outsidesupport}. Hence by Fact \ref{ft:characterization} we conclude that $\mu$ is the equilibrium measure.

  By Fact \ref{ft:fundamental}, for $|z|\le T$, we have
\begin{align*}
p_{\mu}(z)&=\frac{1}{4\pi}\int_{0}^T\int_0^{2\pi}\log\frac{1}{|z-re^{i\theta}|}.(g''(r)+\frac{1}{r}g'(r))rdrd\theta\nonumber
\\&=\frac{1}{2}\left[\int_0^{|z|}\log\frac{1}{|z|}.(rg''(r)+g'(r))dr+\int_{|z|}^T\log\frac{1}{r}.(rg''(r)+g'(r))dr\right]\nonumber
\\&=\frac{1}{2}\left[2\log\frac{1}{T}+g(T)-g(|z|)\right],
\end{align*}
the last equality follows from the facts that $\lim_{r\to 0^+}rg'(r)=0$ and $Tg'(T)=2$. Therefore 
\begin{align}\label{eqn:disk1}
p_{\mu}(z)+\frac{g(|z|)}{2}=\frac{1}{2}\left[2\log\frac{1}{T}+g(T)\right] \;\;\mbox{for $|z|\le T$}.
\end{align}
Hence $\mu$ satisfies the condition \eqref{eqn:insupport}. On other hand, for $|z|>T$
\begin{align*}
p_{\mu}(z)&=\frac{1}{4\pi}\int_{0}^T\int_0^{2\pi}\log\frac{1}{|z-re^{i\theta}|}.(g''(r)+\frac{1}{r}g'(r))rdrd\theta
\\&=\frac{1}{2}\left[\int_0^{T}\log\frac{1}{|z|}.(rg''(r)+g'(r))dr\right]\;\; \mbox{ (by Fact \ref{ft:fundamental})}
\\&=\log\frac{1}{|z|},
\end{align*}
we get last equality by using the facts $\lim_{r\to 0^+}rg'(r)=0$ and $Tg'(T)=2$. The function $f(r)=\log\frac{1}{r}+\frac{g(r)}{2}$ is increasing function on $[T,\infty)$. Indeed, $f'(r)=-\frac{1}{r}+\frac{g'(r)}{2}\ge 0$ for $r\ge T$, as $rg'(r)$ is increasing. Therefore 
\begin{align}\label{eqn:disk2}
p_{\mu}(z)+\frac{g(|z|)}{2}=\log\frac{1}{|z|}+\frac{g(|z|)}{2}\ge \frac{1}{2}\left[2\log\frac{1}{T}+g(T)\right]\;\;\mbox{for $|z|> T$}.
\end{align}
Hence $\mu$ satisfies the condition \eqref{eq:outsidesupport}. Therefore $\mu$ is the equilibrium measure.

\noindent {\bf Value of $R_{\emptyset}^{(g)}$: } We have
\begin{align*}
R_{\emptyset}^{(g)}&=\int p_{\mu}(z)d\mu(z)+\int g(|z|)d\mu(z)
\\&=\frac{1}{2}\left[2\log\frac{1}{T}+g(T)\right]+\frac{1}{2}\int g(|z|)d\mu(z) \;\;\mbox{(by \eqref{eqn:disk1})}
\\&=\log\frac{1}{T}+\frac{1}{2}g(T)+\frac{1}{2}\int_0^{T}\int_0^{2\pi}g(r).\frac{1}{4\pi}(rg''(r)+g'(r))dr d\theta
\\&=\log\frac{1}{T}+g(T)-\frac{1}{4}\int_0^{T}r(g'(r))^2dr,
\end{align*}
by integration by parts and using $\lim_{r\to 0^+}rg'(r)=0$ and $Tg'(T)=2$.
\end{proof}

 Next theorem gives the equilibrium measures and the minimum energies for complements of open subsets of $D(0,T)$. 

\begin{theorem}\label{thm:generalsetting}
Let $E=\C\backslash U=U^c$, where $U \subset D(0,T)$ is an open set.  Then the equilibrium measure for $U^c$, with the external field $\frac{g(|z|)}{2}$, is
$\nu=\mu_1+\nu_2$ and
\begin{align}\label{eqn:genrobinconstant}
R_U^{(g)}=R_{\emptyset}^{(g)}+\frac{1}{2}\left[\int_{\partial U}g(|z|)d\nu_2(z)-\int_{U}g(|z|)d\mu_2(z)\right],
\end{align}
where $\mu_1$ and  $\mu_2$ are  restrictions of the measure $\mu$, as in Theorem \ref{thm:diskgeneral}, on to the sets  $D(0,T)\backslash U $ and $ U$ respectively, i.e.,
\begin{eqnarray*}
d\mu_1(z)&=&\left\{\begin{array}{lr}
\frac{1}{4\pi}(g''(r)+\frac{1}{r}g'(r))dm(z)& \mbox{  if  } z\in D(0,T)\backslash U\\0& \mbox{otherwise}
\end{array}\right.
\\d\mu_2(z)&=&\left\{\begin{array}{lr}
\frac{1}{4\pi}(g''(r)+\frac{1}{r}g'(r))dm(z) & \mbox{  if  } z\in  {U}\\0& \mbox{otherwise}
\end{array}\right.
\end{eqnarray*}
and  $\nu_2$ is the balayage measure on $\partial U$ with respect to the measure $\mu_2$.
\end{theorem}

\begin{remark}\label{re:scale}
\begin{enumerate}
\item If $g(r)=r^{\alpha}$ and $\alpha>0$. Then for $U\subseteq D(0,(\frac{2}{\alpha})^{\frac{1}{\alpha}})$, the constant is 
$$
R_U^{(\alpha)}=\frac{3}{4}\cdot\frac{2}{\alpha}-\frac{1}{\alpha}\log\frac{2}{\alpha}+\frac{1}{2}\left[\int_{\partial U}g(|z|)d\nu_2(z)-\int_{U}g(|z|)d\mu_2(z)\right].
$$

\item If $g(r)=r^{\alpha}$ and $\alpha>0$. Then the constant
$$
R_U^{(\alpha)'}=\frac{1}{2}\left[\int_{\partial U}g(|z|)d\nu_2(z)-\int_{U}g(|z|)d\mu_2(z)\right]
$$ 
satisfies the scaling property $R_{a U}^{(\alpha)'}=a^{2\alpha}R_U^{(\alpha)'}$.

\item We can see Theorem \ref{thm:diskgeneral} as a particular case of Theorem \ref{thm:generalsetting}, when $U=\emptyset$. But we shall use Theorem \ref{thm:diskgeneral} to prove Theorem \ref{thm:generalsetting}.
\end{enumerate}
\end{remark}

Now, we proceed to prove Theorem \ref{thm:generalsetting}.

\begin{proof}[Proof of Theorem \ref{thm:generalsetting}] 
Let $\mu$ be the equilibrium measure for $\C$ with external field $\frac{g(|z|)}{2}$, as in Theorem \ref{thm:diskgeneral}. Let $\mu=\mu_1+\mu_2$, where $\mu_1$ and $\mu_2$ are $\mu$ restricted to $U^c$ and $U$ respectively.  By Fact \ref{ft:balayage}, we know that there exists a measure $\nu_2$ on $\partial U$ such that $\nu_2(\partial U)=\mu_2(U)$, $\nu_2(B)=0$ for every Borel polar set and 
$$
p_{\nu_2}(z)=p_{\mu_2}(z)\;\; \mbox{ q.e. on $U^c$}.
$$ 
Define $\nu=\mu_1+\nu_2$. Then  we have that the support of $\nu$ is $\overline{D(0,T)}\backslash U$ and 
$$
p_{\nu}(z)=p_{\mu_1}(z)+p_{\nu_2}(z)=p_{\mu_1}(z)+p_{\mu_2}(z)=p_{\mu}(z)\;\; \mbox{ q.e. on $U^c$}.
$$
Again  we have \eqref{eqn:disk1} and \eqref{eqn:disk2}, i.e., 
\begin{eqnarray*}
p_{\mu}(z)+\frac{g(|z|)}{2}&=&\frac{1}{2}\left[2\log\frac{1}{T}+g(T)\right]\;\;\mbox{ if  $|z|\le T$}
\\ p_{\mu}(z)+\frac{g(|z|)}{2}&\ge &\frac{1}{2}\left[2\log\frac{1}{T}+g(T)\right]\;\;\mbox{ if $|z|>T$}.
\end{eqnarray*}
This gives us that 
\begin{eqnarray*}
p_{\nu}(z)+\frac{g(|z|)}{2}&=&\frac{1}{2}\left[2\log\frac{1}{T}+g(T)\right]\;\;\mbox{ for q.e. $z\in \mbox{supp}(\nu) $}
\\ p_{\nu}(z)+\frac{g(|z|)}{2}&\ge &\frac{1}{2}\left[2\log\frac{1}{T}+g(T)\right]\;\;\mbox{ for q.e. $z\in  U^c$}.
\end{eqnarray*}
The energy of the measure  $\nu$,
\begin{eqnarray}\label{eqn:genenegry}
I_{\nu}&=&\int p_{\nu}(z)d\nu(z)
=\frac{1}{2}\left[2\log\frac{1}{T}+g(T)\right]-\frac{1}{2}\int g(|z|)d\nu(z),
\end{eqnarray}
is finite. The second equality follows from the fact that  $\nu(B)=0$ for all Borel polar sets $B$. So, $\nu $ has finite energy and satisfies conditions \eqref{eqn:insupport} and \eqref{eq:outsidesupport}. Therefore,  by Fact \ref{ft:characterization}, $\nu$ is the equilibrium measure for $U^c$ with the external field $\frac{g(|z|)}{2}$.

\vspace{.5cm} \noindent{\bf Value of $R_U^{(g)}$ :}  
We have 
$$
R_{\nu}^{(g)}=\int p_{\nu}(z)d\nu(z)+\int g(|z|)d\nu(z)=I_{\nu}+\int g(|z|)d\nu(z).
$$
Therefore, by \eqref{eqn:genenegry},  we have
\begin{eqnarray*}
R_{\nu}^{(g)}&=&\frac{1}{2}\left[2\log\frac{1}{T}+g(T)\right]+\frac{1}{2}\int g(|z|)d\nu(z)
\\&=&\frac{1}{2}\left[2\log\frac{1}{T}+g(T)\right]+\frac{1}{2}\int g(|z|)d\mu_1(z)+\frac{1}{2}\int g(|z|)d\nu_2(z)
\\&=&R_{\emptyset}^{(g)}-\frac{1}{2}\int_U g(|z|)d\mu_2(z)+\frac{1}{2}\int_{\partial U} g(|z|)d\nu_2(z)
\\&=&R_{\emptyset}^{(g)}+\frac{1}{2}\left[\int_{\partial U} g(|z|)d\nu_2(z)-\int_U g(|z|)d\mu_2(z)\right].
\end{eqnarray*}
The result follows from the fact that $R_U^{(g)}=R_{\nu}^{(g)}$.
\end{proof}

\begin{remark}\label{relations}
Let $\nu_2$ and $\mu_2$ be as in the above proof i.e. $\nu_2$ is the balayage measure associated with $\mu_2$.  We have $p_{\nu_2}(z)=p_{\mu_2}(z)$ for q.e.  $z\in { U}^c$.  As the logarithmic potential of a measure is harmonic outside its support, $p_{\nu_2}(z)=p_{\mu_2}(z)$ holds for every $z \in {\bar U}^c$. Outside $D(0,T)$,  $p_{\nu_2}(z)$ and $p_{\mu_2}(z)$ are real parts of the analytic functions $ -\int_{\partial U}\log{(z-w)}d\nu_2(w)$ and $ -\int_{\partial U}\log{(z-w)}d\mu_2(w)$, respectively. 
So  there exists a constant $c$ such that for all $|z|>T$,
\begin{eqnarray}\label{eqn:moment}
\int_{\partial U}\log{(z-w)}d\nu_2(w)&=&\int_{ U}\log{(z-w)}d\mu_2(w) +c,\nonumber\\
 \Leftrightarrow \int_{\partial U}[\log{z}+\sum_{n=1}^{\infty}
\frac{w^n}{nz^n}]d\nu_2(w)&=&\int_{ U}[\log{z}+\sum_{n=1}^\infty
\frac{w^n}{nz^n}]d\mu_2(w)+c\nonumber\\
\Leftrightarrow \int_{\partial U}w^n d\nu_2(w)&=&\int_{ U}w^nd\mu_2(w), \forall
n\geqslant0 \mbox{  and $c=0$}
\end{eqnarray}
To see the converse of the above, suppose $\nu_2$ is a measure on $\partial U$ which satisfies  the relations \eqref{eqn:moment}, then $p_{\nu_2}(z)=p_{\mu_2}(z)$ for every   $z\in { \bar U}^c.$ If $\partial U$ is a piecewise smooth curve and $\nu_2$ has  density with respect to arc-length on $\partial U$, then $p_{\nu_2}(z)$ is continuous  at all the continuity points of the density of $\nu_2$ (\cite{totikbook} Chapter II Theorem 1.5). In this case $p_{\mu_2}(z)$ is also continuous on $\partial U$. So if the density of $\nu_2$ is piecewise continuous   on $\partial U$, we get that 
$p_{\nu_2}(z)=p_{\mu_2}(z)$ for q.e.   $z\in {  U}^c.$  Therefore when $\partial U$ is piecewise smooth curve,  a measure $\nu_2$ on $\partial U$ which has piecewise continuous density with respect to arclength and satisfies relations  \eqref{eqn:moment} is the balayage measure on $\partial U$.  
\end{remark}

\subsection{Examples of balayage measures}
In this section we calculate the balayage measures for disk and annulus with the external field $\frac{g(|z|)}{2}$.
\begin{example}\label{ex:gdisk}
Let $U=D(0,a)$, disk of radius $a<T$ centered at origin. Then the balayage measure on $\partial U$ with respect to the measure $\mu\big |_U$, where $\mu$ as in Theorem \ref{thm:diskgeneral}, is
\begin{align*}
d\nu_2(z)&=\left\{\begin{array}{lr}
\frac{1}{4\pi}ag'(a)d\theta & \mbox{if } |z|=a,
\\0 & \mbox{otherwise}. 
\end{array}\right.
\end{align*} 
\end{example}

\begin{example}\label{ex:gannulus}
Let $U=\{z\suchthat 0< a<|z|<b<T\}$, annulus with inner radius $a$ and outer radius $b$. Then the balayage measure on $\partial U$ with respect to the measure $\mu\big |_U$, where $\mu$ as in the Theorem \ref{thm:diskgeneral}, is 
\begin{align*}
d\nu_2(z)&=\left\{\begin{array}{lr}
\lambda.\frac{1}{4\pi}(bg'(b)-ag'(a))d\theta & \mbox{if } |z|=a,
\\(1-\lambda).\frac{1}{4\pi}(bg'(b)-ag'(a))d\theta & \mbox{if } |z|=b,
\\0 & \mbox{otherwise},
\end{array}\right.
\end{align*} 
where $\lambda$ is given by
$$
\lambda=\frac{(g(b)-g(a))-ag'(a)\log(b/a)}{(bg'(b)-ag'(a))\log(b/a)}.
$$
\end{example}

\begin{remark}\label{re:constant}
Suppose $g(t)=t^{\alpha}$, for $\alpha>0$. 
\begin{enumerate}
\item If $U=D(0,a)$, where $a\le (\frac{2}{\alpha})^{\frac{1}{\alpha}}$. Then the balyage measure on $\partial U$ and minimum energy are given below:
\begin{align*}
d\nu_2(z)=\left \{\begin{array}{lr}
\frac{\alpha}{4\pi}a^{\alpha}d\theta & \mbox{if } |z|=a,
\\0 & \mbox{otherwise},
\end{array}\r.
\;\;\mbox{ and }\;\;\; R_U^{(\alpha)}-R_{\emptyset}^{(\alpha)}=\frac{\alpha}{2}\cdot\frac{a^{2\alpha}}{4}.
\end{align*}
\item If $U=\{z\suchthat 0< a<|z|<b<(\frac{2}{\alpha})^{\frac{1}{\alpha}}\}$, annulus with inner radius $a$ and outer radius $b$. Then the balayage measure on $\partial U$ is 
\begin{align*}
d\nu_2(z)&=\left\{\begin{array}{lr}
\lambda.\frac{\alpha}{4\pi}(b^{\alpha}-a^{\alpha})d\theta & \mbox{if } |z|=a,
\\(1-\lambda).\frac{\alpha}{4\pi}(b^{\alpha}-a^{\alpha})d\theta & \mbox{if } |z|=b,
\\0 & \mbox{otherwise},
\end{array}\right.
\mbox{ for }\lambda=\frac{(b^{\alpha}-a^{\alpha})-\alpha a^{\alpha}\log(b/a)}{\alpha(b^{\alpha}-a^{\alpha})\log(b/a)}.
\end{align*}
The minimum energy is given by
\begin{align*}
R_U^{(\alpha)}-R_{\emptyset}^{(\alpha)}=\frac{\alpha}{2}\cdot \l(\frac{b^{2\alpha}}{4}-\frac{a^{2\alpha}}{4}-\frac{(b^{\alpha}-a^{\alpha})^2}{2\alpha\log(b/a)}\r).
\end{align*}

\end{enumerate}

\end{remark}

\noindent We show the computations for the Example \ref{ex:gannulus} and we skip the (similar) calculations for Example \ref{ex:gdisk}. 

\begin{proof}[Computations for Example \ref{ex:gannulus}] If $|z|\le a$, 
then by Fact \ref{ft:fundamental} we have 
\begin{align*}
p_{\mu_2}(z)&=\frac{1}{4\pi}\int_a^b\int_0^{2\pi}\log \frac{1}{|z-re^{i\theta}|}.\left[g''(r)+\frac{1}{r}g'(r)\right]rdr\theta
\\&=\frac{1}{2}\int_a^b[rg''(r)+g'(r)].\log\frac{1}{r}.dr
\\&=\frac{1}{2}\left[-bg'(b)\log b+ag'(a)\log a+\int_a^bg'(r)dr\right]
\\&=\frac{1}{2}\left[(g(b)-g(a))-bg'(b)\log b+ag'(a)\log a\right].
\end{align*}
Again for $|z|\le a$, the potential for $\nu_2$ at $z$ is
\begin{align*}
p_{\nu_2}(z)=&\frac{1}{4\pi}\int_{0}^{2\pi}\lambda.(bg'(b)-ag'(a))\log\frac{1}{|z-ae^{i\theta}|}d\theta 
\\& + \frac{1}{4\pi}\int_{0}^{2\pi}(1-\lambda).(bg'(b)-ag'(a))\log\frac{1}{|z-be^{i\theta}|}d\theta
\\=& \frac{\lambda}{2}\cdot\left[bg'(b)-ag'(a)\right]\cdot\log(b/a)-\frac{1}{2}\cdot[bg'(b)-ag'(a)]\cdot\log b,
\end{align*}
last equality follows from the Fact \ref{ft:fundamental}. By equating $p_{\mu_2}(z)=p_{\nu_2}(z)$ for $|z|\le a$, we get 
\begin{align*}
\lambda=\frac{(g(b)-g(a))-ag'(a)\log(b/a)}{(bg'(b)-ag'(a))\log(b/a)}.
\end{align*}
Similarly, it can be shown that $p_{\mu_2}(z)=p_{\nu_2}(z)$ for $|z|\ge b$ for all choice of $\lambda$. Therefore $p_{\mu_2}(z)=p_{\nu_2}(z)$ if $z\in U^c$ for the above particular choice of $\lambda$. Hence the result.  
\end{proof}

\section{Equilibrium measure with quadratic external fields}
In this section we calculate the equilibrium measures  and the minimum energies for some sets with quadratic external  fields. The following theorem has been proved in \cite{hole}, a particular case (for $g(|z|)=|z|^2$) of Theorem \ref{thm:generalsetting}.
\begin{theorem}\label{thm:generalformula}
Let $U$ be an open set such that ${U} \subseteq {\mathbb D}$.  Then the equilibrium measure for $U^c$, under logarithmic potential with quadratic  external field, is
$\nu=\nu_1+\nu_2$ and
\begin{align}\label{eqn:robinconstant}
R_U^{(2)}=\frac{3}{4}+\frac{1}{2}\left[\int_{\partial U}|z|^2d\nu_2(z)-\frac{1}{\pi}\int_{{U}}|z|^2dm(z)\right],
\end{align}
where 
\begin{eqnarray*}
d\nu_1(z)&=&\left\{\begin{array}{lr}
\frac{1}{\pi}dm(z)& \mbox{  if  } z\in{\mathbb D}\backslash U\\0& \mbox{o.w.}
\end{array}\right.
\end{eqnarray*}
and  $\nu_2$ is the balayage measure on $\partial U$ with respect to the measure $\frac{1}{\pi}m\big |_U$. 
\end{theorem}
Next subsection we have calculated the equilibrium measures $\nu  $ and the constant  $R_{\nu}$ for some particular sets.
\subsection{Table of examples}\label{results}
 Suppose $\nu=\nu_1+\nu_2$ is the equilibrium measure for $\bar{\mathbb D}\backslash U$ as in Theorem \ref{thm:generalformula}. In this section we write $R_U$ and $R_U'$ instead of $R_U^{(2)}$ and $R_U^{(2)'}$ respectively. Then $R_U=\frac{3}{4}+R_U'$.
The balayage measure $\nu_2$ and $R_U'$, for some particular
open sets $U$, are given in the following table.

\begin{center}
\footnotesize{\begin{tabular}{ | m{3.1cm} | m{6.4cm}| m{2cm} | }
\hline
\begin{center}{\bf $U$}\end{center} & \begin{center}{\bf $\nu_2$}\end{center} & \begin{center}{\bf $R_{U}'$}\end{center} \\
\hline
$\{z:|z|<a\}$ (disk).&
$d\nu_2(z)=\left\{\begin{array}{lr}\frac{a^2}{2\pi}d\theta &
\mbox{if $z=ae^{i\theta}$}\\0& \mbox{o.w.}\end{array}\right.$
& \begin{center}$\frac{a^4}{4}$\end{center} \\
\hline
$\{z:|z-a_0|<a\}$ \mbox{for fixed $a_0\in \mathbb D .$}
&$d\nu_2(z)=\left\{\begin{array}{lr}\frac{a^2}{2\pi}d\theta &
\mbox{if $z=a_0+ae^{i\theta}$}\\0& \mbox{o.w.}\end{array}\right.$
& \begin{center}$\frac{a^4}{4}$\end{center}\\
\hline
$\{z:a<|z|<b\}$ for $0<a<b<1,$ (annulus).&
 $ d\nu_2'(z) =\left\{\begin{array}{lr}
\lambda(b^2-a^2)\frac{d\theta}{2\pi}&\mbox{if } z=ae^{i\theta}\\0&
\mbox{o.w.}
\end{array}\right.$,
$d\nu_2''(z)$ $= \left\{\begin{array}{lr}
(1-\lambda)(b^2-a^2)\frac{d\theta}{2\pi}&\mbox{if } z=be^{i\theta}\\0&
\mbox{o.w.}
\end{array}\right.$ where
$\lambda = \frac{(b^2-a^2)-2a^2\log(b/a)}{2(b^2-a^2)\log(b/a)}$
and $\nu_2=\nu_2'+\nu_2''$.
& $\frac{1}{4}(b^4-a^4)$ $-\frac{1}{4}\frac{(b^2-a^2)^2}{\log(b/a)}$ \\
\hline
 $\{(x,y) |
\frac{x^2}{a^2}+\frac{y^2}{b^2}\leqslant 1\}$\;\;\; (ellipse).
&$d\nu_2(z)=\frac{ab}{2\pi}\left[1-{\frac{a^2-b^2}{a^2+b^2}}\cos(2\theta)\right]d\theta$
\;\;\;\;when $z\in \partial U$.&
$\frac{1}{2}\cdot\frac{(ab)^3}{a^2+b^2}$ \\
\hline
 $\{re^{i\theta}|0\leq r< b(1+2a\cos\theta), 0\le \theta\le
2\pi\}$ (Cardioid). & $d\nu_2(z)=\frac{b^2}{2\pi}(1+a^2+2a\cos
\theta)d\theta
$ \;\;\;\;\;\;\;when $z\in \partial U$. & $\frac{b^4}{2}(a^2+1)^2-\frac{b^4}{4}$\\
\hline
Fix $a<1$. $aT$ where $T$ be triangle with cube roots of
unity $1,\omega,\omega^2$ as vertices. &
\begin{center}$\ldots$\end{center}
&$\frac{a^4}{2\pi}\cdot \frac{9\sqrt{3}}{80}$\\
\hline
$\{re^{i\theta}:0<r<a,0<\theta<\pi\}$, (half-disk).
&\begin{center}$\ldots$\end{center}  &
$\frac{a^4}{2}\l(\frac{1}{2}-\frac{4}{\pi^2}\r)$\\
 \hline
\end{tabular}}
\end{center}
For detailed calculations, see Section \ref{example}.

\subsection{Computations for examples}\label{example}
In this section we calculate the balayage measure $\nu_2$ and the
constant $R_U$ explicitly for some particular open sets $U$. In the first example we consider annulus with inner and outer radius $a$ and $b$ respectively, a particular case of 
second part of Remark \ref{re:constant} (for $\alpha=2$).

\begin{example}\label{thm:annulus}
Fix $0<a<b<1$. Suppose $U=\{z\in \mathbb C: a<|z|<b\}$.
Then the balayage measure is $\nu_2=\nu_2'+\nu_2''$, where
\begin{eqnarray*}
 d\nu_2'(z)&=&\left\{\begin{array}{lr}
\lambda(b^2-a^2)\frac{d\theta}{2\pi}& \mbox{ if  } |z|=a\\0&
\mbox{o.w.}
\end{array}\right.
\\d\nu_2''(z)&=&\left\{\begin{array}{lr}
(1-\lambda)(b^2-a^2)\frac{d\theta}{2\pi}& \mbox{ if  } |z|=b\\0&
\mbox{o.w.}
\end{array}\right.
\end{eqnarray*}
and  $\lambda $ is given by
$$
\lambda = \frac{(b^2-a^2)-2a^2\log(b/a)}{2(b^2-a^2)\log(b/a)}.
$$
The constant is
$$
R_{U}=\frac{3}{4}+\frac{1}{4}(b^4-a^4)-\frac{1}{4}\frac{(b^2-a^2)^2}{\log(b/a)}.
$$
\end{example}

\noindent {\bf Note: }In particular if $a=b$, then
$R_{U}=\frac{3}{4}$. Again if
we take aspect ratio $a/b=c$, then
$$
R_{U}=\frac{3}{4}+\frac{1}{4}\left((1-c^4)+\frac{(1-c^2)^2}{\log
c}\right)b^4.
$$
Note that the same  expression has appeared  in hole probability for infinite Ginibre ensemble (Corollary \ref{thm:annulus1}).


In the next example we consider disk of radius $a$ contained in unit disk. Let $B(c_0,a)$ be the ball of radius $a$ centered at $c_0$.

 \begin{example}\label{ex=cB}
For $U=B(c_0,a)\subseteq \D$,
 the equilibrium measure is $\nu=\nu_1+\nu_2$, where
 $$
d\nu_1(z)=\left\{\begin{array}{lr}
\frac{1}{\pi}dm(z)& \mbox{  if  } z\in\bar{\mathbb D}\backslash U\\0& \mbox{o.w.}
\end{array}\right.
\hspace{.5cm} \mbox{ and } \;\;
d\nu_2(z)=\left\{\begin{array}{lr}
\frac{a^2}{2\pi}d\theta& \mbox{  if  } |z-c_0|=a\\0& \mbox{o.w.}
\end{array}\right.
$$
and the constant is $R_U=\frac{3}{4}+\frac{1}{4}a^4$.

\end{example}

Note that the equilibrium measure and the constant do not depend
on the position of the ball. These values depend only on the radius of
the ball. This follows directly from the fact that the balayage measure corresponding to uniform measure on a ball is uniform on its boundary, which follows easily from Fact \ref{ft:fundamental}. 
 Now we consider  ellipse.

\begin{example}\label{ex:ellipse}
Fix $0<a,b<1$. Suppose
$U=\{(x,y) | \frac{x^2}{a^2}+\frac{y^2}{b^2}< 1\}$. Then
$$
d\nu_2(z)=\frac{ab}{2\pi}\left[1-{\frac{a^2-b^2}{a^2+b^2}}\cos(2\theta)\right]d\theta\;\;\mbox{and
}\;\; R_U=\frac{3}{4}+\frac{1}{2}\cdot\frac{(ab)^3}{a^2+b^2}.
$$
\end{example}

\begin{proof}[Computation for Example \ref{ex:ellipse}]
Let $x=ar\cos{\theta},  y=br\sin{\theta}, 0 \leq \theta \leq 2\pi, 0
\leq r \leq 1$. Then we have
\begin{eqnarray*}
\frac{1}{\pi}\int_{U}w^ndm(w)&=&\frac{1}{\pi}\int_{U}(x+iy)^n dx dy\\
  &=&\frac{1}{\pi} \int_{0}^{1}\int_{0}^{2\pi}r^n(a\cos{\theta}+ib\sin{\theta})^n abr d{\theta}dr
\\&=&\frac{1}{\pi}\int_{0}^{2\pi}\frac{ab(a\cos{\theta}+ib\sin{\theta})^n}{n+2}d{\theta}
\\&&(\mbox{by substituting $\alpha=\frac{a+b}{2}, \beta=\frac{a-b}{2}$})
\\&=&\frac{1}{\pi}\int_{0}^{2\pi}\frac{({\alpha}^2-{\beta}^2)(\alpha e^{i\theta}+\beta e^{-i\theta})^n}{n+2}d{\theta} \\
\\&=& \l\{\begin{array}{lr}
\frac{1}{\pi}\cdot\frac{({\alpha}^2-{\beta}^2){\alpha}^{n/2}{\beta}^{n/2}{n \choose
n/2}}{n+2} & \mbox{if $n$ is even}\vspace{.3cm}
\\0 & \mbox{if $n$ is odd.}
\end{array}\r.
\end{eqnarray*}

Let $d\nu_2(w)=\frac{1}{\pi}[c_0+c_1(e^{2i\theta}+e^{-2i\theta})]d\theta$, then we have
\begin{eqnarray*}
\int_{\partial U}w^nd\nu_2(w)&=&\frac{1}{\pi}\int_{0}^{2\pi}{(\alpha e^{i\theta}+\beta e^{-i\theta})^n}[c_0+c_1(e^{2i\theta}+e^{-2i\theta})]d{\theta}\\
&=& \l\{\begin{array}{lr}\frac{1}{\pi}\left[ c_0{\alpha}^{n/2}{\beta}^{n/2}{n \choose
n/2}+c_1{n \choose
n/2-1}{\alpha}^{n/2}{\beta}^{n/2}(\frac{\alpha^2+\beta^2}{\alpha\beta})\right]
& \mbox{if $n$ is even}\vspace{.3cm}
\\0 & \mbox{if $n$ is odd.}
\end{array}\r.\\
\end{eqnarray*}
Note that if we take
$$({\alpha}^2-{\beta}^2)\frac{{n\choose n/2}}{n+2}= c_0{n \choose
n/2}+c_1{n \choose
n/2-1}\l(\frac{\alpha^2+\beta^2}{\alpha\beta}\r),\; \mbox{ for all $n$
even}
$$
which implies that
$$
\frac{({\alpha}^2-{\beta}^2)}{2}=c_0=-c_1\l(\frac{\alpha^2+\beta^2}{\alpha\beta}\r).
$$
Therefore $\nu_2$ satisfies \eqref{eqn:moment} for all $n$ and also has continuous density with respect to arclength of $\partial U$. Therefore, by Remark \ref{relations},  the   measure $\nu_2$ on $\partial U$  given by
$$
d\nu_2(z)=\frac{ab}{2\pi}\left[1-{\frac{a^2-b^2}{a^2+b^2}}\cos(2\theta)\right]d\theta
$$
is the balayage measure on $\partial U$ and constant $R_U$ is given by
$$
R_U=\frac{3}{4}+\frac{1}{2}\l[\int_{\partial
U}|w|^2d\nu_2(w)-\frac{1}{\pi}\int_{ U}|w|^2dm(w)\r]=\frac{3}{4}+\frac{1}{2}\cdot\frac{(ab)^3}{a^2+b^2}.
$$
Hence the result.
\end{proof}

Note that if we take $a=b$ then we get the Example \ref{ex=cB}. In
the next example we consider cardioid.

\begin{example}\label{ex:cardiod}
Fix $a,b>0$ such that 
$U=\{re^{i\theta}|0\leq r< b(1+2a\cos\theta), 0\le \theta\le 2\pi\} \subseteq \D$. Then the balayage measure $\nu_2$ and the constant $R_U$ are given by
$$
d\nu_2(w)=\frac{b^2}{\pi}(1+a^2+2a\cos \theta)d\theta\;\mbox{ and } R_U=\frac{3}{4}+\frac{b^4}{2}\l((a^2+1)^2-\frac{1}{2}\r).
$$
\end{example}

\noindent {\bf Note: }The cardioid $U$ can be thought of as small
perturbation of disk of radius $b$.

\begin{proof}[Computation for  Example \ref{ex:cardiod}]
We have
\begin{eqnarray*}
\frac{1}{\pi}\int_{U}w^ndm(w)&=&\frac{1}{\pi}\int_{0}^{2\pi}\int_{0}^{b(1+2a\cos{\theta})}r^ne^{in\theta}rdrd\theta\\
&=&\frac{1}{\pi}\int_{0}^{2\pi}\frac{b^{n+2}(1+2a\cos{\theta})^{n+2}}{n+2}e^{in\theta}d\theta\\
&=&\frac{b^{n+2}}{\pi}\int_{0}^{2\pi}\frac{(1+ae^{i\theta}+ae^{-i\theta})^{n+2}}{n+2}e^{in\theta}d\theta\\
&=&\frac{b^{n+2}}{\pi(n+2)}\int_{0}^{2\pi}\sum_{0\leq{u+v}\leq{n+2}}\frac{(n+2)!}{u!v!(n+2-u-v)!}a^{u+v}e^{i(n+u-v)\theta}d\theta\\
&=&\frac{b^{n+2}}{\pi(n+2)}\left(\frac{(n+2)!}{0!n!2!}a^{n}+\frac{(n+2)!}{1!(n+1)!0!}a^{n+2}\right).2{\pi}\\
&=&{b^{n+2}}\cdot a^n(n+1+2a^2).
\end{eqnarray*}
We show that the measure
$d\mu(w)=\frac{b^2}{2\pi}(1+2a^2+ae^{i\theta}+ae^{-i\theta})d\theta$ satisfies the
required condition \eqref{eqn:moment} to be the  balayage measure for cardioid. We have
\begin{eqnarray*}
\int_{\partial U}w^nd\mu(w)&=&\frac{b^{n+2}}{2\pi}\int_{0}^{2\pi}{(1+ae^{i\theta}+ae^{-i\theta})^n}e^{in\theta}(1+2a^2+ae^{i\theta}+ae^{-i\theta})d\theta
\\&=&\frac{b^{n+2}}{2\pi}\l(a^n(1+2a^2)+a\frac{n!}{1!0!(n-1)!}a^{n-1}\r).{2\pi}\\
&=&{b^{n+2}}.a^n(1+2a^2+n)
\\&=&\frac{1}{\pi}\int_{U}w^ndm(w) \;\;\mbox{ for all $n$. }
\end{eqnarray*}
 Therefore $\mu$ satisfies \eqref{eqn:moment} for all $n$ and also has continuous density with respect to arclength of $\partial U$. Therefore, by Remark \ref{relations},  the balayage  measure $\nu_2$ on boundary of $U$ is given by
 $$
d\nu_2(w)=\frac{b^2}{2\pi}(1+2a^2+ae^{i\theta}+ae^{-i\theta})d\theta.
$$
Therefore we have
\begin{eqnarray*}
&&\int_{\partial U}|w|^2d\nu_{2}(w)-\frac{1}{\pi}\int_{U}|w|^2dm(w)
\\&=&\frac{b^4}{2\pi}\int_{0}^{2\pi}{(1+2a\cos{\theta})^2}(1+2a^2+2a\cos{\theta})d\theta-\frac{1}{\pi}\int_{0}^{\pi}\int_{0}^{b(1+2a\cos{\theta})}r^3drd\theta\\
&=&b^4\l((a^2+1)^2-\frac{1}{2}\r).
\end{eqnarray*}
Hence we get the required constant $R_U$ from \eqref{eqn:robinconstant}.
\end{proof}

 In the next few examples, we could not find the balayage measure explicitly, however we calculated the constant $R_U$ explicitly.

\begin{example}\label{ex:triangle}
Fix $0<a<1$. Suppose $U=aT$, where $T$ be triangle with cube roots of unity $1,\omega,\omega^2$ as
vertices. Then the constant is
$$
R_U=\frac{3}{4}+\frac{a^4}{2\pi}\cdot \frac{9\sqrt{3}}{80}.
$$
\end{example}

\begin{proof}[Computation for Example \ref{ex:triangle}]
The region $T$ can be written as
$$
T=\{r(tw^p+(1-t)w^{p+1})|0\leq r < 1, 0\leq t \leq 1, p=0,1,2\}.
$$
Suppose $x+iy=ar(t\omega^p+(1-t)\omega^{p+1})$. Then by change of variables, we have
$$
\frac{1}{\pi}dm(z)=\frac{1}{\pi} dxdy=\frac{1}{\pi}\frac{\sqrt{3}}{2}a^2rdrdt.
$$
Let $d\nu_2(t)$ be the balayage measure on the boundary of triangle $T$. Then from \eqref{eqn:moment}, we get
$$
\int_{\partial U}z^nd\nu_2(z)=\frac{1}{\pi}\int_{ U}z^ndm(z), \;\;\mbox{for all }\; n\geqslant 0.
$$
Which implies for all $n\ge 0$,
\begin{eqnarray*}
&&\int_{0}^{1}(t+(1-t)\omega)^n(1+\omega^n+\omega^{2n})d\nu_2(t)
\\&=&\int_{0}^{1}\int_{0}^{1}r^n(t+(1-t)\omega)^n(1+\omega^n+\omega^{2n})\frac{\sqrt{3}}{2\pi}a^2rdrdt.
\end{eqnarray*}
Since $1+\omega^n+\omega^{2n}=0$ when $n$ is not multiple of $3$. Therefore we get
\begin{eqnarray}\label{eqn:triangle}
\int_{0}^{1}(t+(1-t)\omega)^{3n}d\nu_2(t)= \int_{0}^{1}
(t+(1-t)\omega)^{3n}\frac{\sqrt{3}a^2}{2\pi(3n+2)}dt.
\end{eqnarray}
for all $n\ge 0$. This is the key equation to calculate the balayage measure on $\partial U$.
Solve this equation, we can get the balayage measure on $\partial U$. But we could not solve this equation.

We manage to calculate the constant $R_U$ using \eqref{eqn:triangle}.
By putting $n=1$ and comparing the real parts in both side of \eqref{eqn:triangle}, we have
\begin{eqnarray*}
\int_{0}^1(1-\frac{9}{2}t(1-t))d\nu_2(t)=\int_{0}^1(1-\frac{9}{2}t(1-t))\frac{\sqrt{3}a^2}{10\pi}dt.
\end{eqnarray*}
(As real part of $(t+(1-t)\omega)^{3}$ is $(1-\frac{9}{2}t(1-t))$. By using the fact that $\int_0^1 d\nu_2(t)=\frac{\sqrt{3}a^2}{4\pi}$
and simplifying the last equation we get
\begin{eqnarray}\label{eqn:nu2}
\int_0^1 t(1-t))d\nu_2(t)=\frac{\sqrt{3}a^2}{20\pi}.
\end{eqnarray}

\noindent Therefore we have
\begin{eqnarray*}
&&\int_{\partial U}|z|^2d\nu_2(z)-\frac{1}{\pi}\int_U |z|^2 dm(z)
\\&=&3\l[\int_{0}^1|ar(t+(1-t)\omega)|^2d\nu_2(t)-\int_0^1\int_0^1|ar(t+(1-t)\omega)|^2\frac{\sqrt{3}a^2}{2\pi}rdrdt\r]
\\&=&3a^2\l[\int_{0}^1[1-3t(1-t)]d\nu_2(t)-\frac{\sqrt{3}a^2}{2\pi}\int_0^1\int_0^1[1-3t(1-t)]r^3drdt\r]
\\&=&3a^2\l(\frac{\sqrt{3}a^2}{10\pi}-\frac{\sqrt{3}a^2}{16\pi}\r)\;\hspace{2cm}\mbox{( by \eqref{eqn:nu2} )}
\\&=&\frac{9\sqrt{3}a^4}{80\pi}.
\end{eqnarray*}
Hence we have the required constant $R_U$ from \eqref{eqn:robinconstant}.
\end{proof}

In the next example, we consider semi-disk. We only calculated the constant $R_U$. We were unable to find equilibrium measure.

\begin{example}\label{ex:halfdisk}
Fix $0<a<1$. Suppose $U=\{re^{i\theta}:0<r<a,0<\theta<\pi\}$. Then the constant is
$$
R_U=\frac{3}{4}+\frac{a^4}{2}\l(\frac{1}{2}-\frac{4}{\pi^2}\r).
$$
\end{example}

\begin{proof}[Computation for Example \ref{ex:halfdisk}]
Let $\nu_2=\nu_2'+\nu_2''$ be the balayage measure. Where $\nu_2'$  is the measure on diameter of semicircle ( $\{re^{i\theta }: r\le a, \theta=0,\pi\}$) and $d\nu_2''(z)=g(\theta)d\theta$
is the  measure on circular arc ($\{ae^{i\theta}:0\le \theta \le \pi\}$). Then from \eqref{eqn:moment}, we get
$$
\int_{\partial U}z^nd\nu_2(z)=\frac{1}{\pi}\int_{ U}z^nm(z), \;\;\mbox{for all }\; n\geqslant 0.
$$
Which implies for all $n\ge 0$,
\begin{eqnarray*}
&&\int_{-a}^{a}t^nd\nu_2'(t)+\int_{0}^{\pi}a^ne^{in\theta}g(\theta)d\theta
=\int_{0}^{\pi}\int_{0}^{a}r^ne^{in\theta}rdr\frac{d\theta}{\pi} .
\end{eqnarray*}
 Therefore we get
\begin{eqnarray}\label{eqn:semicircle}
\int_{-a}^{a}t^nd\nu_2'(t)+\int_{0}^{\pi}a^ne^{in\theta}g(\theta)d\theta=
 \left\{\begin{array}{lr}\frac{2ia^{n+2}}{n(n+2)\pi}&\mbox{$n$ is odd}\\o&\mbox{$n$ is even}\end{array}\right.
\end{eqnarray}
 This is the key equation to calculate the balayage measure   on $\partial U$. In principle if we solve this equation then we get the balayage measure on $\partial U$.
 But we could not solve this equation.

However, with out calculating balyage measure we are  able to calculate
the constant $R_U$ using \eqref{eqn:semicircle}. Comparing
imaginary part in both side of \eqref{eqn:semicircle}, we have
\begin{eqnarray}\label{eqn:relation}
\int_{0}^{\pi}\sin{n\theta}g(\theta)d\theta= \left\{\begin{array}{lr}\frac{2a^{2}}{n(n+2)\pi}&\mbox{$n$ is odd}\\o&\mbox{$n$ is even}\end{array}\right.
\end{eqnarray}
Since $g(\theta)$ is defined on $[0,\pi]$, its fourier series can be made to contain only sine terms, and moreover because of its
symmetry with respect to $\frac{\pi}{2}$, $g(\theta)$'s fourier series contains only odd sine terms. Therefore we get
\begin{eqnarray}\label{eqn:series}
g(\theta)=\frac{4a^2}{\pi^2}\sum_{k=1}^{\infty} \frac{\sin{(2k-1)\theta}}{4k^2-1}.
\end{eqnarray}

\noindent Using \eqref{eqn:semicircle}, \eqref{eqn:relation} and \eqref{eqn:series} we get
\begin{eqnarray*}
\int_{\partial U}|z|^2d\nu_2(z)&=& \int_{-a}^{a}t^2d\nu_2'(t)+\int_{0}^{\pi}a^2g(\theta)d\theta
\\&=&a^2\int_{0}^{\pi}(1-\cos{2\theta})g(\theta)d\theta
\\&=&\frac{4a^4}{{\pi}^2}\l[\sum_{k=1}^{\infty}\frac{1}{(4k^2-1)}\l(\frac{2}{2k-1}-\frac{1}{2k+1}-\frac{1}{2k-3}\r)\r]
\\&=&\frac{2a^4}{{\pi}^2}\l(\frac{3\pi^2}{8}-2\r)
\\&=&a^4\l(\frac{3}{4}-\frac{4}{{\pi}^2}\r).
\end{eqnarray*}
Therefore  the constant
$$
R_U=\frac{3}{4}+\frac{1}{2}\left[\int_{\partial U}|z|^2d\nu_2(z)-\frac{1}{\pi}\int_{ U}|z|^2dm(z)\right]=\frac{3}{4}+\frac{a^4}{2}\l(\frac{1}{2}-\frac{4}{{\pi}^2}\r).
$$
Hence the desired result.
\end{proof}
 
\chapter{Determinantal point processes with finitely many points}\label{ch:finite}
 
 In this chapter we calculate the asymptotics of the hole probabilities for $\X_n^{(\alpha)}$ and $\X_n^{(g)}$, defined in Section \ref{sec:exdet}, as $n\to \infty$. Through out this section we assume that the function $g$ satisfies all the conditions given in Section \ref{sec:exdet}.   First we derive few properties of the point process  $\X_n^{(\alpha)}$. We show that if we scale down the points of $\X_n^{\alpha}$ by $n^{-\frac{1}{\alpha}}$, then the expected empirical distributions converge to a compactly supported measure. In other words, the points of $\X_n^{(\alpha)}$ are roughly distributed within disk of radius $n^{\frac{1}{\alpha}}$.

\section{Distribution of the absolute values of points of $\X_{n}^{(\alpha)}$}
In this section we derive the distribution of the absolute values of the points of $\X_{n}^{(\alpha)}$.
Recall that, the vector of points of $\mathcal X_n^{(\alpha)}$ (in uniform random order) has density
 \begin{align}\label{eqn:densitygen}
 \frac{\alpha^n}{n!( 2\pi)^n\prod_{k=0}^{n-1}\Gamma(\frac{2}{\alpha}(k+1))}e^{-\sum_{k=1}^{n}|z_k|^{\alpha}}\prod_{i<j}|z_i-z_j|^2
 \end{align}
with respect to the Lebesgue measure on $\C^n$. The following lemma gives the distribution of the absolute values of the points of $\mathcal X_n^{(\alpha)}$, generalizing of the result of Kostlan (\cite{kostlan}, Theorem 1.1), who proved this for $\alpha=2$ (Ginibre ensemble).

\begin{lemma}\label{lem:genkostlan}
The set of absolute values of the points of $\mathcal X_n^{(\alpha)}$ has the same distribution as $\{R_1,R_2,\ldots,R_n\} $ where $R_k$ are independent and $R_k^{\alpha}\sim \mbox{Gamma}(\frac{2}{\alpha}
 k,1)$.
\end{lemma}

\noindent From Lemma \ref{lem:genkostlan}, we have the following corollary, which will be used in Chapter \ref{ch:infinite}.
\begin{cor}\label{cor:radialdistribution}
The set of absolute values of the points of $\mathcal X_{\infty}^{(\alpha)}$  has the same distribution as $\{R_1,R_2,\ldots \},$ where $R_k^{\alpha}\sim \mbox{Gamma}(\frac{2}{\alpha}k,1)$ and all the $R_k$s are independent.
\end{cor} 

\begin{proof}[Proof of Lemma \ref{lem:genkostlan}]
The Vandermonde determinant is
$$
\prod_{i<j}(z_i-z_j)=\det(z_i^{j-1})_{1\le i,j\le n}=\sum_{\sigma\in S_n}{{\mbox{sgn}}(\sigma)}\prod_{j=1}^{n}z_j^{\sigma(j)-1}.
$$
Set $z_k=r_ke^{i\theta}$ for $k=1,2,\ldots,n$. Then the density of the set of points of $\mathcal X_n^{(\alpha)}$, \eqref{eqn:densitygen}, can be written as 
\begin{align*}
& f_{R_1,\ldots,R_n,\Theta_1,\ldots,\Theta_n}(r_1,\ldots,r_n,\theta_1\ldots,\theta_n)
\\=&C_n.e^{-\sum_{j=1}^{n}r_j^{\alpha}}.\left|\sum_{\sigma\in S_n}{{\mbox{sgn}}(\sigma)}\prod_{j=1}^n\left(r_je^{i\theta_j}\right)^{\sigma(j)-1}\right|^2.r_j
\\=&C_n.e^{-\sum_{j=1}^{n}r_j^{\alpha}}.\sum_{\sigma,\tau\in S_n}
{{\mbox{sgn}}(\sigma){\mbox{sgn}}(\tau)}\prod_{j=1}^nr_j^{(\sigma(j)+\tau(j)-1)}e^{i\theta_j(\sigma(j)-\tau(j))}.
\end{align*}
where $0\le r_j<\infty$ and $0\le \theta_j\le 2\pi$ for $j=1,2,\ldots,n$, and $C_n=\frac{\alpha^n}{n!(2\pi)^n\prod_{k=0}^{n-1}\Gamma(\frac{2}{\alpha}(k+1))}$. Again we have 
$$
\int_{0}^{2\pi}e^{i\theta(\sigma(j)-\tau(j))}d\theta=\left\{\begin{array}{lr}
2\pi & \mbox{if } \sigma(j)=\tau(j),
\\0 & \mbox{otherwise}.
\end{array}\right.
$$
Therefore the density of absolute values of the set of points of $\mathcal X_n^{(\alpha)}$ is 
\begin{align*}
f_{R_1,\ldots,R_n}(r_1,\ldots,r_n)&=C_n.(2\pi)^ne^{-\sum_{k=1}^{n}r_j^{\alpha}}\sum_{\sigma\in S_n}\prod_{j=1}^{n}r_j^{2\sigma(j)-1}
=\frac{1}{n!}\sum_{\sigma\in S_n}\prod_{j=1}^{n}\frac{\alpha}{\Gamma\left(\frac{2}{\alpha}\sigma(j)\right)} e^{-r_j^{\alpha}}r_j^{2\sigma(j)-1}.
\end{align*}
By change of variables formula, the joint density of $(R_1^{\alpha},R_2^{\alpha},\ldots,R_n^{\alpha})$ is
\begin{align*}
f_{R_1^{\alpha},\ldots,R_n^{\alpha}}(r_1,\ldots,r_n)=\frac{1}{n!}\sum_{\sigma\in S_n}\prod_{j=1}^{n}\frac{1}{\Gamma\left(\frac{2}{\alpha}\sigma(j)\right)} e^{-r_j}r_j^{\frac{2}{\alpha}\sigma(j)-1}.
\end{align*}
Therefore the joint density of $\{R_1^{\alpha},R_2^{\alpha},\ldots,R_n^{\alpha}\}$ is
\begin{align*}
\sum_{\sigma\in S_n}\prod_{j=1}^{n}\frac{1}{\Gamma\left(\frac{2}{\alpha}\sigma(j)\right)} e^{-r_j}r_j^{\frac{2}{\alpha}\sigma(j)-1}.
\end{align*}
This implies that $R_1^{\alpha},\ldots,R_n^{\alpha}$ are independent and $R_k^{\alpha}\sim$Gamma$(\frac{2}{\alpha}k,1)$. Hence the result.
\end{proof}

\section{Scaling limit}
In this section we prove a result analogous to the circular law. Let $z_1,z_2,\ldots,z_n$ be a set of points in $\mathcal X_n^{(\alpha)}$. Define the empirical measure   
$$
\rho_n^{(\alpha)}(\cdot)=\frac{1}{n}\sum_{k=1}^{n}\delta_{z_k}(\cdot)=\frac{1}{n}\cdot|\{k\suchthat z_k\in (\cdot), k=1,\ldots,n\}|,
$$
where $|A|$ denotes the number of points in $A$. The following theorem gives the limiting expected distribution of scaled $\X_n^{(\alpha)}$. Let $n^{-\frac{1}{\alpha}}.\mathcal X_n^{(\alpha)}:=\{n^{-\frac{1}{\alpha}}.z\;:\;z\in\mathcal X_n^{(\alpha)}\}$, points of $\X_n^{(\alpha)}$ are scaled by $n^{-\frac{1}{\alpha}}$. 

\begin{theorem}\label{thm:scalelim}
The limiting expected empirical distribution of points of $n^{-\frac{1}{\alpha}}.\mathcal X_n^{(\alpha)}$ is same as  $\mu^{(\alpha)}$, where 
$$
d\mu^{(\alpha)}(re^{i\theta})=\left\{\begin{array}{lr}
\frac{\alpha^2}{4\pi}r^{\alpha-1}drd\theta & \mbox{if  }\;\; 0<r<(\frac{2}{\alpha})^{\frac{1}{\alpha}},0\le \theta \le 2\pi,
\\ 0 & \mbox{otherwise}.
\end{array}\right.
$$
\end{theorem} 

\noindent Note that $\alpha=2$ gives the limiting expected empirical distribution of scaled Ginibre converges to uniform probability measure on the unit disk. The following lemma is the key result to prove Theorem \ref{thm:scalelim}.

\begin{lemma}\label{lem:radial} 
The limiting expected empirical distribution of absolute values of points of $n^{-\frac{1}{\alpha}}.\mathcal X_n^{(\alpha)}$ is same as the distribution of $S$, where $S$ is a positive random variable such that  $S^{\alpha}\sim \mbox{U}[0,\frac{2}{\alpha}]$ (uniform random variable on $[0,\frac{2}{\alpha}]$). 
\end{lemma}

\begin{proof}[Proof of Lemma \ref{lem:radial}]
 Let $\mu_n^{(\alpha)}$ be the expected empirical distribution of the points of $\mathcal X_n^{(\alpha)}$. Then 
\begin{align*}
\mu_n^{(\alpha)}(B)&=\frac{1}{n}\E[|\{k\suchthat z_k\in B, k=1,\ldots,n\}|]
=\frac{1}{n}\int_{B}\mathbb K_n^{(\alpha)}(z,z)dm(z),
\end{align*}
for any Borel set $B\subset \C$. Therefore we have 
$$
d\mu_n^{(\alpha)}(z)=\frac{1}{n}\mathbb K_n^{(\alpha)}(z,z)dm(z), \mbox{ for $z\in \C$}.
$$
Set $z=re^{i\theta}$. The density of $\mu_n^{(\alpha)}$ can be written as 
\begin{align}\label{eqn:densitylim}
f_{R,\Theta}(r,\theta)=\frac{\alpha}{2\pi n}e^{-r^{\alpha}}\sum_{k=0}^{n-1}\frac{r^{2k+1}}{\Gamma(\frac{2}{\alpha}(k+1))},
\end{align}
for $0< r <\infty$ and $0\le \theta\le 2\pi$. Therefore the density of the radial part is
$$
f_R(r)=\frac{\alpha}{n}e^{-r^{\alpha}}\sum_{k=0}^{n-1}\frac{r^{2k+1}}{\Gamma(\frac{2}{\alpha}(k+1))},
$$
for $0< r <\infty$. By change of variables formula, the density of $X_n=\frac{R^{\alpha
}}{n}$ is
$$
f_{X_n}(x)=e^{-nx}\sum_{k=0}^{n-1}\frac{(nx)^{\frac{2}{\alpha}(k+1))-1}}{\Gamma(\frac{2}{\alpha}(k+1))},
$$
for $0< x <\infty$. The moment generating function of $X_n$, for large $n$, is 
\begin{align*}
\E[e^{tX_n}]&=\sum_{k=0}^{n-1}\int_{0}^{\infty}\frac{(nx)^{\frac{2}{\alpha}(k+1))-1}}{\Gamma(\frac{2}{\alpha}(k+1))}\cdot e^{-(n-t)x}dx
\\&=\frac{1}{n-t}\sum_{k=0}^{n-1}\int_{0}^{\infty}\frac{(\frac{n}{n-t}x)^{\frac{2}{\alpha}(k+1))-1}}{\Gamma(\frac{2}{\alpha}(k+1))}\cdot e^{-x}dx
\\&=\frac{1}{n-t}\sum_{k=0}^{n-1}\left(\frac{n}{n-t}\right)^{\frac{2}{\alpha}(k+1)-1}
\\&=\frac{1}{n}\sum_{k=0}^{n-1}\left(1-\frac{t}{n}\right)^{-\frac{2}{\alpha}(k+1)}
\\&=\frac{1}{n}\cdot\frac{\left(1-\frac{t}{n}\right)^{-\frac{2}{\alpha}(n+1)}-\left(1-\frac{t}{n}\right)^{-\frac{2}{\alpha}}}{\left(1-\frac{t}{n}\right)^{-\frac{2}{\alpha}}-1},
\end{align*}
for $t>0$. Therefore, for $t>0$,
$$
\lim_{n\to \infty}\E[e^{tX_n}]=\frac{e^{\frac{2}{\alpha}t}-1}{\frac{2}{\alpha}t}
$$
is the moment generating function of $ S^{\alpha}= U[0,\frac{2}{\alpha}]$. The positive random variables $\frac{R^{\alpha}}{n}$ converges to $S^{\alpha}$ in distribution. Hence the result.
\end{proof}

\begin{proof}[Proof of Theorem \ref{thm:scalelim}]
The random variable $S^{\alpha}$ is uniformly distributed on $[0,\frac{2}{\alpha}]$, then the density of the random variable of $S$ is given by
$$
f_S(s)=\frac{\alpha^2}{2}s^{\alpha-1} \;\; \mbox{ for $0\le s\le \left(\frac{2}{\alpha}\right)^{\frac{1}{\alpha}}$}.
$$
Therefore the result follows from \eqref{eqn:densitylim} and Lemma \ref{lem:radial}.
\end{proof}
 
\section{Hole probabilities for $\X_n^{(g)}$}
In this section we calculate asymptotics of hole probabilities for $\X_n^{(g)}$, defined in Section \ref{sec:exdet}. In particular we get the hole probabilities for $\X_n^{(\alpha)}$. Recall that the set of points of $\mathcal X_n^{(g)}$ (with uniform order) has the density 
\begin{align}\label{eqn:gdensity}
\frac{1}{Z_n^{(g)}}\prod_{i<j}|z_i-z_j|^2e^{-n\sum_{k=1}^{n}g(|z_k|)},
\end{align}
with respect to the Lebesgue measure on $\C^n$, where $Z_n^{(g)}$ is the normalizing constant, i.e.,
\begin{align*}
Z_n^{(g)}=\int_{\C}\cdots \int_{\C}\prod_{i<j}|z_i-z_j|^2e^{-n\sum_{k=1}^{n}g(|z_k|)}\prod_{k=1}^ndm(z_k)
\end{align*}
Next two subsection we give upper bounds and lower bounds for the hole probabilities for $\X_n^{(g)}$.

\subsection{Upper bounds}
 The following theorem gives the upper bound for the hole probabilities.
\begin{theorem}\label{thm:upperbound}
Let $U$ be an open subset of $D(0,T)$, where $Tg'(T)=2$. Then 
\begin{align*}
\limsup_{n\to \infty}\frac{1}{n^2}\log\P[\mathcal X_n^{(g)}(U)=0]\le -R_{U}^{(g)}-\liminf_{n\to \infty}\frac{1}{n^2}\log Z_n^{(g)}.
\end{align*}
\end{theorem}

\begin{proof}[Proof of Theorem \ref{thm:upperbound}]
Let $z_1,z_2,\ldots,z_n$ be the points of $\mathcal X_n^{(g)}$. Then from \eqref{eqn:gdensity} we have
\begin{align}\label{eqn:feket}
\P_n[\mathcal X_n^{(g)}(U)=0]&=\frac{1}{Z_n^{(g)}}\int_{U^c}\ldots \int_{U^c}
e^{-n\sum_{k=1}^ng(|z_k|)}\prod_{i<j}|z_i-z_j|^2\prod_{k=1}^n
dm(z_k)\nonumber
\\&=\frac{1}{Z_n^{(g)}}\int_{U^c}\ldots \int_{U^c}\left\{ \prod_{i<j}|z_i-z_j|\omega(z_i)\omega(z_j)\right\}^2\prod_{k=1}^ne^{-g(|z_k|)} dm(z_k),
\end{align}
where $\omega(z)=e^{-\frac{g(|z|)}{2}}$. Let  $z_1^*,z_2^*,\ldots,z_n^*$ be weighted Fekete points for $U^c$ with weight $\omega(z)$. Therefore we have
$$
\delta_n^{\omega}(U^c)=\left\{
\prod_{i<j}|z_i^*-z_j^*|\omega(z_i^*)\omega(z_j^*)\right\}^{\frac{2}{n(n-1)}}.
$$
Therefore from \eqref{eqn:feket}, we have
\begin{eqnarray*}
\P[\mathcal X_n^{(g)}(U)=0]&\le
&\frac{1}{Z_n^{(g)}}(\delta_n^{\omega}(U^c))^{n(n-1)}\prod_{k=1}^n\l(\int_{U^c}e^{-g(|z_k|)}dm(z_k)\r)
\\&\le &\frac{1}{Z_n^{(g)}}.a^n.(\delta_n^{\omega}(U^c))^{n(n-1)} ,
\end{eqnarray*}
where $a=\int_{U^c}e^{-g(|z|)}dm(z)$. The fourth condition on $g$, in Section \ref{sec:exdet}, implies that $a$ is finite. Hence 
\begin{align*}	
\limsup_{n\to \infty}\frac{1}{n^2}\log\P[\mathcal X_n^{(g)}(U)=0]\le \limsup_{n\to \infty}\frac{1}{n^2}\log (\delta_n^{\omega}(U^c))^{n(n-1)} -\liminf_{n\to \infty}\frac{1}{n^2}\log Z_n^{(g)}.
\end{align*} 
Therefore by \eqref{eqn:limit}, we get
\begin{align*}\label{inequality}
\limsup_{n\to \infty}\frac{1}{n^2}\log \P_n[\mathcal X_n^{(g)}(U)=0] &\le -\inf_{\mu\in \mathcal {P}(\C \backslash U)}R_{\mu}^{(g)}-\liminf_{n\to \infty}\frac{1}{n^2}\log Z_n^{(g)} 
\\&= -R_U^{(g)}-\liminf_{n\to \infty}\frac{1}{n^2}\log Z_n^{(g)}.
\end{align*}
Hence the upper bound.
 \end{proof}
 
\noindent Note that, by the same arguments it can be shown that 
 \begin{align*}
  Z_n^{(g)}\le a^n(\delta_n^{\omega}(\C))^{n(n-1)}, 
\end{align*} 
where $a=\int_{\C}e^{-g(|z|)}dm(z)$. Therefore by \eqref{eqn:limit}, we have 
\begin{equation}\label{limsup}
\limsup_{n\to \infty}\frac{1}{n^2}\log Z_n^{(g)}\le -\inf_{\mu\in \mathcal {P}(\C)}R_{\mu}^{(g)} =-R_{\emptyset}^{(g)}.
\end{equation}

\subsection{Lower bounds}
In this section we give the lower bounds of hole probabilities for $\mathcal X_n^{(g)}$ for two class of open sets $U$. Let $\mu_n:=\mu\big|_{U_n}$, restriction of $\mu$ (as in Theorem \ref{thm:diskgeneral}) on $U_n$. 

\begin{theorem}\label{thm:gholeprobability12}
Let $U\subset D(0,T)$ be an open set, where $Tg'(T)=2,$  such that 
there exists a sequence of open sets ${U_n}$ such that $\bar
U \subset U_n  \subseteq D(0,T)$ with $ U_{n+1}\subseteq U_n$ for all $n$ and $\bigcap U_n=\overline{U}$ and the balayage measure $\nu_n$
on $\partial U_n$ converges weakly to the balayage measure $\nu$ on
$\partial U$ with respect to $\mu_n$.
Then
$$
\liminf_{n\to
\infty}\frac{1}{n^2}\log\P[\X_n^{(g)}(U)=0]\ge -R_U^{(g)}-\limsup_{n\to \infty}\frac{1}{n^2}\log Z_n^{(g)}.
$$
\end{theorem}

\begin{remark} Examples of such open sets: 
\begin{enumerate}
\item It follows from Example \ref{ex:gdisk} and Example \ref{ex:gannulus} that disk and annulus satisfy the conditions of Theorem \ref{thm:gholeprobability12}.

\item If $g(t)=t^{\alpha}$. Then convex open sets in $D(0,(\frac{2}{\alpha})^{\frac{1}{\alpha}})$, which do not intersect unit circle, satisfy the conditions of Theorem \ref{thm:gholeprobability12}. More generally  note that if $U$ is
an open set containing origin such that $\bar U \subset aU$ for
all $a>1$, then the balayage measure $\nu_a$ on $\partial (aU)$ is
given in terms of the balayage measure $\nu$ on $\partial U$ as
 $\nu_a (B)=\frac{1}{a^{\alpha}}\nu(\frac{1}{a}B)$ for any measurable set $B\subset \C$. Therefore $\nu_a$ converges weakly to $\nu$ as $a\to 1$.

\item In particular for $\alpha=2$, the translations of the above sets  also
 satisfy the conditions  of Theorem \ref{thm:gholeprobability12}. So all the examples we
 considered in Section \ref{results} satisfy the conditions  of Theorem \ref{thm:gholeprobability12}.
 \end{enumerate}
\end{remark}

\noindent The following lemma will be used in the proof of Theorem \ref{thm:gholeprobability12}.
\begin{lemma}\label{thm:gholeprobability}
Let $U\subset D(0,T)$ be an open set, where $Tg'(T)=2$. Then
$$
\liminf_{n\to
\infty}\frac{1}{n^2}\log \P[\X_n^{(g)}( U)=0]\ge-\inf_{\mu\in
\mathcal A}R_{\mu}^{(g)}-\limsup_{n\to \infty}\frac{1}{n^2}\log Z_n^{(g)},
$$
where 
$\mathcal A=\{\mu\in \mathcal P(\mathbb C): \mbox {dist(Supp($\mu$),$\bar
U$)$>0$}\}$.
\end{lemma}

\noindent Assuming the lemma we proceed to prove Theorem \ref{thm:gholeprobability12}.

\begin{proof}[Proof of Theorem \ref{thm:gholeprobability12}]
Let $U,U_1,U_2,\ldots $ be open subsets of $D(0,T)$ satisfying conditions
of Theorem \ref{thm:gholeprobability12}. By Lemma
\ref{thm:gholeprobability}, we have
\begin{eqnarray*}
 \liminf_{n\to \infty}\frac{1}{n^2}\log\P[\X_n^{(g)}(U)=0]&\ge & -\inf_{\mu\in
\mathcal A}R_{\mu}^{(g)}-\limsup_{n\to \infty}\frac{1}{n^2}\log Z_n^{(g)},
\\& \ge &-\inf_{\mu\in \mathcal A_m}R_{\mu}^{(g)}-\limsup_{n\to \infty}\frac{1}{n^2}\log Z_n^{(g)},
\end{eqnarray*}
where
$\mathcal A=\{\mu\in \mathcal P(\mathbb C): \mbox {dist(Supp($\mu$),$\bar
U$)$>0$\}},$ $\mathcal A_m=\{\mu\in \mathcal P(\mathbb C):\mu(U_m)=0\}$. The last inequality follows from the facts that $\bar{U} \subset U_m$ and
  $\mathcal A_m\subset \mathcal A $. Again from Theorem
\ref{thm:generalsetting} we  have 
$$
R_{U_m}^{(g)}=R_{\emptyset}^{(g)}+\frac{1}{2}\left[\int_{\partial U_m}g(|z|)d\nu_m(z)-\int_{U_m}g(|z|)d\mu_m(z)\right],
$$
where $\nu_m$ is the balayage on $\partial U_m$ with respect to the measure $\mu_m$, restriction of $\mu$ (as in Theorem \ref{thm:diskgeneral}) on $U_m$. Since $U_m$ are monotone and converge to $U$, $\mu \big|_{U_m}$ converge weakly to $\mu\big|_{U}$. Again the  balayage measures $\nu_m$ converge weakly to the balayage measure $\nu$,  therefore $R_{U_m}^{(g)}$ converges to
$R_{U}^{(g)}$ as $m\to\infty$. Therefore we
have
\begin{eqnarray*}
\liminf_{n\to
\infty}\frac{1}{n^2}\log\P[\X_n^{(g)}(U)=0]
\ge -R_U^{(g)}-\limsup_{n\to \infty}\frac{1}{n^2}\log Z_n^{(g)}.
\end{eqnarray*}
Hence the result.
\end{proof}

\noindent It remains to prove Lemma \ref{thm:gholeprobability}.

\begin{proof}[Proof of Lemma \ref{thm:gholeprobability}]
From \eqref{eqn:gdensity}, the density of the set of points of $\X_n^{(g)}$, we have
\begin{eqnarray}\label{eqn:lower}
\P[\X_n^{(g)}(U)=0]&=&\frac{1}{Z_n^{(g)}}\int_{U^c}\ldots \int_{U^c}
e^{-n\sum_{k=1}^ng(|z_k|)}\prod_{i<j}|z_i-z_j|^2\prod_{k=1}^n
dm(z_k)\nonumber
\\&\ge &\frac{1}{Z_n^{(g)}}\int_{U^c}\ldots \int_{U^c}
e^{-n\sum_{k=1}^ng(|z_k|)}\prod_{i<j}|z_i-z_j|^2\prod_{k=1}^n\frac{f(z_k)}{M}
dm(z_k),\nonumber
\end{eqnarray}
where $f$ is a compactly supported  probability density function with support in  $U^c$ and uniformly bounded by $M$. Applying logarithm on both sides we have
\begin{eqnarray*}
&&\log \P[\X_n^{(g)}(U)=0]
\\&\ge & -\log (Z_n^{(g)} .M^n) + \log \l(\int_{U^c}\ldots \int_{U^c}
e^{-n\sum_{k=1}^ng(|z_k|)}\prod_{i<j}|z_i-z_j|^2\prod_{k=1}^n f(z_k)
dm(z_k)\r)
\\&\ge& -\log (Z_n^{(g)} .M^n) +  \int_{U^c}\ldots \int_{U^c}\log\l(
e^{-n\sum_{k=1}^ng(|z_k|)}\prod_{i<j}|z_i-z_j|^2\r)\prod_{k=1}^n f(z_k)
dm(z_k)
\\&=& -\log (Z_n^{(g)} .M^n)+n(n-1)\int_{U^c} \int_{U^c} (\log |z_1-z_2|-\frac{n}{n-1}g(|z_1|))\prod_{k=1}^2 f(z_k)dm(z_k),
\end{eqnarray*}
where the second inequality follows from Jensen's inequality.
Therefore by taking limits on both sides, we have
\begin{eqnarray}\label{eqn:bddsupport}
\liminf_{n\to \infty}\frac{1}{n^2}\log \P[\X_n^{(g)}(U)=0]&\ge &
-\limsup_{n\to \infty}\frac{1}{n^2}\log Z_n^{(g)} - R_{\mu}^{(g)}
\end{eqnarray}
for any  probability measure $\mu$ with density bounded and
compactly supported on $U^c$.

Let $\mu$ be probability measure with density $f$ compactly
supported on $U^c$.  Consider the sequence of measures with
bounded densities 
$$
d{\mu}_M(z)= \frac{f_M(z) dm(z)}{\int f_M(w)dm(w) },
$$ 
where $f_M(z)=\min\{f(z),M\}$. From monotone convergence theorem for positive and negative parts of logarithm, it follows (as positive part of logarithm is bounded) that
$$
\lim_{M \to \infty} \int_{U^c} \int_{U^c} \log|z_1-z_2|\prod_{i=1}^2 f_M(z_i)dm(z_i) = \int_{U^c} \int_{U^c} \log|z_1-z_2|\prod_{i=1}^2 f(z_i)dm(z_i).
$$ 
From monotone convergence theorem, it follows that $\lim_{M \to \infty} \int
f_M(w)dm(w)=1$ and since $g$ is continuous function, $\lim_{M \to \infty} \int
g(|z|)f_M(z)dm(w)=\int g(|z|)f(z)dm(w)$. Therefore 
$$
\lim_{M \to \infty} R_{\mu_{M}}^{(g)} = R_{\mu}^{(g)}.
$$ 
So \eqref{eqn:bddsupport} is true for any measure with density compactly supported on $U^c$.

Let $\mu$ be a probability measure with compact support at a
distance of at least $\delta$ from $U$. Then the convolution
$\mu*\sigma_{\epsilon} $,  where $\sigma_{\epsilon}$ is  uniform
probability measure  on  disk of radius $\epsilon$ around origin,
has density compactly supported in $U^c$, if $\epsilon$ is less
than $\delta$. We have
\begin{eqnarray*}
I_{\mu*\sigma_{\epsilon}}&=&\iint \log|z-w|d(\mu*\sigma_{\epsilon})(z)d(\mu*\sigma_{\epsilon})(w)
\\&=&\iint\iint\iint \log|z+\epsilon r_1 e^{i \theta_1}-w-\epsilon r_2 e^{i \theta_2}|\frac{r_1dr_1d\theta_1}{\pi}\frac{r_2dr_2d\theta_2}{\pi}d\mu(z)d\mu(w)
\\&&\mbox{(limits of $r_1,r_2$ are from $0$ to 1 and $\theta_1,\theta_2$ are from $0$ to $2\pi$ )}
\\&\ge &\iint \log |z-w| d\mu(z)d\mu(w),
\end{eqnarray*}
where the inequality follows from the repeated application of the mean value property of the subharmonic function $\log |z|$. And also we have
\begin{eqnarray*}
I_{\mu*\sigma_{\epsilon}}\le \iint \log [|z-w| +2\epsilon]d\mu(z)d\mu(w).
\end{eqnarray*}
Therefore, $\lim_{\epsilon\to 0}I_{\mu*\sigma_{\epsilon}}=I_{\mu} $ and hence $\lim_{\epsilon\to 0}R_{\mu*\sigma_{\epsilon}}^{(g)}=R_{\mu}^{(g)}.$

So, \eqref{eqn:bddsupport} is true for any probability measure with compact support
whose distance from $U$ is positive. Hence the required lower bound.
\end{proof}

\noindent Note that by the same arguments it can be shown that 
\begin{align}\label{eqn:liminf}
\liminf_{n\to \infty}\frac{1}{n^2}\log Z_n^{(g)}\ge -R_{\mu}^{(g)},
\end{align}
for all compactly supported probability measure $\mu$ on $\C$.

\begin{cor}\label{corollary}
Let $Z_n^{(g)}$  be the normalizing constant, as in \eqref{eqn:gdensity}. Then  
\begin{align*}
\lim_{n\to \infty}\frac{1}{n^2}\log Z_n^{(g)}= -R_{\emptyset}^{(g)},
\end{align*}
where $\emptyset$ denotes the empty set.
\end{cor}

\begin{proof}[Proof of Corollary \ref{corollary}]
The proof is directly follows from \eqref{limsup} and \eqref{eqn:liminf}.
\end{proof}

The second class of sets $U$ we consider satisfy the exterior ball condition, i.e.,
there exists $\epsilon > 0$  such that for every $z\in \partial U$ there exists a 
$\eta\in U^c$ such that
\begin{eqnarray}\label{eqn:gcondition}
U^c\supset B(\eta,\epsilon)\;\; \mbox{and
}\; |z-\eta |= \epsilon.
\end{eqnarray}
 Note that all convex domains satisfy the
condition  \eqref{eqn:gcondition} wth any $\epsilon>0$. Annulus is not a convex domain but
it satisfies the condition \eqref{eqn:gcondition}. The following theorem gives hole probabilities for such open sets.

\begin{theorem}\label{thm:gholeprobability1}
Let $g'$ be bounded in $[0,T+1]$ and $U\subseteq D(0,T)$ be an open set satisfying  condition
\eqref{eqn:gcondition}. Then
$$
\lim_{n\to \infty}\frac{1}{n^2}\log\P[\X_n^{(g)}(U)=0]\ge -R_U^{(g)}-\limsup_{n\to \infty}\frac{1}{n^2}\log Z_n^{(g)}.
$$
\end{theorem}

Note that $g(r)=r^{\alpha}$ satisfies the condition of Theorem \ref{thm:gholeprobability1} for $\alpha\ge 1$, not for $\alpha<1$.
The above theorem does not include the cases of the cardioid or sectors with an obtuse angle at center. Theorem \ref{thm:gholeprobability12} takes care of these and all the other sets $U$ which can contain scaled-down copies of themselves. But Theorem \ref{thm:gholeprobability12}, unlike Theorem \ref{thm:gholeprobability1}, requires the boundary of $U$ to not intersect the boundary of $D(0,T)$.  The proof of Theorem \ref{thm:gholeprobability1} makes use of Fekete points, whereas that of Theorem  \ref{thm:gholeprobability12} makes use of the balayage measure.  The following lemma, which is used in the proof of Theorem \ref{thm:gholeprobability1}, provides separation between weighted Fekete points. This is not tightest separation result but suffices for our purpose. The separation of Fekete points has been studied by many authors, e.g., see \cite{ortega}, \cite{bos} and references therein.

\begin{lemma}\label{lem:feketedistance}
Let $g'$ be bounded in $[0,T+1]$ and $U\subseteq D(0,T)$ be an open set satisfying  condition
\eqref{eqn:gcondition}. If  $z_1^*,z_2^*,\ldots,z_n^*$ are the weighted
Fekete points for $U^c$ with weight $\omega(z)=e^{-\frac{g(|z|)}{2}}$, then for large $n$,
\begin{eqnarray*}\label{eqn:feketedistance}
\min\{|z_i^*-z_k^*|: 1\le i\neq k \le n\}\ge C.\frac{1}{n^3}
\end{eqnarray*}
for some constant $C$ (does not depend on $n$).
\end{lemma}

\noindent Assuming Lemma \ref{lem:feketedistance} we proceed to
prove  Theorem \ref{thm:gholeprobability1}.
\begin{proof}[Proof of Theorem \ref{thm:gholeprobability1}]
 Let
$z_1^*,z_2^*,\ldots,z_n^*$ be weighted Fekete points for $U^c$ with the weight function $\omega(z)=e^{-\frac{g(|z|)}{2}}$. Since the support of the Fekete points is contained in support of equilibrium measure (see
\cite{totikbook}, Chapter III Theorem 2.8), it follows that $|z_{\ell}^*|\le T$ for $\ell=1,2,\ldots, n$. Let $B_\ell=U^c\cap B(z_{\ell}^*,\frac{C}{n^4})$ for
$\ell=1,2,\ldots, n$. Then, for large $n$, we have
\begin{eqnarray*}\label{eqn:lower}
\P[\X_n^{(g)}(U)=0] &=&\frac{1}{Z_n^{(g)}}\int_{U^c}\ldots
\int_{U^c}
e^{-n\sum_{k=1}^ng(|z_k|)}\prod_{i<j}|z_i-z_j|^2\prod_{k=1}^n
dm(z_k)
\\&\ge &\frac{1}{Z_n^{(g)}}\int_{B_1}\ldots \int_{B_n}
\left\{
\prod_{i<j}|z_i-z_j|\omega(z_i)\omega(z_j)\right\}^2\prod_{k=1}^ne^{-g(|z_k|)}
dm(z_k),\nonumber
\\&\ge &\frac{e^{-g(T+1)n}}{Z_n^{(g)}}\int_{B_1}\ldots \int_{B_n}
\left\{ \prod_{i<j}|z_i-z_j|\omega(z_i)\omega(z_j)\right\}^2\prod_{k=1}^n dm(z_k).
\end{eqnarray*}
 By
Lemma \ref{lem:feketedistance}, for large $n$, we have $|z_i^*-z_j^*|\ge
\frac{C}{n^3}$ for $i\neq j$, for some constant $C$ independent of
$n$. Suppose $z_i\in B(z_i^*,\frac{C}{n^4})$ and $z_j\in
B(z_j^*,\frac{C}{n^4})$ for $i\neq j$, then for large $n$ 
\begin{eqnarray*}
|z_i-z_j|\ge |z_i^*-z_j^*|-\frac{2C}{n^4}\ge
|z_i^*-z_j^*|-\frac{2}{n}\cdot |z_i^*-z_j^*| \ge
|z_i^*-z_j^*|\l(1-\frac{2}{n}\r).
\end{eqnarray*}
Therefore we have
\begin{eqnarray*}
\P[\X_n^{(g)}(U)=0] &\ge & \frac{e^{-g(T+1)n}}{Z_n^{(g)}}\int_{B_1}\ldots
\int_{B_n}\left\{
\prod_{i<j}|z_i^*-z_j^*|\l(1-\frac{2}{n}\r)\omega(z_i)\omega(z_j)\right\}^2
\prod_{k=1}^n dm(z_k) 
\end{eqnarray*}
Since $g'$ is bounded on $[0,T+1]$, there exists a constant $K$ such that $|g(|z|)-g(|w|)|\le K.|z-w|$ for all $z,w\in D(0,T+1)$. Therefore for large $n$,
$$
e^{-\frac{1}{2}g(|z_i|)}\ge
e^{-\frac{1}{2}g(|z_i^*|)}.e^{-\frac{C'}{n^4}},
$$
for $z_i\in B(z_i^*,\frac{C}{n^4}), i=1,2,\ldots,n$, where $C'=C.K/2$. Hence for large $n$, we have
\begin{align*}
\P[\X_n^{(g)}(U)=0] \ge \frac{e^{-g(T+1)n}}{Z_n^{(g)}}\l(1-\frac{2}{n}\r)^{{n(n-1)}}.e^{-\frac{C'}{n^2}}.\left\{
\prod_{i<j}|z_i^*-z_j^*|\omega(z_i^*)\omega(z_j^*)\right\}^2\prod_{k=1}^n\int_{B_k}
dm(z_k)
\end{align*}
For large $n$, we have
$\int_{B_i} dm(z_i)\ge \pi.(\frac{ C}{2n^4})^2,
i=1,2,\ldots,n$ (condition \eqref{eqn:gcondition} implies that $B_i$ contains atleast a ball of radius $\frac{ C}{2n^4}$). Hence we have
$$
\P[\X_n^{(g)}(U)=0]\ge
\frac{e^{-g(T+1)n}}{Z_n^{(g)}}\l(1-\frac{2}{n}\r)^{{n(n-1)}}.e^{-\frac{C'}{n^2}}.(\delta_n^{\omega}(U^c))^{n(n-1)}.\l(\pi.\l(\frac{
C}{2n^4}\r)^2\r)^n.
$$
Therefore by \eqref{eqn:limit}, we have
\begin{align*}
\liminf_{n\to \infty}\frac{1}{n^2}\log\P[\X_n^{(g)}(U)=0]&\ge
-\inf_{\mu\in \mathcal {P}(\C \backslash U)}R_{\mu}^{(g)}-\limsup_{n\to \infty}\frac{1}{n^2}\log Z_n^{(g)} 
\\&= -R_U^{(g)}-\limsup_{n\to \infty}\frac{1}{n^2}\log Z_n^{(g)}.
\end{align*}
 Hence the result.
\end{proof}

It remains to prove  Lemma \ref{lem:feketedistance}.

\begin{proof}[Proof of Lemma \ref{lem:feketedistance}]
Let $P(z)=(z-z_2^*)\cdots (z-z_n^*)$. Now we show that
\begin{eqnarray*}\label{eqn:feketedist}
\min\{|z_1^*-z_k^*|: 2\le k \le n\}\ge C.\frac{1}{n^3}
\end{eqnarray*}
for some constant $C$. Suppose $|z_1^*-z_2^*|\le \frac{1}{n^2}$.
By Cauchy integral formula we have
\begin{eqnarray*}
|P(z_1^*)|&=&|P(z_1^*)-P(z_2^*)|
\\&=&\l|\frac{1}{2\pi i}\int_{|\zeta-z_1^*|=\frac{2}{n^2}}\frac{P(\zeta)}{(\zeta-z_1^*)}d\zeta-
\frac{1}{2\pi
i}\int_{|\zeta-z_1^*|=\frac{2}{n^2}}\frac{P(\zeta)}{(\zeta-z_2^*)}d\zeta\r|
\\&\le&
\frac{1}{2\pi}\int_{|\zeta-z_1^*|=\frac{2}{n^2}}\frac{|P(\zeta)||z_1^*-z_2^*|}{|\zeta-z_1^*||\zeta-z_2^*|}|d\zeta|
\\&\le
&\frac{1}{2\pi}.|P(\zeta^*)|.\frac{n^2}{2}.n^2.|z_1^*-z_2^*|.2\pi.\frac{2}{n^2},\;\;\;\;\mbox{(as $|\zeta-z_2^*|\ge \frac{1}{n^2}$)}
\end{eqnarray*}
where $\zeta^*\in \{\zeta:|\zeta-z_1^*|=\frac{2}{n^2}\}$ such that
$P(\zeta^*)=\sup\{|P(\zeta)|:|z_1^*-\zeta|=\frac{2}{n^2}\}$. Therefore we have
\begin{eqnarray}\label{eqn:z1}
|P(z_1^*)|\le n^2.|z_1^*-z_2^*|.|P(\zeta^*)|.
\end{eqnarray}
Since $g'$ is bounded on $[0,T]$, there exists a constant $K$ such that $|g(|z|)-g(|w|)|\le K.|z-w|$ for all $z,w\in D(0,T)$. Therefore, if $z,w\in D(0,T)$ and $|z-w|\le \frac{2}{n}$, then
\begin{eqnarray}\label{eqn:exp}
e^{-(n-1)\frac{g(|z|)}{2}}\le C_1.e^{-(n-1)\frac{g(|w|)}{2}},
\end{eqnarray}
where $C_1$ is a constant. Indeed, if $z,w\in D(0,T)$ and $|z-w|\le \frac{2}{n}$, we have
\begin{eqnarray*}
e^{-\frac{(n-1)}{2}(g(|z|)-g(|w|))}\le e^{\frac{(n-1)}{2}K.|z-w|}\le e^{\frac{(n-1)}{2}.K.\frac{2}{n}}= e^{K}.
\end{eqnarray*}

\noindent{\bf Case I:} Suppose $\zeta^*\in U^c$. Since  $z_1^*,z_2^*,\ldots,z_n^*$ are the Fekete points for $U^c$ with the weight function $\omega(z)=e^{-\frac{g(|z|)}{2}}$, then 
$$
|P(\zeta^*)|.e^{-(n-1)\frac{g(|\zeta^*|)}{2}}\le |P(z_1^*)|.e^{-(n-1)\frac{g(|z_1^*|)}{2}}.
$$
Then from \eqref{eqn:z1} and \eqref{eqn:exp} we get
\begin{eqnarray*}
|P(z_1^*)|e^{-(n-1)\frac{g(|z_1^*|)}{2}}&\le & n^2.
|z_1^*-z_2^*|.|P(\zeta^*)|.C_1.e^{-(n-1)\frac{g(|\zeta^*|)}{2}}
\\&\le& C_1.n^2. |z_1^*-z_2^*|.|P(z_1^*)|.e^{-(n-1)\frac{g(|z_1^*|)}{2}}.
\end{eqnarray*}
 And hence we get
\begin{eqnarray*}
|z_1^*-z_2^*|\ge \frac{1}{C_1.n^2}.
\end{eqnarray*}

\noindent{\bf Case II:} Suppose $\zeta^*\in U$. Therefore dist$(z_1^*,\partial U)=\inf \{|z-z_1^*|\; :\;z\in \partial U\}<\frac{2}{n^2}$. Choose large $n$ such that $\frac{1}{n}<\epsilon $. From the given
condition \eqref{eqn:gcondition} on $U$, we can choose $\eta\in U^c$ such that  $z_1^*\in \bar{B}(\eta,\frac{1}{n}) \subseteq U^c $. By taking the power series expansion of $P$ around $\eta$ and by triangle inequality, we
get
\begin{eqnarray}\label{eqn:zeta1}
\;\;\;\;|P(\zeta^*)|\le
|P(\eta)|+|\zeta^*-\eta|.\l|\frac{P^{(1)}(\eta)}{1!}\r|+\cdots+|\zeta^*-\eta|^{(n-1)}.\l|\frac{P^{(n-1)}(\eta)}{(n-1)!}\r|,
\end{eqnarray}
where $P^{(r)}(\cdot)$ denotes the $r$-th derivative of $P$. From
the Cauchy integral formula we have
$$
\l|\frac{P^{(r)}(\eta)}{r!}\r|\le
\frac{1}{2\pi}\int_{|z-\eta|=\frac{1}{n}}\frac{|P(z)|}{|z-\eta|^{r+1}}|dz|\le
|P(\eta^*)|.n^r,
$$
where $\eta^*\in \{z:|z-\eta|=\frac{1}{n}\}$ such that
$P(\eta^*)=\sup\{|P(z)|:|z-\eta|=\frac{1}{n}\}$. Note that
$|\zeta^*-\eta|\le |\zeta^*-z_1^*|+|z_1^*-\eta|\le
\frac{2}{n^2}+\frac{1}{n}$, therefore we have
$$
|\zeta^*-\eta|^{r}.\frac{|P^{(r)}(\eta)|}{r!}\le \l(1+\frac{2}{ n}\r)^r.|P(\eta^*)|\le e^{2}.|P(\eta^*)|,
$$
for $r=1,2,\ldots, n-1$. Therefore from \eqref{eqn:zeta1} we get
$$
|P(\zeta^*)|\le |P(\eta)|+e^{2}.n.|P(\eta^*)|.
$$
And hence from \eqref{eqn:z1} and \eqref{eqn:exp} we have
\begin{eqnarray*}
&&|P(z_1^*)|e^{-(n-1)\frac{g(|z_1^*|)}{2}}
\\&\le& n^2.|z_1^*-z_2^*|.C_1.\l(|P(\eta)|e^{-(n-1)\frac{g(|\eta|)}{2}}+n.e^{2}.|P(\eta^*)|e^{-(n-1)\frac{g(|\eta^*|)}{2}}\r)
\\&\le& n^2.|z_1^*-z_2^*|.C_1.\l(1+n.e^{2}\r)|P(z_1^*)|e^{-(n-1)\frac{g(|z_1^*|)}{2}},
\end{eqnarray*}
since $z_1^*,z_2^*,\ldots,z_n^*$ are the Fekete points for $U^c$ with weight $e^{-(n-1)\frac{g(|z|)}{2}}$ and $\eta,\eta^*\in U^c$. Therefore we get
\begin{eqnarray*}
|z_1^*-z_2^*|\ge \frac{1}{2.C_1.e^{2}n^3}.
\end{eqnarray*}
By Case I and Case II we get that if $|z_1^*-z_2^*|\le
\frac{1}{n^2}$, then $|z_1^*-z_2^*|\ge \frac{1}{2.C_1.e^{2}n^3}$. Similarly,
if $|z_1^*-z_k^*|\le \frac{1}{n^2}$ for $k=2,3,\ldots,n$, then
$|z_1^*-z_k^*|\ge \frac{1}{2.C_1.e^{2}n^3}$. Therefore we have
$$
\min\{|z_1^*-z_k^*|: k=2,3,\ldots,n\}\ge \frac{1}{2.C_1.e^{2}n^3}.
$$
Similarly it can be shown that $|z_{\ell}^*-z_k^*|\ge
\frac{1}{2.C_1.e^{2}n^3}$ for all $1\le \ell\neq k \le n$ and hence
$$
\min\{|z_{\ell}^*-z_k^*|:1\le  \ell\neq k \le n \}\ge
\frac{1}{2.C_1.e^{2}n^3}.
$$
Hence the result.
\end{proof}

\section{Hole probabilities for $\X_n^{(\alpha)}$}
In this section we derive the hole probabilities for $\X_n^{(\alpha)}$ from Theorem \ref{thm:upperbound}, Theorem \ref{thm:gholeprobability12}, Theorem \ref{thm:gholeprobability1} and Corollary \ref{corollary}. Recall, the notations: 
 \begin{align*}
 R_U^{(\alpha)}:=\inf_{\mu\in \mathcal P(U^c)}R_{\mu}^{(\alpha)}, \;\;\;\mbox{ where $R_{\mu}^{(\alpha)}=\iint\log \frac{1}{|z-w|}d\mu(z)d\mu(w)+\int |z|^{\alpha}d\mu(z)$}.
 \end{align*}
 Note that $R_U^{(\alpha)}$ is a abuse of the notation $R_U^{(g)}$. The first result is analogous to Theorem \ref{thm:gholeprobability12}.
\begin{theorem}\label{thm:allalpha}
Let $U$ be an open subset of $D(0,(\frac{2}{\alpha})^{\frac{1}{\alpha}})$ satisfying the conditions of Theorem \ref{thm:gholeprobability12}. Then 
\begin{align*}
\lim_{n\to \infty}\frac{1}{n^2}\log\P[\X_n^{(\alpha)}(n^{\frac{1}{\alpha}}U)=0]=R_{\emptyset}^{(\alpha)}-R_{U}^{(\alpha)},
\end{align*}
for all $\alpha>0$.
\end{theorem}

\begin{proof}[Proof of Theorem \ref{thm:allalpha}]
The density of the points of $\X_n^{(\alpha)}$ gives 
\begin{align}
&\P[\X_n^{(\alpha)}(n^{\frac{1}{\alpha}}U)=0]\nonumber
\\=&\frac{\alpha^n}{n!( 2\pi)^n\prod_{k=0}^{n-1}\Gamma(\frac{2}{\alpha}(k+1))}\int_{(n^{\frac{1}{\alpha}}U)^c}\cdots \int_{(n^{\frac{1}{\alpha}}U)^c}e^{-\sum_{k=1}^{n}|z_k|^{\alpha}}\prod_{i<j}|z_i-z_j|^2\prod_{k=1}^{n}dm(z_k)\nonumber
\\=&\frac{1}{Z_n^{(\alpha)}}\int_{U^c}\cdots \int_{U^c}e^{-n\sum_{k=1}^{n}|z_k|^{\alpha}}\prod_{i<j}|z_i-z_j|^2\prod_{k=1}^{n}dm(z_k)\nonumber
\end{align} 
where  the normalizing constant 
$$
Z_n^{(\alpha)}=\int_{\C}\cdots \int_{\C}e^{-n\sum_{k=1}^{n}|z_k|^{\alpha}}\prod_{i<j}|z_i-z_j|^2\prod_{k=1}^{n}dm(z_k)=\left(\frac{n^{\frac{1}{\alpha}\cdot n(n+1)}\alpha^n}{n!( 2\pi)^n\prod_{k=0}^{n-1}\Gamma(\frac{2}{\alpha}(k+1))}\right)^{-1}.
$$
Therefore we have 
\begin{align}\label{eqn:galpha}
\P[\X_n^{(\alpha)}(n^{\frac{1}{\alpha}}U)=0]=\P[\X_n^{(g)}(U)=0],
\end{align}
 with $g(t)=t^{\alpha}$. The function $g(t)=t^{\alpha}$ gives $T=(\frac{2}{\alpha})^{\frac{1}{\alpha}}$, the solution of $tg'(t)=2$. Again $U$ satisfies  the condition of Theorem \ref{thm:gholeprobability12}. Therefore by Theorem \ref{thm:upperbound}, Theorem \ref{thm:gholeprobability12} and Corollary \ref{corollary}, we have
\begin{align*}
\lim_{r\to \infty}\frac{1}{n^2}\log\P[\X_n^{(\alpha)}(n^{\frac{1}{\alpha}}U)=0]=R_{\emptyset}^{(\alpha)}-R_{U}^{(\alpha)},
\end{align*}
for all $\alpha>0$.
\end{proof}
Next result is analogous to Theorem \ref{thm:gholeprobability1}.
\begin{theorem}\label{thm:alphagreaterthan1}
Let $U$ be an open subset of $D(0,(\frac{2}{\alpha})^{\frac{1}{\alpha}})$ satisfying the condition \eqref{eqn:gcondition}. Then for all $\alpha\ge 1$, 
\begin{align*}
\lim_{n\to \infty}\frac{1}{n^2}\log\P[\X_n^{(\alpha)}(n^{\frac{1}{\alpha}}U)=0]=R_{\emptyset}^{(\alpha)}-R_{U}^{(\alpha)}.
\end{align*}
\end{theorem}

\begin{proof}[Proof of Theorem \ref{thm:alphagreaterthan1}]
From \eqref{eqn:galpha}, we have 
\begin{align*}
\P[\X_n^{(\alpha)}(n^{\frac{1}{\alpha}}U)=0]=\P[\X_n^{(g)}(U)=0],
\end{align*}
with $g(t)=t^{\alpha}$. Since  $g(t)=t^{\alpha}$ satisfies the condition of Theorem \ref{thm:gholeprobability1} when $\alpha\ge 1$ and  $U$ satisfies the condition \eqref{eqn:gcondition}. Therefore by Theorem \ref{thm:upperbound}, Theorem \ref{thm:gholeprobability1} and Corollary \ref{corollary}, we have
\begin{align*}
\lim_{n\to \infty}\frac{1}{n^2}\log\P[\X_n^{(\alpha)}(n^{\frac{1}{\alpha}}U)=0]=R_{\emptyset}^{(\alpha)}-R_{U}^{(\alpha)},
\end{align*}
for all $\alpha\ge 1$.
\end{proof}

As a corollary of Theorem \ref{thm:allalpha} and Theorem \ref{thm:alphagreaterthan1} we get asymptotics of the hole probabilities for finite Ginibre ensemble $\X_n^{(2)}$, proved in \cite{hole}.
 
 \begin{cor}
 Let $U$ be a open subset of $\;\D$ satisfying the conditions of Theorem \ref{thm:gholeprobability12} or  \eqref{eqn:gcondition}. Then 
 $$
 \lim_{n\to\infty}\frac{1}{n^2}\log\P[\X_n^{(2)}(\sqrt n U)=0]=R_{\emptyset}^{(2)}-R_U^{(2)}.
 $$
 \end{cor}
\chapter{Determinantal point processes with infinitely many points }\label{ch:infinite}
In this chapter we calculate the hole probabilities for a family of determinantal point processes, $\X_{\infty}^{(\alpha)}$ for $\alpha>0$, in the complex plane with infinitely many points. In particular, $\alpha=2$ gives the hole probabilities for infinite Ginibre ensemble, proved in \cite{hole}.

Recall that for fixed $\alpha>0$, $\mathcal X_{\infty}^{(\alpha)}$ is a determinatal point process in the complex plane  with the kernel  $\mathbb K_{\infty}^{(\alpha)}(z,w)=\frac{\alpha}{2\pi}E_{\frac{2}{\alpha},\frac{2}{\alpha}}(z\bar w)e^{-\frac{|z|^{\alpha}}{2}-\frac{|w|^{\alpha}}{2}}$ with respect to the Lebesgue measure on the complex plane, where $E_{a,b}(z)$ denotes the Mittag-Leffler function (for the details of Mittag-Leffler function see \cite{mittag} and references therein), an entire function when $a>0$ and $b>0$, defined by
$$
E_{a,b}(z)=\sum_{k=0}^{\infty}\frac{z^k}{\Gamma(ak+b)}.
$$
Note that, for $\alpha=2,$ $\mathcal X_{\infty}^{(2)}$ is the infinite Ginibre ensemble.  First we calculate the asymptotics of the hole probabilities for the scaled unit disk and scaled annulus, using Corollary \ref{cor:radialdistribution}, as the scaling factor goes to infinity. Then we calculate the asymptotics of the hole probabilities for general scaled open sets, using potential theory techniques.

\section{Hole probabilities for $\X_{\infty}^{(\alpha)}$ in circular domains}\label{sec:cir}
In this section we calculate the asymptotics of the hole probabilities for $\X_{\infty}^{(\alpha)}$ in circular domains, e.g., disk and annulus. The following result gives the asymptotic of the hole probability in scaled unit disk.

\begin{theorem}\label{thm:holedisk}
Let $r\D:= \{rz\;|\;z \in \D\}$, where $\D$ is open unit disk. Then we have 
$$
\lim_{r\to \infty}\frac{1}{r^{2\alpha}}\log\P[\mathcal X_{\infty}^{(\alpha)}(r\D)=0]=-\frac{\alpha}{2}\cdot\frac{1}{4}.
$$
\end{theorem}

\noindent In particular for $\alpha=2$, we get the hole probability (see \cite{manjubook}, Proposition 7.2.1) for the infinite Ginibre ensemble,
$$
\lim_{r\to \infty}\frac{1}{r^4}\log\P[\mathcal X_{\infty}^{(2)}(r\D)=0]=-\frac{1}{4}.
$$

\begin{proof}[Proof of the Theorem \ref{thm:holedisk}]
 Corollary \ref{cor:radialdistribution} implies that 
$$
\P[\mathcal X_{\infty}^{(\alpha)}(r\D)=0]=\prod_{k=1}^{\infty}\P[R_k^{\alpha}>r^{\alpha}],
$$
where $R_k^{\alpha}\sim$Gamma$(\frac{2}{\alpha}k,1)$.

\noindent{\bf Upper bound:} For $\theta<1$, we have 
$$
\P[R_k^{\alpha}>r^{\alpha}]\le e^{-\theta r^{\alpha}}\E[e^{\theta R_k^{\alpha}}]= e^{-\theta r^{\alpha}}(1-\theta)^{-\frac{2k}{\alpha}}.
$$
For $k<\frac{\alpha}{2}r^{\alpha}$, the bound is optimized for $\theta=1-\frac{2k}{\alpha r^{\alpha}}$ and 
\begin{align}\label{eqn:gup}
\P[R_k^{\alpha}>r^{\alpha}]\le e^{-\left\{(1-\frac{2k}{\alpha r^{\alpha}}) r^{\alpha}+\frac{2k}{\alpha}\log \frac{2k}{\alpha r^{\alpha}}\right\}}.
\end{align}
Therefore (we write as if $\frac{\alpha}{2}r^{\alpha}$ is an integer, it does not make any difference in asymptotics as $r\to \infty$) we have 
\begin{eqnarray*}
\P[\mathcal X_{\infty}^{(\alpha)}(r\D)=0]\le \prod_{k=1}^{\frac{\alpha}{2}r^{\alpha}}\P[R_k^{\alpha}>r^{\alpha}]
&\le &\prod_{k=1}^{\frac{\alpha }{2}r^{\alpha}}e^{-\left\{(1-\frac{2k}{\alpha r^{\alpha}}) r^{\alpha}+\frac{2k}{\alpha}\log \frac{2k}{\alpha r^{\alpha}}\right\}}.
\end{eqnarray*}
Therefore we have 
\begin{eqnarray*}
\limsup_{r\to \infty}\frac{1}{r^{2\alpha}}\log\P[\mathcal X_{\infty}^{(\alpha)}(r\D)=0]&\le &-\liminf_{r\to \infty}\frac{1}{r^{2\alpha}}\sum_{k=1}^{\frac{\alpha}{2}r^{\alpha}}\left\{\left(1-\frac{2k}{\alpha r^{\alpha}}\right) r^{\alpha}+\frac{2k}{\alpha}\log \frac{2k}{\alpha r^{\alpha}}\right\}
\\&=&-\frac{\alpha}{2}\liminf_{r\to \infty}\frac{2}{\alpha r^{\alpha}}\sum_{k=1}^{\frac{\alpha }{2}r^{\alpha}}\left\{\left(1-\frac{2k}{\alpha r^{\alpha}}\right) +\frac{2k}{\alpha r^{\alpha}}\log \frac{2k}{\alpha r^{\alpha}}\right\}
\\&=&-\frac{\alpha}{2}\int_{0}^{1}(1-x+x\log x)dx
\\&=&-\frac{\alpha}{2}\cdot \frac{1}{4}.
\end{eqnarray*}

\noindent{\bf Lower bound:} We have 
{\footnotesize\begin{align}\label{eqn:low}
\P[\mathcal X_{\infty}^{(\alpha)}(r\D)=0]
=&\prod_{k=1}^{\log r}\P[R_k^{\alpha}>r^{\alpha}]\prod_{k=\log r+1}^{\frac{\alpha}{2}r^{\alpha}}\P[R_k^{\alpha}>r^{\alpha}]\prod_{k=\frac{\alpha}{2}r^{\alpha}+1}^{\alpha r^{\alpha}}\P[R_k^{\alpha}>r^{\alpha}]\prod_{k={\alpha}r^{\alpha}+1}^{\infty}\P[R_k^{\alpha}>r^{\alpha}].
\end{align}}

Since $R_k^{\alpha}\sim \mbox{Gamma}(\frac{2}{\alpha}k,1)$, we can write $R_k^{\alpha}\stackrel{d}{=}X_1+X_2+\cdots+X_k$, where $X_1,X_2,\ldots,X_k$ are i.i.d. and Gamma$(\frac{2}{\alpha},1)$ distributed. Therefore we have 
$$
\P[R_k^{\alpha}>r^{\alpha}]\ge \P[X_1>r^{\alpha}]\ge\frac{1}{\Gamma(2/\alpha)}e^{-r^{\alpha}}r^{2-\alpha}.
$$
 Therefore we have 
\begin{equation}\label{eqn:1alpha}
\liminf_{r\to \infty}\frac{1}{r^{2\alpha}}\log \prod_{k=1}^{\log r}\P[R_k^{\alpha}>r^{\alpha}]\ge 0.
\end{equation}
The Stirling's formula for Gamma function implies that $\log \Gamma(t+1)\le \log(\sqrt{2\pi t})+t\log t -t +1 $ for large $t$. If $\log r\le k \le \frac{\alpha}{2}r^{\alpha} $, then for large $r$ 
\begin{eqnarray*}
\P[R_k^{\alpha}>r^{\alpha}]&\ge & \frac{1}{\Gamma(\frac{2}{\alpha}k)}e^{-r^\alpha}r^{\alpha(\frac{2}{\alpha}k-1)}
\\&\ge &e^{- \{r^\alpha-(\frac{2}{\alpha}k-1)\log r^{\alpha}+\log \Gamma(\frac{2}{\alpha}k+1)\}}
\\&\ge &e^{- \{r^\alpha -  \frac{2}{\alpha}k\log r^{\alpha}+\log r^{\alpha}+\log(\sqrt{2\pi.\frac{2}{\alpha}k})+\frac{2}{\alpha}k\log(\frac{2}{\alpha}k)-\frac{2}{\alpha}k+1\}}
\\&\ge &e^{- \{r^\alpha +\log r^{\alpha}+\log(\sqrt{2\pi.r^{\alpha}})+\frac{2}{\alpha}k\log(\frac{2k}{\alpha r^{\alpha}})-\frac{2}{\alpha}k+1\}}
\end{eqnarray*}
Therefore we have 
\begin{eqnarray}\label{eqn:2alpha}
&&\liminf_{r\to \infty}\frac{1}{r^{2\alpha}}\log \prod_{k=\log r}^{\frac{\alpha}{2} r^{\alpha}}\P[R_k^{\alpha}>r^{\alpha}]\nonumber
\\&\ge &-\limsup_{r\to \infty}\frac{1}{r^{2\alpha}}\sum_{k=\log r}^{\frac{\alpha}{2}r^{\alpha}}\left\{r^\alpha +\log r^{\alpha}+\log(\sqrt{2\pi.r^{\alpha}})+\frac{2}{\alpha}k\log\left(\frac{2k}{\alpha r^{\alpha}}\right)-\frac{2}{\alpha}k+1\right\}\nonumber
\\&=&-\frac{\alpha}{2}\cdot\limsup_{r\to \infty}\frac{2}{\alpha r^{\alpha}}\sum_{k=\log r}^{\frac{\alpha}{2}r^{\alpha}}\left\{1 +\frac{2k}{\alpha r^{\alpha}}\log\left(\frac{2k}{\alpha r^{\alpha}}\right)-\frac{2k}{\alpha r^{\alpha}}\right\}\nonumber
\\&= &-\frac{\alpha}{2}\cdot \int_{0}^1(1+x\log x-x)dx\nonumber
\\&=&-\frac{\alpha}{2}\cdot \frac{1}{4}.
\end{eqnarray}
For $\frac{\alpha}{2}r^{\alpha}\le k \le \alpha r^{\alpha}$, we have $\P[R_k^{\alpha}>r^{\alpha}]\ge \frac{1}{4}$ for large enough $r$. Indeed, by central limit theorem,
$$
\P[R_k^{\alpha}>r^{\alpha}]\ge \P\left[\frac{R_k^{\alpha}-\frac{2}{\alpha}k}{\sqrt{\frac{2}{\alpha}k}}>\frac{r^{\alpha}-\frac{2}{\alpha}k}{\sqrt{\frac{2}{\alpha}k}}\right] \ge \P\left[\frac{R_k^{\alpha}-\frac{2}{\alpha}k}{\sqrt{\frac{2}{\alpha}k}}>0\right]\to \frac{1}{2},
$$
as $k\to \infty$. Therefore we have 
\begin{eqnarray}\label{eqn:3}
\liminf_{r\to \infty}\frac{1}{r^{2\alpha}}\log \prod_{k=\frac{\alpha}{2} r^{\alpha}}^{\alpha r^{\alpha}}\P[R_k^{\alpha}>r^{\alpha}]\ge 0.
\end{eqnarray} 
 By the large deviation principle (Cramer's bound) for Gamma$(\frac{2}{\alpha},1)$ random variable, for $k>\alpha r^{\alpha}$, we get
 $$
 \P[R_k^{\alpha} \le r^{\alpha}]\le \P[R_k^{\alpha}\le \frac{k}{\alpha}]\le e^{-ck},
 $$ 
 for a positive  constant $c$ independent of $k$, as $R_k^{\alpha}\stackrel{d}{=}X_1+X_2+\cdots+X_k$ and $\E X_1=\frac{2}{\alpha}$ (where $X_1,\ldots, X_k$ are i.i.d. Gamma$(\frac{2}{\alpha},1)$ distributed). Therefore for large $r$ 
\begin{eqnarray}\label{eqn:4}
\prod_{k=\alpha r^{\alpha}+1}^{\infty} \P[R_k^{\alpha}>r^{\alpha}]\ge C,
\end{eqnarray}
where $C$ is positive constant. By \eqref{eqn:1alpha}, \eqref{eqn:2alpha}, \eqref{eqn:3} and \eqref{eqn:4} from \eqref{eqn:low} we have,
$$
\liminf_{r\to \infty}\frac{1}{r^{2\alpha}}\log\P[\mathcal X_{\infty}^{(\alpha)}(r\D)=0]\ge -\frac{\alpha}{2}\cdot \frac{1}{4},
$$
 the lower bound. Hence the result.
 \end{proof}

The next result gives  the asymptotic of the hole probability for annulus.
\begin{theorem}\label{thm:gannulus}
Let $U_c=\{z \;|\; c<|z|<1\}$ for fixed $0<c<1$. Then
$$
\lim_{r\to \infty}\frac{1}{r^{2\alpha}}\log \P[\X_{\infty}^{(\alpha)}(rU_c)=0]= -\frac{\alpha}{2}\cdot\l(\frac{1}{4}-\frac{c^{2\alpha}}{4}+\frac{(1-c^{\alpha})^2}{2\alpha\log c}\r).
$$
\end{theorem}

As a corollary of this result, for $\alpha=2$, we get asymptotic of the hole probability for infinite Ginibre ensemble, proved in \cite{hole}.
\begin{cor}\label{thm:annulus1}
Let $U_c=\{z \;|\; c<|z|<1\}$ for fixed $0<c<1$. Then
$$
\lim_{r\to \infty}\frac{1}{r^4}\log \P[\X_{\infty}^{(2)}(rU_c)=0]=-\frac{(1-c^2)}{4}\cdot\l(1+c^2+\frac{1-c^2}{\log c}\r).
$$
\end{cor}

\begin{proof}[Proof of Theorem \ref{thm:gannulus}]
From Corollary \ref{cor:radialdistribution}, we have 
{\footnotesize\begin{align}\label{eqn:ganupp}
\P[\X_{\infty}^{(\alpha)}(rU_c)=0]=\prod_{k=1}^{\infty}\left(\P[R_k^{\alpha}<c^{\alpha}r^{\alpha}]+\P[R_k^{\alpha}>r^{\alpha}]\right)
\le \prod_{k=\frac{\alpha}{2}\cdot c^{\alpha}r^{\alpha}}^{\frac{\alpha}{2}\cdot r^{\alpha}}\left(\P[R_k^{\alpha}<c^{\alpha}r^{\alpha}]+\P[R_k^{\alpha}>r^{\alpha}]\right).
\end{align}}
(We consider those $k$ for which $c^{\alpha}r^{\alpha}\le \E R_k^{\alpha}=\frac{2k}{\alpha}\le r^{\alpha}$).
Since the function $e^{-t}t^a$ is increasing when $t\le a$, we have 
\begin{align*}
\P[R_k^{\alpha}<c^{\alpha}r^{\alpha}]=\frac{1}{\Gamma(\frac{2}{\alpha}k)}\int_0^{c^{\alpha}r^{\alpha}}e^{-x}x^{\frac{2}{\alpha}k-1}dx\le \frac{1}{\Gamma(\frac{2}{\alpha}k)}e^{-c^{\alpha}r^{\alpha}}(c^{\alpha}r^{\alpha})^{\frac{2}{\alpha}k},
\end{align*}
for  all $k\ge \frac{\alpha}{2}\cdot c^{\alpha}r^{\alpha}$. The Stirling's formula for Gamma function implies that $\log \Gamma(t+1)\ge \log(\sqrt{2\pi t})+t\log t -t -1\ge t\log t -t $ for large $t$. Therefore for $c^{\alpha}r^{\alpha}<\frac{2}{\alpha}k\le r^{\alpha}$ and for large $r$,
\begin{align*}
\P[R_k^{\alpha}<c^{\alpha}r^{\alpha}]
 &\le \frac{2k}{\alpha}e^{-c^{\alpha}r^{\alpha}+\frac{2k}{\alpha}\log c^{\alpha}r^{\alpha}-\log\Gamma(\frac{2k}{\alpha}+1)}
\\&\le  r^{\alpha}e^{-\l\{c^{\alpha}r^{\alpha}-\frac{2k}{\alpha}+\frac{2k}{\alpha}\log(\frac{2k}{\alpha .c^{\alpha}r^{\alpha}})\r\}}.
\end{align*}
From \eqref{eqn:gup}, for $k\le \frac{\alpha}{2}r^{\alpha}$,  we have
\begin{align*}
\P[R_k^{\alpha}>r^{\alpha}]\le e^{-\l\{r^{\alpha}-\frac{2k}{\alpha}+\frac{2k}{\alpha}\log\frac{2k}{\alpha r^{\alpha}}\r\}}.
\end{align*}
Note that if $ k\le \frac{\alpha}{2}\cdot\lambda r^{\alpha}$, where $ \lambda=\frac{1-c^{\alpha}}{-\alpha\log c}
$, then we have 
$$
c^{\alpha}r^{\alpha}-\frac{2k}{\alpha}+\frac{2k}{\alpha}\log\frac{2k}{\alpha. c^{\alpha}r^{\alpha}}\le r^{\alpha}-\frac{2k}{\alpha}+\frac{2k}{\alpha}\log \frac{2k}{\alpha r^{\alpha}}.
$$
 Therefore, for large $r$, for $\frac{\alpha }{2}\cdot c^{\alpha}r^{\alpha} \le k \le \frac{\alpha}{2}\cdot  \lambda r^{\alpha}$, 
 \begin{align}\label{eqn:g1up}
 \P[R_k^{\alpha}<c^{\alpha}r^{\alpha}]+\P[R_k^{\alpha}>r^{\alpha}]\le (r^{\alpha}\log (c^{\alpha}r^{\alpha})+1)e^{-\l\{c^{\alpha}r^{\alpha}-\frac{2k}{\alpha}+\frac{2k}{\alpha}\log\frac{2k}{\alpha.c^{\alpha} r^{\alpha}}\r\}}.
 \end{align}
Similarly, for large $r$, we have 
 \begin{align}\label{eqn:g2up}
 \P[R_k^{\alpha}<c^{\alpha}r^{\alpha}]+\P[R_k^{\alpha}>r^{\alpha}]\le (r^{\alpha}\log (c^{\alpha}r^{\alpha})+1)e^{-\l\{r^{\alpha}-\frac{2k}{\alpha}+\frac{2k}{\alpha}\log\frac{2k}{\alpha r^{\alpha}}\r\}},
 \end{align}
when $ \frac{\alpha}{2}\cdot  \lambda r^{\alpha}\le k \le \frac{\alpha}{2}r^{\alpha}$. By \eqref{eqn:g1up} and \eqref{eqn:g2up}, from \eqref{eqn:ganupp} we get
\begin{align*}
&\P[\X_{\infty}^{(\alpha)}(rU_c)=0]
\\ \le & (r^{\alpha}\log (c^{\alpha}r^{\alpha})+1)^{\frac{\alpha}{2}r^{\alpha}(1-c^{\alpha})}\prod_{\frac{\alpha}{2}\cdot c^{\alpha}r^{\alpha}}^{\frac{\alpha}{2}\cdot\lambda r^{\alpha}}e^{-\l\{c^{\alpha}r^{\alpha}-\frac{2k}{\alpha}+\frac{2k}{\alpha}\log\frac{2k}{\alpha.c^{\alpha} r^{\alpha}}\r\}}\prod_{\frac{\alpha}{2}\cdot \lambda r^{\alpha}}^{\frac{\alpha}{2}\cdot r^{\alpha}}e^{-\l\{r^{\alpha}-\frac{2k}{\alpha}+\frac{2k}{\alpha}\log\frac{2k}{\alpha r^{\alpha}}\r\}},
\end{align*}
for large $r$. Therefore we have 
\begin{align*}
&\limsup_{r\to \infty}\frac{1}{r^{2\alpha}}\log \P[\X_{\infty}^{(\alpha)}(rU_c)=0]
\\&\le -\liminf_{r\to\infty}\frac{1}{r^{2\alpha}}\l\{\sum_{\frac{\alpha}{2}\cdot c^{\alpha}r^{\alpha}}^{\frac{\alpha}{2}\cdot\lambda r^{\alpha}}\l\{c^{\alpha}r^{\alpha}-\frac{2k}{\alpha}+\frac{2k}{\alpha}\log\frac{2k}{\alpha.c^{\alpha} r^{\alpha}}\r\}+\sum_{\frac{\alpha}{2}\cdot \lambda r^{\alpha}}^{\frac{\alpha}{2}\cdot r^{\alpha}}\l\{r^{\alpha}-\frac{2k}{\alpha}+\frac{2k}{\alpha}\log\frac{2k}{\alpha r^{\alpha}}\r\}\r\}
\\&\le -\frac{\alpha}{2}\liminf_{r\to\infty}\frac{2}{\alpha r^{\alpha}}\l\{\sum_{\frac{\alpha}{2}\cdot c^{\alpha}r^{\alpha}}^{\frac{\alpha}{2}\cdot\lambda r^{\alpha}}\l\{c^{\alpha}-\frac{2k}{\alpha r^{\alpha}}+\frac{2k}{\alpha r^{\alpha}}\log\frac{2k}{\alpha.c^{\alpha} r^{\alpha}}\r\}+\sum_{\frac{\alpha}{2}\cdot \lambda r^{\alpha}}^{\frac{\alpha}{2}\cdot r^{\alpha}}\l\{1-\frac{2k}{\alpha r^{\alpha}}+\frac{2k}{\alpha r^{\alpha}}\log\frac{2k}{\alpha r^{\alpha}}\r\}\r\}
\\&=-\frac{\alpha}{2}\cdot\l\{ \int_{c^{\alpha}}^{\lambda}\l(c^{\alpha}-x+x\log \frac{x}{c^{\alpha}}\r)dx+\int_{\lambda}^{1}(1-x+x\log x)dx \r\}
\\&=-\frac{\alpha}{2}\cdot\l\{\frac{1}{4}-\frac{c^{2\alpha}}{4}-\lambda(1-c^{\alpha})-\frac{\lambda^2}{2}\cdot \log c^{\alpha}\r\}.
\end{align*}
Putting $\lambda=\frac{1-c^{\alpha}}{-\log c^{\alpha}}$, we have the upper bound 
$$
\limsup_{r\to \infty}\frac{1}{r^{2\alpha}}\log \P[\X_{\infty}^{(\alpha)}(rU_c)=0]\le -\frac{\alpha}{2}\cdot\l(\frac{1}{4}-\frac{c^{2\alpha}}{4}+\frac{(1-c^{\alpha})^2}{2\alpha\log c}\r).
$$

\noindent{\bf Lower bound: }We have 
\begin{align*}
&\P[\X_{\infty}^{(\alpha)}(rU_c)=0]=\prod_{k=1}^{\infty}\l(\P[R_k^{\alpha}<c^{\alpha}r^{\alpha}]+\P[R_k^{\alpha}>r^{\alpha}]\r)
\\&\ge \prod_{k=1}^{\log r}\P[R_k^{\alpha}>r^{\alpha}] \prod_{\log r}^{\frac{\alpha}{2}c^{\alpha}r^{\alpha}}\P[R_k^{\alpha}<c^{\alpha}r^{\alpha}]\prod_{\frac{\alpha}{2}c^{\alpha}r^{\alpha}}^{\frac{\alpha}{2}\lambda r^{\alpha}}\P[R_k^{\alpha}<c^{\alpha}r^{\alpha}]
\prod_{\frac{\alpha}{2}\lambda r^{\alpha}}^{\frac{\alpha}{2}r^{\alpha}}\P[R_k^{\alpha}>r^{\alpha}]
\prod_{\frac{\alpha}{2} r^{\alpha}}^{\infty}\P[R_k^{\alpha}>r^{\alpha}]
\end{align*}
Since $R_k^{\alpha}\sim $Gamma$(\frac{2}{\alpha}k)$, for $k\le \frac{\alpha}{2}c^{\alpha}r^{\alpha}$, we have $\P[R_k^{\alpha}<c^{\alpha}r^{\alpha}]\ge \frac{1
}{4}$ for large $r$. Indeed, as $R_k^{\alpha}=X_1+X_2+\cdots+X_k$ (where $X_1, X_2,\ldots,X_k$ are i.i.d. Gamma($\frac{2}{\alpha},1$) random variables ), by central limit theorem 
$$
\P[R_k^{\alpha}<c^{\alpha}r^{\alpha}]=\P\l[\frac{R_k^{\alpha}-\frac{2}{\alpha}k}{\sqrt{\frac{2}{\alpha}k}}<\frac{c^{\alpha}r^{\alpha}-\frac{2}{\alpha}k}{\sqrt{\frac{2}{\alpha}k}}\r]\ge\P\l[\frac{R_k^{\alpha}-\frac{2}{\alpha}k}{\sqrt{\frac{2}{\alpha}k}}\le 0\r]\to \frac{1}{2},
$$
as  $k\to \infty$. Therefore we have 
\begin{align}\label{eqn:glow}
\liminf_{r\to \infty}\frac{1}{r^{2\alpha}}\log \prod_{\log r}^{\frac{\alpha}{2}c^{\alpha}r^{\alpha}}\P[R_k^{\alpha}<c^{\alpha}r^{\alpha}]\ge \liminf_{r\to \infty}\frac{\frac{\alpha}{2}c^{\alpha}r^{\alpha}-\log r}{r^{2\alpha}} \log\frac{1}{4}=0
\end{align}Again we have 
\begin{align*}
\P[R_k^{\alpha}<c^{\alpha}r^{\alpha}]=\frac{1}{\Gamma(\frac{2}{\alpha}k)}\int_{c^{\alpha}r^{\alpha}}^{\infty}e^{-x}x^{\frac{2}{\alpha}k-1}dx\ge \frac{e^{-c^{\alpha}r^{\alpha}}(c^{\alpha}r^{\alpha})^{\frac{2}{\alpha}k-1}}{\Gamma(\frac{2}{\alpha}k)}.
\end{align*}
Therefore we have 
\begin{align}\label{eqn:glow1}
&\liminf_{r\to \infty}\frac{1}{r^{2\alpha}}\log \prod_{\frac{\alpha}{2}c^{\alpha}r^{\alpha}}^{\frac{\alpha}{2}\lambda r^{\alpha}}\P[R_k^{\alpha}<c^{\alpha}r^{\alpha}]
= \liminf_{r\to \infty}\frac{1}{r^{2\alpha}}\sum_{\frac{\alpha}{2}c^{\alpha}r^{\alpha}}^{\frac{\alpha}{2}\lambda r^{\alpha}}\log \P[R_k^{\alpha}<c^{\alpha}r^{\alpha}]\nonumber
\\&\ge -\limsup_{r\to \infty}\frac{1}{r^{2\alpha}}\sum_{\frac{\alpha}{2}c^{\alpha}r^{\alpha}}^{\frac{\alpha}{2}\lambda r^{\alpha}}\l\{c^{\alpha}r^{\alpha}-\l(\frac{2k}{\alpha}-1\r)\log c^{\alpha}r^{\alpha}+\log\Gamma(\frac{2}{\alpha}k+1)\r\}\nonumber
\\&\ge  -\limsup_{r\to \infty}\frac{1}{r^{2\alpha}}\sum_{\frac{\alpha}{2}c^{\alpha}r^{\alpha}}^{\frac{\alpha}{2}\lambda r^{\alpha}}\l\{c^{\alpha}r^{\alpha}-\frac{2k}{\alpha}\log c^{\alpha}r^{\alpha} +\log c^{\alpha}r^{\alpha}+\frac{2k}{\alpha}\log(\frac{2k}{\alpha})-\frac{2k}{\alpha}+1\r\}\nonumber
\\&\;\;\;\;\mbox{ (by Stirling's formula)}\nonumber
\\&=  -\limsup_{r\to \infty}\frac{1}{r^{2\alpha}}\sum_{\frac{\alpha}{2}c^{\alpha}r^{\alpha}}^{\frac{\alpha}{2}\lambda r^{\alpha}}\l\{c^{\alpha}r^{\alpha}-\frac{2k}{\alpha} +\frac{2k}{\alpha}\log\frac{2k}{\alpha.c^{\alpha}r^{\alpha}}\r\}\nonumber
\\&=-\frac{\alpha}{2}\cdot \limsup_{r\to \infty}\frac{2}{\alpha r^{\alpha}}\sum_{\frac{\alpha}{2}c^{\alpha}r^{\alpha}}^{\frac{\alpha}{2}\lambda r^{\alpha}}\l\{c^{\alpha}-\frac{2k}{\alpha r^{\alpha}} +\frac{2k}{\alpha r^{\alpha}}\log\frac{2k}{\alpha.c^{\alpha}r^{\alpha}}\r\}\nonumber
\\&=-\frac{\alpha}{2}\cdot\int_{c^{\alpha}}^{\lambda}\l\{c^{\alpha}-x+x\log\frac{x}{c^{\alpha}}\r\}dx
\end{align}
From \eqref{eqn:gup}, we have 
\begin{align}\label{eqn:glow2}
\liminf_{r\to \infty}\frac{1}{r^{2\alpha}}\log \prod_{\frac{\alpha}{2}\lambda r^{\alpha}}^{\frac{\alpha}{2}r^{\alpha}}\P[R_k^{\alpha}>r^{\alpha}]\nonumber
&\ge -\limsup_{r\to \infty} \frac{1}{r^{2\alpha}}\sum_{\frac{\alpha}{2}\lambda r^{\alpha}}^{\frac{\alpha}{2}r^{\alpha}}\l\{r^{\alpha}-\frac{2k}{\alpha}+\frac{2k}{\alpha}\log\frac{2k}{\alpha r^{\alpha}}\r\}\nonumber
\\&=-\frac{\alpha}{2}\cdot\limsup_{r\to \infty}\frac{2}{\alpha r^{\alpha}}\sum_{\frac{\alpha}{2}\lambda r^{\alpha}}^{\frac{\alpha}{2}r^{\alpha}}\l\{1-\frac{2k}{\alpha r^{\alpha}}+\frac{2k}{\alpha r^{\alpha}}\log\frac{2k}{\alpha r^{\alpha}}\r\}\nonumber
\\&= -\frac{\alpha}{2}\cdot \int_{\lambda}^{1}\l\{1-x+x\log x\r\}dx
\end{align}
Therfore from \eqref{eqn:1alpha}, \eqref{eqn:4}, \eqref{eqn:glow}, \eqref{eqn:glow1}, \eqref{eqn:glow2},  we have 
\begin{align*}
&\liminf_{r\to \infty}\frac{1}{r^{2\alpha}}\log \P[\X_{\infty}^{(\alpha)}(rU_c)=0]
\\ \ge & -\frac{\alpha}{2}\cdot \l\{\int_{c^{\alpha}}^{\lambda}\l\{c^{\alpha}-x+x\log\frac{x}{c^{\alpha}}\r\}dx+ \int_{\lambda}^{1}\l\{1-x+x\log x\r\}dx\r\}
\\=&-\frac{\alpha}{2}\cdot\l\{\frac{1}{4}-\frac{c^{2\alpha}}{4}-\lambda(1-c^{\alpha})-\frac{\lambda^2}{2}\cdot \log c^{\alpha}\r\}.
\end{align*}
Putting $\lambda=\frac{1-c^{\alpha}}{-\log c^{\alpha}}$, we have the lower bound 
$$
\liminf_{r\to \infty}\frac{1}{r^{2\alpha}}\log \P[\X_{\infty}^{(\alpha)}(rU_c)=0]\ge -\frac{\alpha}{2}\cdot\l(\frac{1}{4}-\frac{c^{2\alpha}}{4}+\frac{(1-c^{\alpha})^2}{2\alpha\log c}\r).
$$
Hence the result.
\end{proof}

\section{Hole probabilities for $\X_{\infty}^{(\alpha)}$ in general domains}
In this section we calculate the asymptotics of the hole probabilities for $\X_{\infty}^{(\alpha)}$ for certain scaled open sets. The following result gives the decay constants of the hole probabilities in terms of minimum energies of the complements of the given open sets.

\begin{theorem}\label{thm:geninfinite}
Let $U$ be an open subset of $ D(0,(\frac{2}{\alpha})^{\frac{1}{\alpha}})$ satisfying the conditions of Theorem \ref{thm:gholeprobability12}. Then 
$$
\lim_{r\to \infty}\frac{1}{r^{2\alpha}}\log\P[\X_{\infty}^{(\alpha)}(rU)=0]=R_{\emptyset}^{(\alpha)}-R_U^{(\alpha)},
$$
for all $\alpha>0$.
\end{theorem}

The following remarks say that we can recover the hole probabilities for $\X_{\infty}^{(\alpha)}$ for circular domains using Theorem \ref{thm:geninfinite}, calculated in Section \ref{sec:cir}.
\begin{remark}
\begin{enumerate}
\item Choose $a>0$ such that $a<(\frac{2}{\alpha})^{\frac{1}{\alpha}}$. Then by Theorem \ref{thm:geninfinite} we have
\begin{align*}
\lim_{r\to\infty}\frac{1}{r^{2\alpha}}\log\P[\X_{\infty}^{(\alpha)}(r\D)=0]&=\lim_{r\to\infty}\frac{1}{r^{2\alpha}}\log\P[\X_{\infty}^{(\alpha)}\l(\frac{r}{a}\cdot a\D\r)=0]
=\frac{1}{a^{2\alpha}}\cdot \l[R_{\emptyset}^{(\alpha)}-R_U^{(\alpha)}\r],
\end{align*}
where $U=D(0,a)$. Therefore by Remark \ref{re:constant} we get 
\begin{align*}
\lim_{r\to\infty}\frac{1}{r^{2\alpha}}\log\P[\X_{\infty}^{(\alpha)}(r\D)=0]=-\frac{1}{a^{2\alpha}}\cdot \frac{\alpha}{2}\cdot\frac{a^{2\alpha}}{4}=-\frac{\alpha}{2}\cdot\frac{1}{4}.
\end{align*}
Thus we get the result of Theorem \ref{thm:holedisk}.

\item Let $b>0$ such that $b<(\frac{2}{\alpha})^{\frac{1}{\alpha}}$ and $U_c=\{z\suchthat 0<c<|z|<1\}$. Then by Theorem \ref{thm:geninfinite} we have 
\begin{align*}
\lim_{r\to\infty}\frac{1}{r^{2\alpha}}\log\P[\X_{\infty}^{(\alpha)}(rU_c)=0]&=\lim_{r\to\infty}\frac{1}{r^{2\alpha}}\log\P[\X_{\infty}^{(\alpha)}\l(\frac{r}{b}\cdot bU_c\r)=0]
=\frac{1}{b^{2\alpha}}\cdot \l[R_{\emptyset}^{(\alpha)}-R_{bU_c}^{(\alpha)}\r],
\end{align*}
where $bU_c=\{z\suchthat 0<cb<|z|<b\}$, annulus with inner radius $cb$ and outer radius $b$. Therefore by the Remark \ref{re:constant}, for $a=cb$, we get 
\begin{align*}
\lim_{r\to\infty}\frac{1}{r^{2\alpha}}\log\P[\X_{\infty}^{(\alpha)}(rU_c)=0]&=-\frac{1}{b^{2\alpha}}\cdot \frac{\alpha}{2} \l(\frac{b^{2\alpha}}{4}-\frac{a^{2\alpha}}{4}-\frac{(b^{\alpha}-a^{\alpha})^2}{2\alpha\log(b/a)}\r)
\\&=-\frac{\alpha}{2} \l(\frac{1}{4}-\frac{c^{2\alpha}}{4}+\frac{(1-c^{\alpha})^2}{2\alpha\log(c)}\r).
\end{align*}
This recovers the result of Theorem \ref{thm:gannulus}.
\end{enumerate}
\end{remark}

Now we proceed to prove Theorem \ref{thm:geninfinite}.
\begin{proof}[Proof of Theorem \ref{thm:geninfinite}]
 Fixed $\alpha>0$. Since $\X_{n}^{(\alpha)}$  converges in distribution to $\X_{\infty}^{(\alpha)}$  as $n \to \infty$, therefore  
$$
\P[\X_{\infty}^{(\alpha)}(rU)=0]=\lim_{n \to \infty} \P [\X_{n}^{(\alpha)}(rU)=0].
$$
Again $\X_{n}^{(\alpha)}$ is a determinantal point process with kernel $\mathbb K_{n}^{(\alpha)}$ with respect to Lebesgue measure. The kernel $\mathbb K_{n}^{(\alpha)}$ can be expressed as
$$
\mathbb K_{n}^{(\alpha)}(z,w)=\sum_{k=0}^{n-1}\varphi_k(z)\bar{\varphi_k(w)}\;\;\mbox{ where $\varphi_k(z)=\frac{\sqrt{\alpha}. z^k}{\sqrt{2\pi\Gamma(\frac{2}{\alpha}(k+1))}}e^{-\frac{|z|^{\alpha}}{2}}$}.
$$
The joint density of the points of $\X_{n}^{(\alpha)}$, with uniform order, is
$$
\frac{1}{n!}\det \l(\mathbb K_{n}^{(\alpha)}(z_i,z_j)\r)_{1\le i,j\le n},
$$
with respect to the Lebesgue measure in $\C^n$. 
Therefore 
\begin{align*}
\P[\X_{n}^{(\alpha)}(rU)=0]
&=\frac{1}{n!}\int_{(rU)^c}\cdots \int_{(rU)^c}\det (K_n(z_i,z_j))_{1\le i,j\le n}\prod_{i=1}^ndm(z_i)
\\&=\frac{1}{n!}\int_{(rU)^c}\cdots \int_{(rU)^c}\det (\varphi_k(z_i))_{1\le i,k\le n}\det (\bar {\varphi_k(z_i)})_{1\le i,k\le n}\prod_{i=1}^ndm(z_i)
\\&=\frac{1}{n!}\int_{(rU)^c}\cdots \int_{(rU)^c}\sum_{\sigma,\tau \in S_n}{{\sgn}(\sigma)}{{\sgn}(\tau)}\prod_{i=1}^{n}\varphi_{\sigma(i)}(z_i)\bar {\varphi_{\tau(i)}(z_i)}\prod_{i=1}^ndm(z_i)
\\&=\sum_{\sigma\in S_n}{{\sgn}(\sigma)}\prod_{i=1}^{n}\int_{(rU)^c}\varphi_{i}(z)\bar {\varphi_{\sigma(i)}(z)}dm(z)
\\&=\det\left(\int_{(rU)^c}\varphi_{i}(z)\bar {\varphi_{j}(z)}dm(z)\right)_{1\le i,j\le n}.
\end{align*}
Let us define
 $$M_n(rU):=\left(\int_{(rU)^c}\varphi_{i}(z)\bar {\varphi_{j}(z)}dm(z)\right)_{1\le i,j\le n}.$$  $M_n(rU)$ is the integral of the positive definite matrix function $\left(\varphi_{i}(z)\bar {\varphi_{j}(z)}\right)_{1\le i,j\le n}$ over the region $(rU)^c$. So, we have  that 
$ M_n(rU)\ge M_n(rD)\ge 0$ for all $n$ and $U\subseteq D=D(0,(\frac{2}{\alpha})^{\frac{1}{\alpha}})$. Since for a positive definite matrix $\begin{bmatrix}A & B \\ B^* & D\end{bmatrix}$ we have 
$$
\det\begin{bmatrix}A & B \\ B^* & D\end{bmatrix}\le \det(A)\det(D). 
$$
 Therefore we have
$$
\det (M_n(rU))\le \det (M_{n-1}(rU)).\int_{(rU)^c}\varphi_n(z)\bar{\varphi_n(z)}dm(z)\le \det (M_{n-1}(rU)).
$$
So $\P[\X_n^{(\alpha)}(rU)=0]=\det (M_n(rU))$ is decreasing and decreases to $\P[\X_{\infty}^{(\alpha)}(rU)=0]$. Therefore,  for all $n\ge 2r^{\alpha}$, we have
\begin{align}\label{eqn:up}
\P[\X_{2r^{\alpha}}^{(\alpha)}(rU)=0]\ge \P[\X_n^{(\alpha)}(rU)=0]\ge \P[\X_{\infty}^{(\alpha)}(rU)=0].
\end{align}
Again for $n\ge 2r^{\alpha}$, we have 
\begin{align}\label{eqn:det}
\P[\X_n^{(\alpha)}(rU)=0]= \det(M_n(rU))=\det(M_{2r^{\alpha}}(rU))\det([M_n(rU)/M_{2r^{\alpha}}(rU)]),
\end{align}
where $[M_n(rU)/M_{2r^{\alpha}}(rU)]$ is  the Schur complement of the block $M_{2r^{\alpha}}(rU)$ of the matrix $M_{n}(rU)$. Recall, the Schur complement of the block $D$ of the matrix  
$$
M=\begin{bmatrix}A & B \\ C & D\end{bmatrix}\;\;\mbox{  is  }\;\; [M/D]=A-BD^{-1}C.
$$
The inverse of block matrix $M$ is given by 
$$
M^{-1}=\begin{bmatrix}[M/D]^{-1} & -A^{-1}B[M/A]^{-1} \\ -D^{-1}C[M/D]^{-1} & [M/A]^{-1}\end{bmatrix},
$$
where  $[M/A]=D-CA^{-1}B$. Since $M_n(rU)\ge M_n(rD)\ge 0$, $(M_n(rD))^{-1}\ge (M_n(rU))^{-1}$ and hence $[M_n(rD)/M_{2r^{\alpha}}(rD)]^{-1}\ge [M_n(rU)/M_{2r^{\alpha}}(rU)]^{-1}$.  Therefore the Schur complements satisfy the inequality
$$
[M_n(rU)/M_{2r^{\alpha}}(rU)]\ge [M_n(rD)/M_{2r^{\alpha}}(rD)].
$$
Therefore, min-max theorem for eigenvalues we have that the $i$-th largest eigenvalue of $[M_n(rU)/M_{2r^{\alpha}}(rU)]$ is greater than $i$-th largest eigenvalue of $[M_n(rD)/M_{2r^{\alpha}}(rD)]$. Hence we have
\begin{align}\label{eqn:schur}
\det([M_n(rU)/M_{2r^{\alpha}}(rU)])\ge \det([M_n(rD)/M_{2r^{\alpha}}(rD)]).
\end{align}
As $D$ is rotationally invariant,  we have 
$$
\int_{(rD)^c}\varphi_i(z)\bar{\varphi_j(z)}dm(z)=0\;\;\; \mbox{for all $i\neq j$}.
$$
Therefore $M_n(rD)=\mbox{diag}\left(\int_{(rD)^c}|\varphi_1(z)|^2dm(z),\ldots,\int_{(rD)^c}|\varphi_n(z)|^2dm(z)\right)$. From \eqref{eqn:schur} 
\begin{eqnarray}\label{eqn:tailed}
\det([M_n(rU)/M_{2r^{\alpha}}(rU)])&\ge&\prod_{k=2r^{\alpha}+1}^{n}\int_{(rD)^c}|\varphi_k(z)|^2dm(z)\nonumber
\\&\ge &\prod_{k=2r^{\alpha}+1}^{\infty}\int_{(rD)^c}|\varphi_k(z)|^2dm(z).
\end{eqnarray}
Again, for $k>2r^{\alpha}$, we have
\begin{align*}
\int_{(rD)^c}|\varphi_k(z)|^2dm(z)&=\P\left[R_{k+1}^{\alpha}>\frac{2}{\alpha}r^{\alpha}\right]=1-\P\left[R_{k+1}^{\alpha}\le \frac{2}{\alpha}r^{\alpha}\right]
\\&\ge 1-\P\left[R_{k+1}^{\alpha}< \frac{k+1}{\alpha}\right]=1-\P\left[\frac{R_{k+1}^{\alpha}}{k+1} <\frac{1}{\alpha}\right]
\\&\ge 1-e^{-c.k},
\end{align*}
last inequality from the large deviation (Crammer's bound) for Gamma$(\frac{2}{\alpha},1)$ random variable, as $R_k^{\alpha}\stackrel{d}{=}X_1+X_2+\cdots+X_k$ and $\E X_1=\frac{2}{\alpha}$ (where $X_1,\ldots, X_k$ are i.i.d. Gamma$(\frac{2}{\alpha},1)$ distributed). Therefore, for large $r$
$$
\prod_{k=2r^{\alpha}+1}^{\infty}\int_{(rD)^c}|\varphi_k(z)|^2dm(z)\ge e^{-2.\sum_{k=2r^{\alpha}}^{\infty}e^{-ck}}\ge C>0.
$$
Therefore from \eqref{eqn:tailed} we get 
\begin{align*}
\det([M_n(rU)/M_{2r^{\alpha}}(rU)])\ge C.
\end{align*}
 Therefore from \eqref{eqn:det}, for large $r$, we have 
\begin{eqnarray}\label{eqn:low}
\P[\X_{\infty}^{(\alpha)}(rU)=0]&=&\lim_{n\to \infty}\P[\X_n^{(\alpha)}(rU)=0]\ge C.\P[\X_{2r^{\alpha}}^{(\alpha)}(rU)=0].
\end{eqnarray}
Therefore from \eqref{eqn:up} and \eqref{eqn:low} and Theorem \ref{thm:gholeprobability12}, we get 
\begin{align}\label{eqn:master}
\lim_{r\to \infty}\frac{1}{r^{2\alpha}}\log\P[\X_{\infty}^{(\alpha)}(rU)=0]&=\lim_{r\to \infty}\frac{1}{r^{2\alpha}}\log\P[\X_{2r^{\alpha}}^{(\alpha)}(rU)=0]\nonumber
\\&=\lim_{n\to \infty}\frac{4}{n^2}\log\P[\X_{n}^{(\alpha)}(n^{\frac{1}{\alpha}}\cdot 2^{-\frac{1}{\alpha}}.U)=0].
\end{align}
Since $U$ satisfies the conditions of Theorem \ref{thm:gholeprobability12}, from Theorem \ref{thm:allalpha} we have
\begin{eqnarray*}
\lim_{r\to \infty}\frac{1}{r^{2\alpha}}\log\P[\X_{\infty}^{(\alpha)}(rU)=0]&=&4.\left(R_{\emptyset}^{(\alpha)}-R_{2^{-\frac{1}{\alpha}}.U}^{(\alpha)}\right)
\\&=&-4.R_{ 2^{-\frac{1}{\alpha}}.U}^{(\alpha)'}
=-R_U^{(\alpha)'}=R_{\emptyset}^{(\alpha)}-R_U^{(\alpha)},
\end{eqnarray*}
third equality follows from Theorem \ref{thm:generalsetting} and $R_{a.U}^{(\alpha)'}=a^{2\alpha}R_U^{(\alpha)'}$ (see Remark \ref{re:scale}).
\end{proof}

The next result gives the hole probabilities for $\X_{\infty}^{(\alpha)}$, when $\alpha\ge 1$, in another class of domains.

\begin{theorem}\label{thm:geninfinite2}
Let $U$ be an open subset of $ D(0,(\frac{2}{\alpha})^{\frac{1}{\alpha}})$ satisfying the condition \eqref{eqn:gcondition}. Then for $\alpha\ge 1$, 
$$
\lim_{r\to \infty}\frac{1}{r^{2\alpha}}\log\P[\X_{\infty}^{(\alpha)}(rU)=0]=R_{\emptyset}^{(\alpha)}-R_U^{(\alpha)}.
$$
\end{theorem}

\begin{proof}[Proof of Theorem \ref{thm:geninfinite2}]
From \eqref{eqn:master} we have 
\begin{align*}
\lim_{r\to \infty}\frac{1}{r^{2\alpha}}\log\P[\X_{\infty}^{(\alpha)}(rU)=0]
=\lim_{n\to \infty}\frac{4}{n^2}\log\P[\X_{n}(n^{\frac{1}{\alpha}}\cdot 2^{-\frac{1}{\alpha}}.U)=0],
\end{align*}
for all $\alpha>0$. Since $U$ satisfies the condition \eqref{eqn:gcondition},  from Theorem \ref{thm:alphagreaterthan1} we have 
\begin{align*}
\lim_{r\to \infty}\frac{1}{r^{2\alpha}}\log\P[\X_{\infty}^{(\alpha)}(rU)=0]
=4.\left(R_{\emptyset}^{(\alpha)}-R_{2^{-\frac{1}{\alpha}}.U}^{(\alpha)}\right).
\end{align*}
The result follows from Theorem \ref{thm:generalsetting} and  $R_{a.U}^{(\alpha)'}=a^{2\alpha}R_U^{(\alpha)'}$.
\end{proof}

As a corollary of Theorem \ref{thm:geninfinite} and Theorem \ref{thm:geninfinite2} we get the asymptotics of the hole probabilities for infinite Ginibre ensemble $\X_{\infty}^{(2)}$, proved in \cite{hole}.

\begin{cor}
Let $U$ be an open subset of  $\;\D$ satisfying the condition of Theorem \ref{thm:gholeprobability12} or \eqref{eqn:gcondition}. Then 
$$
\lim_{r\to \infty}\frac{1}{r^{4}}\log\P[\X_{\infty}^{(2)}(rU)=0]=R_{\emptyset}^{(2)}-R_U^{(2)}.
$$ 
\end{cor}
\chapter{Hole probabilities for finite $\beta$-ensembles in the complex plane}\label{ch:beta}
In this chapter we calculate the hole probabilities for finite $\beta$-ensembles in the complex plane, using potential theory techniques. The $\beta$-ensembles are a generalization of joint probability distributions of eigenvalues of random matrix ensembles. These ensembles appear in physics to explain the $2$-dimensional Coulomb gas models \cite{coulomb}.

Consider a family of point processes  $\X_{n,\beta}^{(g)}$, for $\beta>0$ and the function $g$ as in Section \ref{sec:exdet}, in the complex plane  with $n$ points. The joint density of the set of points of $\X_{n,\beta}^{(g)}$ (with uniform order) is defined by 
\begin{align}\label{eqn:betadensity}
\frac{1}{Z_{n,\beta}^{(g)}}\prod_{i<j}|z_i-z_j|^{\beta}e^{-n\sum_{k=1}^ng(|z_k|)}
\end{align}
with respect to Lebesgue measure on $\C^n$, where $Z_{n,\beta}^{(g)}$ is the normalizing constant, i.e.,
$$
Z_{n,\beta}^{(g)}=\int\ldots \int \prod_{i<j}|z_i-z_j|^{\beta}e^{-n\sum_{k=1}^ng(|z_k|)}\prod_{k=1}^{n}dm(z_k).
$$
In general $\beta$-ensembles, except $\beta=2$, are not determinantal processes. If $\beta=2$, then $\X_{n,\beta}^{(g)}$ are the determinantal point processes $\X_{n}^{(g)}$, defined in Section \ref{sec:exdet}. In particular, if $\beta=2$ and $g(|z|)=|z|^2$, then $\X_{n,\beta}^{(g)}$ is the scaled (by ${1}/{\sqrt n}$) $n$-th Ginibre ensemble. 

These processes in the complex plane and  the analogous processes in the real line have been studied extensively, e.g. see \cite{andersonbook}, \cite{coulomb} and \cite{bloom}. The large deviation for  $\X_{n,\beta}^{(g)}$ has been studied by Bloom \cite{bloom}. We refer the reader to see Section 2.6  of \cite{andersonbook} for the large deviation results for $\beta$-ensembles in the real line. In this chapter we calculate the hole probabilities for $\X_{n,\beta}^{(g)}$, using potential theory techniques.

In this chapter we use the following notation:
\begin{align*}
R_{U,\beta}^{(g)}=\inf \{R_{\mu,\beta}^{(g)}\suchthat \mu \in \mathcal P(U^c)\},\mbox{ where } R_{\mu,\beta}^{(g)}=\iint \log\frac{1}{|z-w|}d\mu(z)d\mu(w)+\frac{2}{\beta}\int g(|z|)d\mu(z),
\end{align*}
where $U$ is an open subset of $D(0,T_{\beta})$, $T_{\beta}$ denotes the solution of $tg'(t)=\beta$. The following remarks are useful for calculating the hole probabilities for $\X_{n,\beta}^{(g)}$.
\begin{remark}
Replacing $g$ by $\frac{2}{\beta}g$ in Theorem \ref{thm:diskgeneral} and Theorem \ref{thm:generalsetting}, we have following results.
\begin{enumerate}
\item The equilibrium measure for $\C$ with respect to the external field $\frac{g(|z|)}{\beta}$ is supported on $D(0,T_{\beta})$ and given by (setting $z=re^{i\theta}$)
\begin{align*}
d\mu_{\beta}(z)=\frac{1}{2\beta\pi}[g''(r)+\frac{1}{r}g'(r)]dm(z) \;\;\mbox{when $|z|<T_{\beta}$}.
\end{align*} 
The minimum energy is given by
$$
R_{\emptyset,\beta}^{(g)}=\log \frac{1}{T_{\beta}}+\frac{2}{\beta}g(T_{\beta})-\frac{1}{2\beta}\int_0^{T_{\beta}}r(g'(r))^2dr.
$$

\item Let  $U$ be an open subset of $D(0,T_{\beta})$. Then from Theorem \ref{thm:generalsetting}, we have
\begin{align*}
R_{U,\beta}^{(g)}=R_{\emptyset,\beta}^{(g)}+\frac{1}{\beta}\l[\int g(|z|)d\nu_2(z)-\int g(|z|)d\mu_2(z)\r],
\end{align*}
where $\nu_{2}$ is the balayage measure on $\partial
U$ with respect to the measure $\mu_2=\mu_{\beta}\big |_{U}$.

\item If $g(t)=t^{\alpha}$. Then $T_{\beta}=\l(\frac{\beta}{\alpha}\r)^{\frac{1}{\alpha}}$, the radius of the support of the equilibrium measure. In particular, $\alpha=2, \beta=2$ give $T_2=1$, the radius of the support of the equilibrium measure (unit disk) with quadratic external field.
\end{enumerate}

\end{remark}

\noindent Let $U$ be an open set. Then from \eqref{eqn:betadensity}, we have
\begin{align}\label{eqn:probhole}
&\P[\X_{n,\beta}^{(g)}(U)=0]\nonumber
\\=&\frac{1}{z_{n,\beta}^{(g)}}\int_{U^c}\ldots \int_{U^c} \prod_{i<j}|z_i-z_j|^{\beta}e^{-n\sum_{k=1}^ng(|z_k|)}\prod_{k=1}^{n}dm(z_k)\nonumber
\\=&\frac{1}{z_{n,\beta}^{(g)}}\int_{U^c}\ldots \int_{U^c}e^{-n\cdot\frac{\beta}{2}\l\{\frac{1}{n^2}\sum_{i\neq j}\log\frac{1}{|z_i-z_j|}+\frac{2}{\beta}\cdot\frac{1}{n}\sum_{k=1}^{n}g(|z_k|)\r\}}\prod_{k=1}^{n}dm(z_k).
\end{align}

\noindent Similar to the previous chapters we have hole probabilities results for $\X_{n,\beta}^{(g)}$, for two classes of open sets.
\begin{theorem}
Let $U$ be an open subset of $D(0,T_{\beta})$ satisfying the condition of Theorem \ref{thm:gholeprobability12}. Then
\begin{align*}
\lim_{n\to \infty}\frac{1}{n^2}\log \P[\X_{n,\beta}^{(g)}(U)=0]=-\frac{\beta}{2}\l[R_{U,\beta}^{(g)}-R_{\emptyset,\beta}^{(g)}\r].
\end{align*}
\end{theorem}

\begin{proof}
Applying the same calculations as in the proofs of Theorem \ref{thm:upperbound} and Theorem \ref{thm:gholeprobability12} in \eqref{eqn:probhole}, we get the result.
\end{proof}

The following result gives the hole probabilities for another class of open sets. 
\begin{theorem}
Let $U$ be an open subset of $D(0,T_{\beta})$ satisfying the condition \eqref{eqn:gcondition} and $g'$ is bounded on $[0,T_{\beta}+1]$. Then
\begin{align*}
\lim_{n\to \infty}\frac{1}{n^2}\log \P[\X_{n,\beta}^{(g)}(U)=0]=-\frac{\beta}{2}\l[R_{U,\beta}^{(g)}-R_{\emptyset,\beta}^{(g)}\r].
\end{align*}
\end{theorem}
\begin{proof}
Following the same calculations as in the proofs of Theorem \ref{thm:upperbound} and Theorem \ref{thm:gholeprobability1} in \eqref{eqn:probhole}, we get the result.
\end{proof}





\chapter{Fluctuations}\label{ch:fluctuations}

In this chapter we move away from the hole probabilities. We calculate variances for linear statistics of a family determinantal point processes on the unit disk.

Let $\mathcal X_L$ be determinantal point processes in the unit disk  with kernels $\mathbb{K}_L$ with respect to the measures $\mu_L$ for $L>0$, where
$$
\mathbb{K}_L(z,w)=\frac{1}{(1-z\bar{w})^{L+1}}\;\;\mbox{and}\;\;d\mu_L(z)=\frac{L}{\pi}(1-|z|^2)^{L-1}dm(z),
$$
for $z,w\in \mathbb{D}$ and $m$ is Lebesgue measure on $\mathbb{D}$. The processes $\X_L$, for positive integer $L$, come from the singular points of matrix valued Gaussian analytic functions \cite{manju}. In the case of $L=1$ was proved in \cite{balint}. 
\begin{result}[Krishnapur]\label{manjunath}
Let $G_k$, $k\ge 0,$ be i.i.d. $L\times L$ matrices, each with i.i.d. standard complex Gaussian entries. Then for each $L\ge 1,$ the singular points of $G_0+zG_1+z^2G_2+\cdots,$ that is to say, the zeros of $\det(G_0+zG_1+z^2G_2+\cdots),$ form a determinantal point process on the unit disc, with kernel $\mathbb{K}_L$ with respect to the measure $\mu_L$.
\end{result}

\noindent 
Let $\varphi$ be a compactly supported function on unit disk. The linear statistics of $\X_L$, for given $\varphi$, is
$$
\mathcal X_L(\varphi)=\sum_{z\in \mathcal X_L}\varphi(z).
$$  
Note that if $\varphi=\one_D$ for some $D\subset \D$, then $\X_L(\varphi)$ is just $\X_L(D)$, the number of points of $\X_L$ that fall in $D$. Let $\mathbb{V}[X]$ be the variances of the random variable $X$. In this chapter we calculate $\mathbb{V}[\X_L(r\D)]$ and $\mathbb V(\X_L(\varphi))$, when  $\varphi(z)=(1-\frac{|z|^2}{r^2})_+^{\frac{p}{2}}$ for $p>0$, as $r\to 1^-$.

The problem is inspired by the work of Buckley \cite{jery}. He consider the point processes $\X_{f_L}$ in the unit disk for $L>0$, the zeros set of hyperbolic Gaussian analytic functions
$$
f_L(z)=\sum_{k=0}^{\infty}\sqrt{\frac{L(L+1)\cdots(L+k-1)}{k!}}\cdot a_kz^k,
$$ 
where $a_k$ are i.i.d. standard complex Gaussian random variables. He proved the following result.
\begin{result}[Jeremiah Buckley, \cite{jery}]
The variances of $\X_{f_L}(r\D)$ are given below:
\begin{enumerate}
\item For each fixed $L>\frac{1}{2}$, as $r\to 1^-$,
$$
\mathbb V[\X_{f_L}(r\D)]=\Theta\l(\frac{1}{1-r}\r), 
$$
\item For $L=\frac{1}{2}$, 
$$
\mathbb V[\X_{f_L}(r\D)]=\Theta\l(\frac{1}{1-r}\log\frac{1}{1-r}\r), 
$$

\item For each fixed $L<\frac{1}{2}$, as $r\to 1^-$,
$$
\mathbb V[\X_{f_L}(r\D)]=\Theta\l(\frac{1}{(1-r)^{2-2L}}\r), 
$$
\end{enumerate}
\end{result}
\noindent
The bounds show that there is a transition in the variances at $L=\frac{1}{2}$. We mentioned that Buckley, Nishry, Peled and Sodin \cite{peled} calculated the hole probabilities for $\X_{f_L}$. However, we are not considering the hole probability problem for $\X_L$. We have following two results for variances.  

\begin{theorem}\label{thm}
For fixed L, as $r\to 1^-$, 
$$
\mathbb{V}[\X_{L}(r\D)]=\Theta\l(\frac{1}{1-r}\r).
$$
In other words, for fixed L, as $r\to 1^-$,
$$
c_2\le (1-r)\mathbb{V}[\X_{L}(r\D)]\le c_1
$$
where $c_1,c_2$ are constants depending on $L$.
\end{theorem}

\noindent Note that there is no transition in the variances. Let $\varphi_p(z)=(1-\frac{|z|^2}{r^2})_+^{\frac{p}{2}}$ for $p>0$ and $0<r<1$. 

\begin{theorem}\label{thm:result2}
For fixed $L>0$ and as $r\to 1^-$, we have
\begin{enumerate}
\item if $p<1$, then
$$
\mathbb{V}[\mathcal X_{L}(\varphi_p)]=\Theta((1-r)^{-(1-p)}),
$$

\item if $p=1$, then
$$
\mathbb{V}[\mathcal X_{L}(\varphi_p)]=\Theta(-\log(1-r)),
$$

\item if $p>1$, then
$$
\mathbb{V}[\mathcal X_{L}(\varphi_p)]=\Theta(1).
$$
\end{enumerate}

\end{theorem}
\noindent Note that there is a transition in variances  at $p=1$, independent of $L$. We state two lemmas to prove Theorem \ref{thm} and Theorem \ref{thm:result2}.

\begin{lemma}\label{lem}
Let $\mathcal X$ be a determinantal point process in the complex plane with kernel  $\mathbb{K}$ with respect to the measure $\mu$. If $D\subseteq\C$, then
$$
\mathbb{V}[\mathcal X(D)]=\int_D\int_{D^c}|\mathbb{K}(z,w)|^2d\mu(z)d\mu(w)
$$
where $D^c=\C\backslash D$.
\end{lemma}

\begin{lemma}\label{lem:fomula}
Let $\mathcal X$ be a determinantal point process in the complex plane with kernel $\mathbb{K}(z,w)$ with respect to the measure $\mu$ and $\varphi$ be a compactly supported function on $\C$. Then
$\E[\mathcal X(\varphi)]=\int \varphi(z)\mathbb{K}(z,z)d\mu(z)$ and
$$
\mathbb{V}[\mathcal X(\varphi)]=\frac{1}{2}\int \int |\varphi(z)-\varphi(w)|^2|\mathbb{K}(z,w)|^2d\mu(z)d\mu(w),
$$
where $\mathcal X(\varphi)=\sum_{z\in \mathcal X}\varphi(z)$.
\end{lemma}

Observe that Lemma \ref{lem} is a particular case of Lemma \ref{lem:fomula}, for $\varphi=\one_D$. The proof of Lemma \ref{lem:fomula} can be found in \cite{subhro}, Proposition 4.1. For completeness we give a proof Lemma \ref{lem:fomula}. 

\begin{proof}[Proof of Lemma \ref{lem:fomula}]
We have
$$
\E[\mathcal X(\varphi)]=\E(\sum_{z\in \mathcal X}\varphi(z))=\int \varphi(z)\rho_1(z)d\mu(z)=\int \varphi(z)\mathbb{K}(z,z)d\mu(z),
$$
where $\rho_1(z)$ is the one-point correlation function.
Now we have
\begin{align}\label{eqn:formula}
\mathbb{V}[\mathcal X(\varphi)]&=\E[(\mathcal X(\varphi)-\E[\mathcal X(\varphi)])(\bar{\mathcal X(\varphi)-\E[\mathcal X(\varphi)]})]\nonumber
\\&=\E\left(\sum_{z\in \mathcal X}\varphi(z)\sum_{w\in \mathcal X}\bar{\varphi(w)}\right)-\int \varphi(z)\mathbb{K}(z,z)d\mu(z)\int\bar{\varphi(w)}\mathbb{K}(w,w)d\mu(w).
\end{align}
Again we have
\begin{align}\label{eqn:reduction}
&\E\left(\sum_{z\in \mathcal X}\varphi(z)\sum_{w\in \mathcal X}\bar{\varphi(w)}\right)\nonumber
\\=&\E\left(\sum_z\varphi(z)\bar{\varphi(z)}\right)+\E\left(\sum_{z\neq w}\varphi(z)\bar{\varphi(w)}\right)\nonumber
\\=&\int\varphi(z)\bar{\varphi(z)}\rho_1(z)d\mu(z)+\int\int\varphi(z)\bar{\varphi(w)}\rho_2(z,w)d\mu(z)d\mu(w)\nonumber
\\=&\int\varphi(z)\bar{\varphi(z)}\mathbb{K}(z,z)d\mu(z)
+\int\int\varphi(z)\bar{\varphi(w)}\mathbb{K}(z,z)\mathbb{K}(w,w)d\mu(z)d\mu(w)
\\&-\int \int \varphi(z)\bar{\varphi(w)}|\mathbb{K}(z,w)|^2d\mu(z)d\mu(w).\nonumber
\end{align}
Therefore by (\ref{eqn:reduction}) from  (\ref{eqn:formula}) we get
\begin{eqnarray*}
&&\mathbb{V}[\mathcal X(\varphi)]
\\&=&\int \varphi(z)\bar{\varphi(z)}\mathbb{K}(z,z)d\mu(z)-\int\int \varphi(z)\bar{\varphi(w)}|\mathbb{K}(z,w)|^2d\mu(z)d\mu(w)
\\&=&\int\int \varphi(z)\bar{\varphi(z)}|\mathbb{K}(z,w)|^2d\mu(z)d\mu(w)-\int\int \varphi(z)\bar{\varphi(w)}|\mathbb{K}(z,w)|^2d\mu(z)d\mu(w)
\\&=&\frac{1}{2}\int \int [\varphi(z)\bar{\varphi(z)}-\varphi(z)\bar{\varphi(w)}-\varphi(w)\bar{\varphi(z)}+\varphi(w)\bar{\varphi(w)}]
|\mathbb{K}(z,w)|^2d\mu(z)d\mu(w)
\\&=&\frac{1}{2}\int \int |\varphi(z)-\varphi(w)|^2|\mathbb{K}(z,w)|^2d\mu(z)d\mu(w).
\end{eqnarray*}
Hence the result.

\end{proof}

Now we proceed to prove Theorem \ref{thm} and Theorem \ref{thm:result2}.

\begin{proof}[Proof of Theorem \ref{thm}]
\noindent{\bf Upper bound:} By Lemma \ref{lem}, we have
\begin{eqnarray*}
\mathbb{V}[\X_L(r\D)]&=&\int_{(r\D)}\int_{(r\D)^c}|\mathbb K_L(z,w)|^2d\mu_L(z)d\mu_L(w) \le\int_{D(0,r)}\mathbb K_L(z,z)d\mu_L(z)
\\&=&\frac{L}{\pi}\int_{D(0,r)}\frac{1}{(1-|z|^2)^2}dm(z)
=L\cdot\frac{r^2}{1-r^2}.
\end{eqnarray*}
Therefore we have, as $r\to 1^-$
\begin{eqnarray}\label{upperbound}
(1-r)\mathbb{V}[\X_L(r\D)]\le \frac{L}{2}.
\end{eqnarray}

\noindent{\bf Lower bound:}
From Lemma \ref{lem}, we have
\begin{eqnarray*}
\mathbb{V}[\X_L(r\D)]&=&\int_{(r\D)}\int_{(r\D)^c}|\mathbb{K}_L(z,w)|^2d\mu_L(z)d\mu_L(w)
\\&=&\frac{L^2}{\pi^2}\int_{(r\D)}\int_{(r\D)^c}\frac{(1-|z|^2)^{L-1}(1-|w|^2)^{L-1}}{|1-z\bar{w}|^{2(L+1)}}dm(z)dm(w).
\end{eqnarray*}
By writing $z=r_1e^{i\theta_1}$ and $w=r_2e^{i\theta_2}$, we have
\begin{align*}
&\mathbb{V}[\X_L(r\D)]\\&=\frac{L^2}{\pi^2}\int_{r_1=0}^{r}\int_{r_2=r}^{1}\int_{\theta_1=0}^{2\pi}\int_{\theta_2=0}^{2\pi}
\frac{(1-r_1^2)^{L-1}(1-r_2^2)^{L-1}}{(1+r_1^2r_2^2-2r_1r_2\cos(\theta_1-\theta_2))^{L+1}}
r_1r_2dr_1dr_2d\theta_1d\theta_2.
\end{align*}
Since the integrand depends on the difference $\theta_1-\theta_2$, hence
\begin{eqnarray}\label{formula}
\mathbb{V}[\X_L(r\D)]=\frac{2L^2}{\pi}\int_{r_1=0}^r\int_{r_2=r}^1\int_{0}^{2\pi}
\frac{(1-r_1^2)^{L-1}(1-r_2^2)^{L-1}}{(1+r_1^2r_2^2-2r_1r_2\cos\theta)^{L+1}}r_1r_2dr_1dr_2d\theta.
\end{eqnarray}
Let $a=r_1r_2$, then we have
\begin{align}\label{perticularcase}
\int_{0}^{2\pi}\frac{1}{(1+r_1^2r_2^2-2r_1r_2\cos\theta)^{L+1}}d\theta &=
\int_{0}^{2\pi}\frac{1}{((1-a)^2+2a(1-\cos\theta))^{L+1}}d\theta\nonumber
\\&=2\int_{0}^{\pi}\frac{1}{((1-a)^2+2a(1-\cos\theta))^{L+1}}d\theta\nonumber
\\&\ge2\int_{0}^{(1-a)}\frac{1}{((1-a)^2+2a(1-\cos\theta))^{L+1}}d\theta.
\end{align}
For $\theta\le 1-a$, we have
\begin{eqnarray*}
&&(1-a)^2+2a(1-\cos\theta)\le (1-a)^2+2a.\frac{\theta^2}{2}\le 2(1-a)^2
\\&\Rightarrow&\frac{1}{((1-a)^2+2a(1-\cos\theta))^{L+1}}\ge \frac{1}{2^{L+1}(1-a)^{2(L+1)}}.
\end{eqnarray*}
Therefore we have
\begin{eqnarray}\label{reduction}
\int_{0}^{(1-a)}\frac{1}{((1-a)^2+2a(1-\cos\theta))^{L+1}}d\theta
\ge \frac{1}{2^{L+1}(1-a)^{2L+1}}.
\end{eqnarray}
By (\ref{perticularcase}) and (\ref{reduction}), from (\ref{formula}) we get
\begin{eqnarray*}
\mathbb{V}[\X_L(r\D)]&\ge& \frac{4L^2}{\pi 2^{L+1}}\int_{r_1=0}^r\int_{r_2=r}^1\frac{(1-r_1^2)^{L-1}(1-r_2^2)^{L-1}}{(1-r_1r_2)^{2L+1}}r_1r_2dr_1dr_2
\\&\ge&\frac{4L^2}{\pi 2^{L+1}}\int_{r_1=0}^r\int_{r_2=r}^1\frac{(1-r_1^2)^{L-1}(1-r_2^2)^{L-1}}{(1-rr_1^2)^{2L+1}}r_1r_2dr_1dr_2,
\end{eqnarray*}
last inequality we have use the facts that $r_2\ge r$ and $r_1<1$. Therefore
\begin{eqnarray*}
\mathbb{V}[\X_L(r\D)]&=&\frac{L^2}{\pi 2^{L+1}}\int_{r_1=0}^{r^2}\int_{r_2=r^2}^1\frac{(1-r_1)^{L-1}(1-r_2)^{L-1}}{(1-rr_1)^{2L+1}}dr_1dr_2
\\&\ge&\frac{L^2(1-r)^{L-1}}{\pi 2^{L+1}}\int_{r_1=0}^{r^2}\frac{1}{(1-rr_1)^{2L+1}}dr_1\int_{r_2=r^2}^1(1-r_2)^{L-1}dr_2
\;\;(\mbox{as $r_1\le r^2$})
\\&=&\frac{L^2(1-r)^{L-1}}{\pi 2^{L+1}}.\frac{1}{2Lr}\{(1-r^3)^{-2L}-1\}.\frac{1}{L}(1-r^2)^L
\\&\ge&\frac{1}{r\pi2^{L+2}}\left\{\frac{1}{(1-r)^{2L}3^{2L}}-1\right\}(1-r)^{2L-1}.
\end{eqnarray*}
Therefore we have
\begin{eqnarray}\label{lowerbound}
(1-r)\mathbb{V}[\X_L(r\D)]\ge \frac{1}{\pi2^{L+2}3^{2L}} \;\;\mbox{as $r\to 1^-$}.
\end{eqnarray}
Finally, by (\ref{upperbound}) and (\ref{lowerbound}) we have
$$
c_2\le (1-r)\mathbb{V}[\X_L(r\D)]\le c_1\;\;\mbox{as $r\to 1^-$},
$$
where $c_1=\frac{L}{2}$ and $c_2=\frac{1}{2^{L+2}3^{2L}}$. Hence the result.
\end{proof}

%

Finally we prove Theorem \ref{thm:result2}.
\begin{proof}[Proof of Theorem \ref{thm:result2}]{\bf Upper bound:} By (\ref{eqn:reduction}) from (\ref{eqn:formula})
we have
\begin{eqnarray*}
\mathbb{V}[\mathcal X_L(\varphi_p)]
&\le&\int |\varphi_p(z)|^2\mathbb{K}_L(z,z)d\mu_L(z)
\\&=&\frac{L}{\pi}\int_{r\D} \left(1-\frac{|z|^2}{r^2}\right)^p\frac{1}{(1-|z|^2)^2}dm(z)
\\&=&\frac{L}{\pi} \int_{0}^r\int_0^{2\pi}\left(1-\frac{t^2}{r^2}\right)^p\frac{1}{(1-t^2)^2}tdt d\theta
\\&=&L\int_0^{r^2}\left(1-\frac{t}{r^2}\right)^p\frac{1}{(1-t)^2}dt
\\&=&\frac{L}{r^{2p}}\int_{1-r^2}^1(t-(1-r^2))^p\frac{1}{t^2}dt \;\;\;\;\mbox{ (by replacing $1-t$ by $t$)}
\\&\le&\frac{L}{r^{2p}}\int_{1-r^2}^{1}t^{p-2}dt
\\&=&\left\{\begin{array}{lcr}\frac{L}{r^{2p}(1-p)}\left(\frac{1}{(1-r^2)^{1-p}}-1\right)&if& p<1
\\-\frac{L}{r^{2p}}\log(1-r^2)&if & p=1\\\frac{L}{r^{2p}(p-1)}\left(1-\frac{1}{(1-r^2)^{p-1}}\right)&if& p>1\end{array}\right..
\end{eqnarray*}
Therefore as $r\to 1^-$ we have
\begin{eqnarray}\label{eqn:upper}
\begin{array}{ccr}
(1-r)^{1-p}\mathbb{V}[\mathcal X(\varphi_p)]\le \frac{L}{1-p}& \mbox{ when}& p<1,
\\-\frac{1}{\log{(1-r)}}\mathbb{V}[\mathcal X(\varphi_p)]\le L & \mbox{ when}& p=1,
\\\mathbb{V}[\mathcal X(\varphi_p)]\le \frac{L}{p-1}& \mbox{ when}& p>1.
\end{array}
\end{eqnarray}

\noindent{\bf Lower bound:} By Lemma \ref{lem:fomula} we have
\begin{eqnarray}\label{fomula2}
&&\mathbb{V}[\mathcal X_L(\varphi_p)]\nonumber
\\&=&\frac{1}{2}\int_{\mathbb D}\int_{\mathbb D}(\varphi_p(z)-\varphi_p(w))^2|
\mathbb{K}_L(z,w)|^2d\mu_L(z)d\mu_L(w)\nonumber
\\&=&\frac{1}{2}\frac{L^2}{\pi^2}\int_D\int_D\l(\l(1-\frac{|z|^2}{r^2}\r)^{p/2}-\l(1-\frac{|w|^2}{r^2}\r)^{p/2}\r)^2
|\mathbb{K}_L(z,w)|^2d\mu_L(z)d\mu_L(w)
\\&&+\frac{L^2}{\pi^2}\int_D\int_{D^c}\l(1-\frac{|z|^2}{r^2}\r)^p
|\mathbb{K}_L(z,w)|^2d\mu_L(z)d\mu_L(w)\nonumber
\end{eqnarray}
where $D^c=\mathbb D\backslash D$.
Therefore we have
\begin{eqnarray*}
&&\mathbb{V}[\mathcal X_L(\varphi_p)]
\\&\ge&\frac{1}{2}\frac{L^2}{\pi^2}\int_D\int_D\l(\l(1-\frac{|z|^2}{r^2}\r)^{p/2}
-\l(1-\frac{|w|^2}{r^2}\r)^{p/2}\r)^2|\mathbb{K}_L(z,w)|^2d\mu_L(z)d\mu_L(w)
\\&\ge &\frac{1}{2^{L+1}}\frac{L^2}{\pi}\int_0^{r^2}\int_0^{r^2}\l\{\l(1-\frac{s}{r^2}\r)^{p/2}-\l(1-\frac{t}{r^2}\r)^{p/2}\r\}^2
\frac{(1-s)^{L-1}(1-t)^{L-1}}{(1-st)^{2L+1}}dsdt\;\;\mbox{(by \eqref{reduction})}
\\&\ge &\frac{1}{2^{L+1}}\frac{L^2}{\pi}
\int_0^{\frac{r^2}{2}}\int_{\frac{2r^2}{3}}^{r^2}\l\{\l(1-\frac{s}{r^2}\r)^{p/2}-\l(1-\frac{t}{r^2}\r)^{p/2}\r\}^2
(1-s)^{L-1}(1-t)^{L-1}dsdt
\\&\ge&\frac{1}{2^{L+1}}\frac{L^2}{\pi}\l(\frac{1}{2^{p/2}}-\frac{1}{3^{p/2}}\r)^2
\int_0^{\frac{r^2}{2}}\int_{\frac{2r^2}{3}}^{r^2}(1-s)^{L-1}(1-t)^{L-1}dsdt
\\&\ge&\frac{1}{2^{L+1}}\frac{L^2}{\pi}\l(\frac{1}{2^{p/2}}-\frac{1}{3^{p/2}}\r)^2
\int_0^{\frac{1}{4}}\int_{\frac{2}{3}}^{\frac{3}{4}}(1-s)^{L-1}(1-t)^{L-1}dsdt,
\end{eqnarray*}
as $r\to 1^-$. Therefore we have
\begin{eqnarray}\label{eqn:lower1}
\mathbb{V}[\mathcal X_L(\varphi_p)]\ge C
\end{eqnarray}
where $C$ is some constant depending on $L,p$. Therefore by (\ref{eqn:upper}) and (\ref{eqn:lower1}) we have
$$
\mathbb{V}[\mathcal X_L(\varphi_p)]=\Theta(1),\;\;\;\;\mbox{as $r\to 1^-$ for $p>1$.}
$$

Now assume that $p<1$. By (\ref{fomula2}) we have
\begin{eqnarray*}
&&\mathbb{V}[\mathcal X_L(\varphi_p)]
\\&\ge &\frac{L^2}{\pi^2}\int_D\int_{D^c}\l(1-\frac{|z|^2}{r^2}\r)^p
|\mathbb{K}(z,w)|^2d\mu(z)d\mu(w)
\\&\ge& \frac{L^2}{\pi 2^{L+1}}\int_0^{r^2}\int_{r^2}^1\l(1-\frac{s}{r^2}\r)^p
\frac{(1-s)^{L-1}(1-t)^{L-1}}{(1-st)^{2L+1}}dsdt \;\;\;\mbox{(by \eqref{reduction})}
\\&\ge& \frac{L^2}{\pi 2^{L+1}}(1-r^2)^{L-1}\int_0^{r^2}\int_{r^2}^1\l(1-\frac{s}{r^2}\r)^p
\frac{(1-t)^{L-1}}{(1-sr^2)^{2L+1}}dsdt
\\&\ge& \frac{L^2}{\pi 2^{L+1}L}(1-r^2)^{L-1}(1-r^2)^L\int_0^{r^2}\l(1-\frac{s}{r^2}\r)^p
\frac{1}{(1-sr^2)^{2L+1}}ds
\\&=&\frac{L}{\pi 2^{L+1}}\frac{(1-r^2)^{2L-1}}{r^{4p+1}}\int_{1-r^4}^{1}\frac{(u-(1-r^4))^p}{u^{2L+1}}du
\mbox{ (putting, $1-sr^2=u$)}
\end{eqnarray*}
Choose $r<1$ such that $r^4>\frac{1}{2}$ which implies that $2(1-r^4)<1$. Hence we have
\begin{eqnarray*}
\mathbb{V}[\mathcal X_L(\varphi_p)]&\ge& \frac{L}{\pi 2^{L+1}}(1-r^2)^{2L-1}\int_{2(1-r^4)}^{1}\frac{(u-(1-r^4))^p}{u^{2L+1}}du
\\&\ge& \frac{L}{\pi 2^{L+1}}(1-r^2)^{2L-1}(1-r^4)^p\int_{2(1-r^4)}^{1}\frac{1}{u^{2L+1}}du
\\&\ge&\frac{L}{\pi 2^{L+1}}(1-r)^{2L-1}(1-r)^p\frac{1}{2L}[(2(1-r^4))^{-2L}-1]
\\&\ge&\frac{1}{\pi 2^{L+2}}(1-r)^{2L-1+p}[(2(1-r))^{-2L}-1]
\end{eqnarray*}
Therefore we have
\begin{eqnarray}\label{lower2}
(1-r)^{1-p}\mathbb{V}[\mathcal X_L(\varphi_p)]\ge C_1
\end{eqnarray}
as $r\to 1^-$, where $C_1$ is a positive constant which depends on $L$. Hence by (\ref{eqn:upper}) and (\ref{lower2}) we have
$$
\mathbb{V}[\mathcal X_L(\varphi_p)]=\Theta\l(\frac{1}{(1-r)^{1-p}}\r)
$$
as $r\to 1^-$, for $p<1$.

Finally assume that $p=1$. From (\ref{fomula2}) we have
\begin{align*}
\mathbb{V}[\mathcal X_L(\varphi_p)]
&\ge\frac{1}{2}\frac{L^2}{\pi^2}\int_D\int_D\l(\l(1-\frac{|z|^2}{r^2}\r)^{1/2}
-\l(1-\frac{|w|^2}{r^2}\r)^{1/2}\r)^2|\mathbb{K}(z,w)|^2d\mu(z)d\mu(w)
\\&\ge \frac{1}{2^{L+1}}\frac{L^2}{\pi^2}\int_0^{r^2}\int_0^{r^2}\l\{\sqrt{\l(1-\frac{s}{r^2}\r)}-\sqrt{\l(1-\frac{t}{r^2}\r)}\r\}^2
\frac{(1-s)^{L-1}(1-t)^{L-1}}{(1-st)^{2L+1}}dsdt,
\end{align*}
 last inequality follows from (\ref{reduction}). Therefore 
\begin{align*}
\mathbb{V}[\mathcal X_L(\varphi_p)]&\ge \frac{1}{2^{L+1}}\frac{L^2}{\pi^2}\int_{\delta}^{r^2}\int_{s^3}^{s^2/r^2}\l\{\sqrt{\l(1-\frac{s}{r^2}\r)}-\sqrt{\l(1-\frac{t}{r^2}\r)}\r\}^2
\frac{(1-s)^{L-1}(1-t)^{L-1}}{(1-st)^{2L+1}}dsdt,
\end{align*}
where $\delta >0$ and fixed. Since $s^3\le t\le \frac{s^2}{r^2}<s$, therefore
\begin{align*}
\mathbb{V}[\mathcal X_L(\varphi_p)]
&\ge \frac{1}{2^{L+1}}\frac{L^2}{\pi^2}\int_{\delta}^{r^2}\int_{s^3}^{s^2/r^2}\l\{\sqrt{\l(1+\frac{s}{r^2}\r)}-1\r\}^2\l(1-\frac{s}{r^2}\r)
\frac{(1-s)^{L-1}(1-s^2/r^2)^{L-1}}{(1-s^4)^{2L+1}}ds
\\&\ge \frac{1}{2^{L+1}4^{2L+1}}\frac{L^2}{\pi^2}\l\{\sqrt{\l(1+\frac{\delta}{r^2}\r)}-1\r\}^2\int_{\delta}^{r^2}\l(1-\frac{s}{r^2}\r)
\frac{s^2(1-r^2s)}{(1-s)^{3}}ds.
\\&\ge C.\int_{\delta}^{r^2}\l(1-\frac{s}{r^2}\r)\frac{1}{(1-s)^{2}}ds \;\;\;\mbox{(as $1-r^2s>1-s$)}
\\&\ge C.\int_{1-r^2}^{1-\delta}(u-(1-r^2))\frac{1}{u^{2}}du
\\&=C.[\log(1-\delta)-\log(1-r^2)+\frac{1-r^2}{1-\delta}-1]
\end{align*}
where $C$ is constant depending on $L$ and $\delta>0$ (fixed). Therefore for $p=1$, we have
\begin{eqnarray}\label{lower3}
\frac{\mathbb{V}[\mathcal X_L(\varphi_p)]}{-\log(1-r)}\ge C\;\;\mbox{ as $r\to 1^-$.}
\end{eqnarray}
Hence by (\ref{eqn:upper}) and (\ref{lower3}) we have
$$
\mathbb{V}[\mathcal X_L(\varphi_1)]=\Theta(-\log(1-r))\;\;\mbox{ as $r\to 1^-$.}
$$
Hence the result.
\end{proof}

\backmatter
\appendix
\bibliographystyle{amsalpha} \bibliography{PhD_thesis}

\end{document}